\documentclass[10pt,reqno,final]{amsart}
\usepackage{epsfig,amssymb,amsmath,version}
\usepackage{amssymb,version,graphicx,fancybox,mathrsfs,multirow}

\usepackage{url,hyperref}
\usepackage[notcite,notref]{showkeys}

\usepackage{subfigure}
\usepackage{color,xcolor}
\usepackage{cases}
\usepackage{mathtools}

\catcode`\@=11 \theoremstyle{plain}
\@addtoreset{equation}{section}   % Makes \section reset 'equation' counter.
\renewcommand{\theequation}{\arabic{section}.\arabic{equation}}
\@addtoreset{figure}{section}
\renewcommand\thefigure{\thesection.\@arabic\c@figure}
\renewcommand{\thefigure}{\arabic{section}.\arabic{figure}}
\newtheorem{thm}{\bf Theorem}

\newtheorem{cor}{\bf Corollary}
\newtheorem{prop}{Proposition}[section]

\newenvironment{corollary}{\begin{cor}} {\end{cor}}
\newtheorem{lmm}{\bf Lemma}

\newenvironment{lemma}{\begin{lmm}}{\end{lmm}}
\theoremstyle{remark}
\newtheorem{rem}{\bf Remark}[section]
\theoremstyle{definition}
\newtheorem{defn}{\bf Definition}[section]

\numberwithin{table}{section}

\newcommand{\comm}[1]{\marginpar{%
\vskip-\baselineskip %raise the marginpar a bit
\raggedright\footnotesize
\itshape\hrule\smallskip#1\par\smallskip\hrule}}

\def \ri {{\rm i}}

\def \rho {\mu}

\newcommand{\bs}[1]{\boldsymbol{#1}}

\newcommand{\D}[3]{{{}_{#1}}{\rm D}_{#2}^{#3}}
\newcommand{\I}[3]{{}_{#1}{\rm I}_{#2}^{#3}}
\newcommand{\CD}[3]{{}_{#1}^{C}{\rm D}_{#2}^{#3}}
\newcommand{\TD}[3]{{{}_{#1}}{\rm D}_{#2}^{#3,\lambda}}
\newcommand{\TI}[3]{{}_{#1}{\rm I}_{#2}^{#3,\lambda}}{}
\newcommand{\GLa}[3]{{\mathcal L}^{(#1,\lambda)}_{#2}{#3}}
\newcommand{\GLb}[3]{\breve {\mathcal{L}}^{(#1,\lambda)}_{#2}{#3}}

\begin{document}
\bibliographystyle{plain}
\graphicspath{{./figs/}}
\baselineskip 13pt

{\title[] {Laguerre Functions and Their Applications to tempered Fractional Differential Equations on Infinite Intervals}
\author{Sheng Chen${}^1$,\;\;  Jie Shen${}^{1,2}$\;\;  and\;\;   Li-Lian Wang${}^3$}

\thanks{${}^1$Fujian Provincial Key Laboratory on Mathematical Modeling \& High Performance Scientific Computing and School of Mathematical Sciences, Xiamen University, Xiamen, Fujian 361005, P. R. China. This work is supported in part by NSFC grants  11371298, 11421110001, 91630204 and 51661135011. }
\thanks{${}^2$Department of Mathematics, Purdue University, West Lafayette,
  IN 47907-1957, USA. J.S. is  partially  supported by NSF grant
        DMS-1620262 and AFOSR grant FA9550-16-1-0102.}
 \thanks{${}^3$Division of Mathematical Sciences, School of Physical
and Mathematical Sciences, Nanyang Technological University,
637371, Singapore. L.W. is partially supported by Singapore MOE AcRF Tier 1 Grant (RG 15/12), and Singapore MOE AcRF Tier 2 Grant (MOE 2013-T2-1-095, ARC 44/13).}
\begin{abstract}
Tempered fractional derivatives originated from the tempered fractional diffusion equations (\textbf{TFDEs}) modeled on the whole space $\mathbb{R}$ (see \cite{temperedsabzikar2014}).   For numerically solving TFDEs, two kinds of generalized Laguerre functions were defined and some important properties were proposed to establish the approximate theory. The related prototype tempered fractional differential problems was proposed and solved as the  guidance. TFDEs are numerically solved by two domains Laguerre spectral method and the numerical experiments show some properties of the TFDEs  and verify the efficiency of the spectral scheme.
\end{abstract}

 \maketitle
\section{Introduction}
%A. Einstein \cite{einstein1905movement,einstein1906theory} first applied

The  normal diffusion equation $\partial_t p(x,t)=\partial_x^2p(x,t)$  can be derived from the Brownian motion which describes the particle's random walks. Over the last few decades, a large body of literature has demonstrated that anomalous diffusion, in which the mean square variance grows faster (super-diffusion) or slower (sub-diffusion) than in a Gaussian process, offers a superior fit to experimental data observed in many important practical applications,
e.g., in physical science \cite{magdziarz2007fractional, metzler2000random, metzler2004restaurant, piryatinska2005models},  finance \cite{gorenflo2001fractional, meerschaert2006coupled, scalas2006five},  biology \cite{jeon2012anomalous, baeumer2007fractional} and hydrology \cite{baeumer2001subordinated, cushman2000fractional,deng2006parameter}.
  The  anomalous diffusion equation takes the form
  \begin{equation}
  \partial^\nu_t p(x,t)=\partial_x^\mu p(x,t),
  \end{equation}
  where  $0<\nu\le 1$ and $0<\mu< 2$ (cf. \cite{metzler2000random} for a review on this subject), whose solution exhibits heavy tails, i.e., power law decays at infinity. %corresponds to heavy tailed power law particle jumps $x^{-\mu}$.
 In order to "temper" the power law decay, the authors of \cite{temperedsabzikar2014} applied an exponential factor $e^{-\lambda |x|}$ to the particle jump density, and showed that the Fourier transform of the tempered probability density function $p(x,t)$ takes the form
$${\mathscr F}[p](\omega,t)=
e^{-[pA_+^{\mu,\lambda}(\omega)+qA_-^{\mu,\lambda}(\omega)]Dt},\quad 0<\mu<2,$$
 where $0\leq p\leq1,~q=1-p$, $D$ is a  constant  and
\begin{equation*}
A_{\pm}^{\mu,\lambda}(\omega):=
\begin{cases}
(\lambda\pm i\omega)^\mu-\lambda^\mu,&0<\mu<1,\\
(\lambda\pm i\omega)^\mu-\lambda^\mu-\pm i\omega \mu\lambda^{\mu-1},&1<\mu<2.
\end{cases}
 \end{equation*}
Moreover, they defined  tempered fractional  derivative operators $\partial_{\pm,x}^{\mu,\lambda}$ through  Fourier transform: $\mathscr{F}{[\partial_{\pm,x}^{\mu,\lambda} u]}(\omega)=A_{\pm}^{\mu,\lambda}(\omega)\mathscr{F}[u](\omega)$, and derived the tempered fractional diffusion equation (\textbf{TFDE}):
\begin{equation}\label{temperFDEoriginal}
\partial_t u(x,t)=(-1)^k C_T \{ p\partial_{+,x}^{\mu,\lambda}+q\partial_{-,x}^{\mu,\lambda} \}u(x,t), \quad \mu\in(k-1,k),~k=1,2.
\end{equation}
It has been argued  that tempered anomalous diffusion models have  advantages over the normal dissuasion models in some applications of geophysics \cite{meerschaert2008tempered,zhang2011gaussian} and finance \cite{carr2002fine}.
%where $\lceil\mu\rceil$ is the minimum integer not less than $\mu$. %Without any generality, hereafter we assume constant $C=1$.

\begin{comment}
Note that the fractional operators $\partial_{\pm,x}^{\mu,\lambda} u$ was defined by Fourier transform and without any explicit formula. Sabzikar etc bring in alternative form of the  tempered fractional integrals and derivatives (\textbf{TFI/Ds}) (we respectively rewrite in Definition \ref{temperedInt} and Definition \ref{temperedDeri}) which derived by applying the inverse Fourier transform,
 $$\TD{}{\pm}{\mu}u(x)= \mathscr{F}^{-1}\big[(\lambda\pm i\omega)^\mu \mathscr{F}{[ u]}(\omega)\big](x).$$ Then, the original fractional operator, defined by Fourier transform, can be rewritten as the following equivalent explicit form,
\begin{equation}\label{temprelamu}
\partial_{\pm,x}^{\mu,\lambda}u=\begin{cases}
\TD{}{\pm}{\mu}u-\lambda^\mu u,& 0<\mu<1,\\
\TD{}{\pm}{\mu}u-\pm\mu\lambda^{\mu-1}\partial_x u-\lambda^\mu u,& 1<\mu<2.
\end{cases}
\end{equation}
\end{comment}

It is a challenging task to numerically solve the tempered fractional diffusion equation \eqref{temperFDEoriginal}, due particularly to (i) the non-local nature of tempered fractional  derivatives; and (ii) the unboundedness of the domain. In  \cite{temperedsabzikar2014}, the authors used a finite-difference approach on a truncated domain. In \cite{ZAK15}, the authors considered tempered derivatives on a finite interval and derived an efficient Petrov-Galerkin method for solving tempered fractional ODEs by using the eigenfunctions of tempered fractional Sturm-Liouville problems. In \cite{HZZ16}, the authors used Laguerre functions to approximate the substantial  fractional ODEs, which are similar to those we consider in Section 3, on the half line.
In order to avoid the difficulty of assigning boundary conditions at the truncated boundary, we shall deal with the unbounded domain directly in this paper.

 Since the tempered factional diffusion equation
is derived from the random walk on the whole line, one is tempted to use Hermite polynomials/functions which are suitable for many problems on the whole line \cite{ShenTangWang2011}.  Unfortunately,
  due to the exponential factor  $e^{\lambda |x|}$ in the  tempered fractional derivatives,  Hermite polynomials/functions are not suitable basis functions. Instead,
 as we will show in Section 3, properly defined generalized Laguerre functions  (\textbf{GLFs})  enjoy particularly simple form under the action of  tempered fractional derivatives, just as the relations between  generalized Jacobi functions and usual fractional derivatives \cite{CSW16}. Hence, the main goal  of this paper is to design  efficient spectral methods using \textbf{GLFs}  to solve the tempered fractional diffusion equation  \eqref{temperFDEoriginal} in various situations.
However, Laguerre polynomials/functions are mutually orthogonal on the half line,
 how do we use them to deal with \eqref{temperFDEoriginal} on the whole line? We shall first consider special cases of  \eqref{temperFDEoriginal} with $p=1, \, q=0$ or  $p=0, \, q=1$. In these cases, we can reduce
  \eqref{temperFDEoriginal} to the half line, and the \textbf{GLFs} can be naturally used. For the general case,
  we shall employ a two-domain spectral-element method, and use \textbf{GLFs} as basis functions  on each subdomain.

The rest of the  paper is organized as follows.  In the next section,  we  present the definition of the tempered fractional derivatives, and recall some useful properties of Laguerre polynomials. In Sections \ref{SectionLaguerre}, we define two classes of generalized Laguerre functions, study their approximation properties, and apply them for solving  simple one sided tempered fractional equations. In Section \ref{SectionApp2}, we develop a spectral-Galerkin method for solving a tempered fractional diffusion equation   on the half line. Finally, we present a spectral-Galerkin method for solving the tempered fractional diffusion equation  on the whole line in Section \ref{SectionTFDE}. Some concluding remarks are given in the last section.

\section{Preliminaries}\label{preliminary}
\setcounter{equation}{0}
\setcounter{lmm}{0}
\setcounter{thm}{0}
%In this section, the tempered fractional integrals and derivatives (proposed in \cite{M2013tempered},\cite{temperedsabzikar2014} recently) are investigated. Some crucial relations with Laguerre polynomial/function will be proved.
%The hypergeometric function be introduced to find the representation of the Laguerre functions,
%some relevant properties of the Laguerre polynomials with real parameters will be reviewed. Moreover, we define the generalized Laguerre functions in the final subsections and show some critical identities to tempered fractional derivative.
%
%Before stepping on the field of the tempered fractional calculus, we introduce some notations will be used throughout this paper.
%
%In order to easily understand the tempered fractional derivative, we suggest reviewing the definitions and properties of the traditional fractional integral and derivative in \cite{samko1993fractional},\cite{diethelm2010analysis} and \cite{podlubny1999fractional}, or the list in the Appendix \ref{Appenix:fractional calculus}.

Let $\mathbb N$ and $\mathbb R$ be respectively the sets of positive integers and real numbers. We further denote
 \begin{equation}\label{mathNR}
{\mathbb N}_0:=\{0\}\cup {\mathbb N}, \;\; {\mathbb R}^+:= \{x\in {\mathbb R}: x> 0\},\;\; {\mathbb R}^-:= \{x\in {\mathbb R}: x< 0\},\;\; %\mathbb{R}^+_0:=\{0\}\cup {\mathbb R}^+,\quad
 \mathbb{R}_0^\pm:=\mathbb{R}^\pm\cup\{0\}.
\end{equation}

\subsection{Usual (non-tempered) fractional integrals and derivatives}\label{Appenix:fractional calculus}
 Recall the definitions of the fractional integrals and fractional derivatives in the sense of Riemann-Liouville (see e.g., \cite{podlubny1999fractional}).
\begin{defn}[{\bf Riemann-Liouville fractional integrals and derivatives}]\label{RLFDdefn}
 {\em For $a,b\in {\mathbb R} $ or $a=-\infty, b=\infty,$ and  $\rho\in {\mathbb R}^+,$  the left and right fractional integrals are respectively defined  as
 \begin{equation}\label{leftintRL}
    \I{a}{x}{\rho} u(x)=\frac 1 {\Gamma(\rho)}\int_{a}^x \frac{u(y)}{(x-y)^{1-\rho}} {\rm d}y,\;\;   \I{x}{b}{\rho} u(x)=\frac {1}  {\Gamma(\rho)}\int_{x}^b \frac{u(y)}{(y-x)^{1-\rho}} {\rm d}y,\;\;\; x\in \Lambda:=(a,b).
\end{equation}
For real $s\in [k-1, k)$ with $k\in {\mathbb N},$  the left-sided  Riemann-Liouville fractional derivative {\rm(LRLFD)} of order $s$ is defined by
\begin{equation}\label{leftRLdefn}
   \D{a}{x}{s} u(x)=\frac 1{\Gamma(k-s)}\frac{d^k}{dx^k}\int_{a}^x \frac{u(y)}{(x-y)^{s-k+1}} {\rm d}y,\;\;\;  x\in  \Lambda,
\end{equation}
and the right-sided Riemann-Liouville fractional derivative  {\rm(RRLFD)} of order $s$ is defined  by
\begin{equation}\label{rightRLdefn}
    \D{x}{b}{s}u(x)=\frac {(-1)^k}{\Gamma(k-s)}\frac{d^k}{dx^k}\int_{x}^b\frac{u(y)}{(y-x)^{s-k+1}} {\rm d}y, \;\;\;  x\in  \Lambda.
\end{equation}}
\end{defn}

 From the above definitions, it is clear that  for any   $k\in {\mathbb N}_0,$
\begin{equation}\label{2rela}
 \D{x}{b}{k}=(-1)^{k} \D{}{}{k},\quad  \D{a}{x}{k}=\D{}{}{k},\;\;\; {\rm where}\;\;\;  \D{}{}{k}:=\frac{d^k}{dx^k}.
\end{equation}
Therefore,  we can express the RLFD as
\begin{equation}\label{ImportRela}
    \D{a}{x}{s}  u(x)=\D{}{}{k}\big\{ \I{a}{x} {k-s} u(x)\big\}; \quad \;   \D{x}{b}{s}  u(x)=(-1)^k \D{}{}{k} \big\{ \I{x}{b} {k-s} u(x)\big\}.
\end{equation}
According to \cite[Thm. 2.14]{diethelm2010analysis},  we have that for any finite $a$ and any $v\in L^1(\Lambda),$ and real $s\ge 0,$
\begin{equation}\label{rulesa}
   \D{a}{x} s\, \I{a}{x} s f(x) = f(x),\quad \text{a.e. in} \;\; \Lambda.
 \end{equation}

Note that by commuting the integral and  derivative operators in  \eqref{ImportRela}, we  define the Caputo fractional derivatives:
\begin{equation}\label{Capfderiv}
\CD{a}{x}{s}   u(x)= \I{a}{x} {k-s} \big\{\D{}{}{k} u(x)\big\}; \quad \;
\CD{x}{b}{s}   u(x)= (-1)^k\I{x}{b} {k-s} \big\{\D{}{}{k} u(x)\big\}. %;   \D{x}{b}{s}  u(x)=(-1)^k \D{}{}{k} \big\{ \D{x}{b} {k-s} u(x)\big\}.
\end{equation}

For an affine transform $x=\lambda t,~\lambda>0$, on account of
 \begin{equation*}
 \begin{aligned}
   \I{a}{t} \rho v(\lambda t)
   &=\frac {1} {\Gamma(\rho)}\int_{a}^{ t} \frac{v(\lambda s)}{(t-s)^{1-\rho}} d s
   =\frac {\lambda^{-\rho}} {\Gamma(\rho)}\int_{a}^{ t} \frac{v(\lambda s)}{(\lambda t-\lambda s)^{1-\rho}} \lambda d s
   \\&
   =\frac {\lambda^{-\rho}} {\Gamma(\rho)}\int_{\lambda a}^{ x} \frac{v(y)}{(x-y)^{1-\rho}}dy=\lambda^{-\rho} \I{\lambda a}{x}\rho v(x),
   \end{aligned}
\end{equation*}
and $\frac{\rm d}{{\rm d}t}=\lambda \frac{\rm d}{{\rm d}x},$ we derive from Definition \ref{RLFDdefn} that
%\begin{equation*}
%\frac{d}{dt}=\frac{dx}{dt}\frac{d}{dx}=\lambda \frac{d}{dx}.
%\end{equation*}
%Then, due to the definition \ref{RLFDdefn}, it's easy to derive the below concise identities
\begin{equation}\label{transform xt}
 \I{a}{t} \rho v(\lambda t)= \lambda^{-\rho} \I{\lambda a}{x}\rho v(x),\quad \D{a}{t} s v(\lambda t)= \lambda^{s} \D{\lambda a}{x} s v(x),\quad s,\rho,\lambda>0.
\end{equation}
Similarly, we have the following identities for the right fractional derivative:
\begin{equation}\label{transform2 xt}
 \I{t}{b} \rho v(\lambda t)= \lambda^{-\rho} \I{t}{\lambda b}\rho v(x),\quad \D{t}{b} s v(\lambda t)= \lambda^{s} \D{x}{\lambda b} s v(x),\quad s,\rho,\lambda>0.
\end{equation}

\subsection{Tempered fractional integrals and derivatives on ${\mathbb R}$}
Recently,   Sabzikar et al.  \cite[(19)-(23)]{temperedsabzikar2014}  introduced
the  tempered fractional integrals and  derivatives on the whole line.
%  Below, we recall their definitions and some of their basic properties. For standing out the meaning of the notations, here we
% make  a little modification  (see Remark \ref{RemNotations}).
%with the explicit expression, where power laws are tempered by an exponential factor, named by tempered fractional calculus. In order to further investigation, we first introduce  the definition  of the tempered fractional integral on the real line $\mathbb{R}$.
\begin{defn}[{\bf Tempered fractional integrals}]\label{temperedInt}
{\em For  $\lambda\in {\mathbb R}^+_0$,  the left tempered fractional integral of a suitable function $u(x)$ of
order $\mu\in {\mathbb R}^+$ is defined by
\begin{equation}\label{temperedInt+}
\TI{-\infty}{x}{\mu}  u(x)=\frac{1}{\Gamma(\mu)}\int_{-\infty}^x \frac{e^{-\lambda(x-y)}}{(x-y)^{1-\mu}}\,u(y)\,{\rm d}y,\quad x\in \mathbb{R},
\end{equation}
and the right tempered fractional integral of order $\mu\in {\mathbb R}^+$ is defined by
\begin{equation}\label{temperedInt-}
\TI{x}{\infty}{\mu} u(x)=\frac{1}{\Gamma(\rho)}\int^{\infty}_x \frac{e^{-\lambda(y-x)}}{(y-x)^{1-\rho}}\,u(y)\,{\rm d}y,\quad x\in \mathbb{R}.
\end{equation}
}
\end{defn}

It is evident that by \eqref{leftintRL} and \eqref{temperedInt+}-\eqref{temperedInt-},   we have
\begin{equation}\label{relationA00}
\I{-\infty}{x}{\mu}={}_{-\infty}{\rm I}_{x}^{\mu,0}, \quad \I{x}{\infty}{\mu}={}_{x}{\rm I}_{\infty}^{\mu,0},
\end{equation}
and
\begin{equation}\label{relationA}
\TI{-\infty}{x}{\mu}  u(x)=e^{-\lambda x}\I{-\infty}{x}{\mu}\big\{e^{\lambda x} u(x)\big\},\quad \TI{x}{\infty}{\mu}  u(x)=e^{\lambda x}\I{x}{\infty}{\mu}\big\{e^{-\lambda x} u(x)\big\}.
\end{equation}
%where $\Gamma(\rho)$ is the Gamma function. Obviously, if $\lambda=0,$ they reduce to the usual left Riemann-Liouville fractional integrals $\I{-\infty}{x}{\mu}u$ and right Riemann-Liouville fractional integrals $\I{x}{\infty}{\mu}u$, respectively.

As shown in \cite{temperedsabzikar2014}, the tempered fractional derivative  can be characterized by its Fourier transform.   %For this purpos, we briefly review some basics of continuous Fourier transformations.
%Let $L^2({\mathbb R})$ be the space of square integrable functions on the whole  line.
Recall that, for any $u\in L^2({\mathbb R}),$   its Fourier transform and inverse Fourier transform  are defined by
\begin{equation}\label{Fourierfinite}
{\mathscr F}[u](\omega)=\int_{-\infty}^\infty u(x)e^{-\ri \omega x}\,{\rm d}x; \quad  u(x)={\mathscr F}^{-1}\big[{\mathscr F}[u](\omega)\big](x)=\frac 1 {2\pi}\int_{-\infty}^\infty {\mathscr F}[u](\omega)e^{\ri \omega x}\,{\rm d}\omega.
\end{equation}
There holds the well-known Parseval's identity:
\begin{equation}\label{parseval}
\int_{-\infty}^\infty u(x)\,\bar v(x)\, {\rm d}x=\frac 1 {2\pi}\int_{-\infty}^\infty {\mathscr F}[u](\omega)\,\overline{{\mathscr F}[v]}(\omega)\,{\rm d}\omega,
\end{equation}
where  $\bar v$ is   the complex conjugate of $v$.
Let $H(x)$ be the   Heaviside  function, i.e., $H(x)=1$ for $x\ge 0,$ and vanishing for all $x< 0.$  Then we can reformulate the left tempered fractional integral as %in a convolution form
\begin{equation}\label{conformA}
\begin{split}
\TI{-\infty}{x}{\mu} u(x)&=\frac{1}{\Gamma(\rho)}\int_0^\infty  y^{\rho-1} e^{-\lambda y} u(x-y)\,{\rm d }y
%\\&
 =\frac{1}{\Gamma(\rho)}\int_{-\infty}^\infty  y^{\rho-1} e^{-\lambda y} H(y)\, u(x-y)\,{\rm d }y
\\& = \big({ K}\ast u\big)(x),\quad {\rm where}\;\;\;  {K}(x):=x^{\rho-1}e^{-\lambda x}H(x)/{\Gamma(\rho)}.
\end{split}
\end{equation}
 Note that ${K}(x)$  is related to the particle jump density
 (cf.  \cite[(8)]{temperedsabzikar2014}).
Using the formula: $ {\mathscr F}[K](\omega)= {(\lambda+\ri \omega)^{-\rho}},$ and the convolution property of Fourier transform (see, e.g., \cite{samko1993fractional,stein1971introduction}), we derive
\begin{equation}\label{Fouriertran+}
{\mathscr F}[\TI{-\infty}{x}{\mu} u](\omega)={\mathscr F}[K\ast u](\omega)=
{\mathscr F}[K](\omega)\, {\mathscr F}[u](\omega)= {(\lambda+\ri \omega)^{-\rho}} {\mathscr F}[u](\omega).
\end{equation}
Similarly,  the Fourier transform of the right tempered fractional integral is
\begin{equation}\label{Fouriertran-}
{\mathscr F}[\TI{x}{\infty}{\mu} u](\omega)= {(\lambda-\ri \omega)^{-\rho}} {\mathscr F}[u](\omega).
\end{equation}
%In particular,  for $\lambda=0$, we have
%\begin{equation}\label{FourierAB}
% {\mathscr F}\big[\D{-\infty}{x}{\mu} u\big](\omega)=( \ri \omega)^{\rho} {\mathscr F}[u](\omega),\quad {\mathscr F}\big[\D{x}{\infty}{\mu} u\big](\omega)=(- \ri \omega)^{\rho} {\mathscr F}[u](\omega).%{\mathscr F}[\I{}{\pm}{\mu} u](\omega)= {(\pm \ri \omega)^{-\rho}} {\mathscr F}[u](\omega).
%\end{equation}

In view of \eqref{Fouriertran+}-\eqref{Fouriertran-},  Sabzikar et al. \cite{temperedsabzikar2014} then introduced the left and right  tempered fractional derivatives  as follows. % from the Fourier transform.
\begin{defn}[{\bf Tempered fractional derivatives}]\label{temperedDeri}
{\em For  $ \lambda\in {\mathbb R}^+_0$, the left and right  tempered fractional derivatives  of order $\mu\in {\mathbb R}^+$ of a suitable function $u(x),$  are defined by
\begin{equation}\label{Temdef2FourierB}
 {\mathscr F}\big[\TD{-\infty}{x}{\mu} u\big](\omega)=(\lambda+ \ri \omega)^{\rho} {\mathscr F}[u](\omega),\quad {\mathscr F}\big[\TD{x}{\infty}{\mu} u\big](\omega)=(\lambda- \ri \omega)^{\rho} {\mathscr F}[u](\omega),
 \end{equation}
that is, for any $x\in {\mathbb R}$,
\begin{equation}\label{Temdef2Fourier}
\TD{-\infty}{x}{\mu} u(x)={\mathscr F^{-1}}\big[{(\lambda+\ri \omega)^{\rho}} {\mathscr F}[u](\omega)\big](x),\quad \TD{x}{\infty}{\mu} u(x)={\mathscr F^{-1}}\big[{(\lambda-\ri \omega)^{\rho}} {\mathscr F}[u](\omega)\big](x).
 \end{equation}
 %that is,
 %where ${\mathscr F}^{-1}$ denotes the operator of inverse Fourier transform defined in .
 }
 \end{defn}
% Since $\TI{-\infty}{x}{\mu} u=u\star v$, where $v(x)=x^{\rho-1}e^{-\lambda x}\mathbf{1}_{[0,\infty)}/{\Gamma(\rho)}$, and  the Fourier transform $\mathcal{F} (v)=(\lambda+i\omega)^{-\rho}$, then by the property of the convolution,  $\mathcal{F} ({\TI{-\infty}{x}{\mu} u)}{}=(\lambda+i\omega)^{-\rho}\mathcal{F} (u)$. Similarly, we can derive the Fourier transform for the right tempered fractional integral that $\mathcal{F} (\TI{x}{\infty}{\mu} u)=(\lambda-i\omega)^{-\rho}\mathcal{F} (u)$.

Introduce the space % of functions
\begin{equation}\label{Wlambda}
 W^{\rho,2}_{\lambda}(\mathbb{R}):=\Big\{u\in L^2(\mathbb{R})\,:~\int_{\mathbb{R}}(\lambda^2+\omega^2)^{\rho}\,\big|\mathscr{F} [u](\omega)\big|^2{\rm d}\omega <\infty\Big\},\quad \rho,\lambda\in {\mathbb R}^+.
\end{equation}
Thanks to the Parseval's identity \eqref{parseval},  the above tempered fractional derivatives are  well-defined for any $u\in W^{\rho,2}_{\lambda}(\mathbb{R}).$ Moreover, one verifies from \eqref{Fouriertran+}-\eqref{Temdef2Fourier}  that
\begin{equation}\begin{aligned}\label{newidentiy}
&\TI{-\infty}{x}{\mu}\TD{-\infty}{x}{\mu} u(x)=u(x), \;\;\;\;\TI{x}{\infty}{\mu}\TD{x}{\infty}{\mu} u(x)=u(x), \quad\forall\, u\in W^{\rho,2}_{\lambda}(\mathbb{R});\;\;
\\
&\TD{-\infty}{x}{\mu}\TI{-\infty}{x}{\mu}  u(x)=u(x), \;\;\;\;\TD{x}{\infty}{\mu}\TI{x}{\infty}{\mu} u(x)=u(x),\quad \forall\, u\in L^2(\mathbb R).
\end{aligned}\end{equation}

%Recall the  usual left and right Riemann-Liouville fractional derivatives are  defined by
%\begin{equation}\label{DformA}
%\begin{split}
%\D{-\infty}{x}{\mu} u(x)&= \D{}{}{k} \I{-\infty}{x}{k-\mu} u(x)= \frac 1{\Gamma(k-\mu)}\frac{d^k}{dx^k}\int_{-\infty}^x \frac{u(y)}{(x-y)^{\mu-k+1}}\, {\rm d}y,\\
%\D{x}{\infty}{\mu} u(x)&= (-1)^k \D{}{}{k} \I{x}{\infty}{k-\mu} u(x)= \frac {(-1)^k}{\Gamma(k-\mu)}\frac{d^k}{dx^k}\int^{\infty}_x \frac{u(y)}{(y-x)^{\mu-k+1}}\, {\rm d}y,\\
%\end{split}
%\end{equation}
%for $\mu\in [k-1,k)$ with $k\in {\mathbb N}.$    Note that we have ${\mathscr F}[\D{}{\pm}{\mu} u](\omega)= {(\pm \ri \omega)^{\rho}} {\mathscr F}[u](\omega).$
%\begin{equation}\label{FourierABC}
%{\mathscr F}[{D}_{\!\pm}^{\rho} u](\omega)= {(\pm \ri \omega)^{\rho}} {\mathscr F}[u](\omega).
%\end{equation}

Similar to \eqref{relationA}, we have the following   explicit representations. %of  $\mathbb{D}^{\rho,\lambda}_\pm u$  in terms of
\begin{prop}\label{newcoms}  For any $u\in W^{\rho,2}_{\lambda}(\mathbb{R}),$ with
$ \lambda\in {\mathbb R}^+_0$, the left and right  tempered fractional derivatives of order $\mu\in [k-1,k)$ with $k\in {\mathbb N},$    have the explicit representations:
\begin{equation}\label{temperedderiLeft}
\TD{-\infty}{x}{\mu} u(x)=e^{-\lambda x} \D{-\infty}{x}{\mu} \big\{e^{\lambda x}u(x)\big\},\quad \TD{x}{\infty}{\mu} u(x)=e^{\lambda x} \D{x}{\infty}{\mu}\big\{e^{-\lambda x}u(x)\big\},
\end{equation}
where $\D{-\infty}{x}{\mu}$ and $\D{x}{\infty}{\mu}$ are the Riemann-Liouville fractional derivative operators  in  Definition \ref{RLFDdefn}.
Alternatively,  %for any  $u\in W^{\rho,2}_{\lambda}(\mathbb{R}),$
 we have
\begin{equation}\label{temperedderiLeftB}
\begin{split}
\TD{-\infty}{x}{\mu} u(x)&=(\D{}{}{}+\lambda)^k\,\big\{\TI{-\infty}{x}{k-\mu} u(x)\big\}=(\D{}{}{}+\lambda)^k\,\big\{e^{-\lambda x}\I{-\infty}{x}{k-\mu}\big\{e^{\lambda x} u(x)\big\}\big\}; \\
\TD{x}{\infty}{\mu} u(x)&=(\D{}{}{}-\lambda)^k\,\big\{\I{x}{\infty}{k-\mu} u(x)\big\}=(\D{}{}{}-\lambda)^k\,\big\{e^{\lambda x}\I{x}{\infty}{k-\mu}\big\{e^{-\lambda x} u(x)\big\}\big\}.
\end{split}
\end{equation}
\end{prop}
\begin{proof} Using the properties of Fourier transform:
${\mathscr F}[e^{-\lambda x}D^k v]=(\lambda+\ri \omega)^k  {\mathscr F}[e^{-\lambda x}v],$ and
\begin{equation}\label{neweqnA}
\begin{split}
{\mathscr F}\big[e^{-\lambda x} & \D{-\infty}{x}{\mu} \big\{e^{\lambda x}u(x)\big\}\big](\omega)=
{\mathscr F}\big[e^{-\lambda x} \D{}{}{k} \I{-\infty}{x}{k-\mu} \big\{e^{\lambda x}u(x)\big\}\big](\omega)\\
& =(\lambda+\ri \omega)^k {\mathscr F}\big[e^{-\lambda x} \I{-\infty}{x}{k-\mu}\big\{e^{\lambda x}u(x)\big\}\big](\omega)
\\
& =(\lambda+\ri \omega)^k {\mathscr F}\big[\TI{-\infty}{x}{k-\mu}u(x)\big](\omega)=(\lambda+\ri \omega)^\mu {\mathscr F}[u](\omega).
\end{split}
\end{equation}
This verifies \eqref{Temdef2FourierB}-\eqref{Temdef2Fourier}.
Similarly, we can derive the second representation of  $\TD{x}{\infty}{\rho} u$ in \eqref{temperedderiLeft}.

The alternative form \eqref{temperedderiLeftB} can be derived   by induction. Here we only
verify the left tempered derivative. % as right    tempered derivative can be proved similarly:
\begin{itemize}
\item For $\mu\in[0,1)$,  we derive from \eqref{relationA}  and \eqref{temperedderiLeft} that
\begin{align*}
\TD{-\infty}{x}{\mu} u(x)&=e^{-\lambda x} \D{-\infty}{x}{\mu} \big\{e^{\lambda x}u(x)\big\}=
e^{-\lambda x}\,{\rm D} {}_{-\infty}{\rm I}_x^{1-\mu} \big\{e^{\lambda x}u(x)\big\}
\\&=e^{-\lambda x}\,{\rm D} \big[e^{\lambda x}{}_{-\infty}{\rm I}_x^{1-\mu,\lambda} \big\{e^{\lambda x}u(x)\big\}\big]=
({\rm D}+\lambda) \big\{{}_{-\infty}{\rm I}_x^{1-\mu,\lambda} \big\{e^{\lambda x}u(x)\big\}\big\},
\end{align*}
which verifies the identity with $k=1.$
%by derivative rule,  it's easy to derive that $$(\D{}{}{}+\lambda)\,\big\{\TI{-\infty}{x}{1-\mu} u(x)\big\}=\TD{-\infty}{x}{\mu} u(x).$$

\item  For $\mu\in{[k-2,k-1)}$, we assume that \eqref{temperedderiLeftB} is true.
%$\TD{-\infty}{x}{\mu} u(x)=(\D{}{}{}+\lambda)^k\,\big\{\TI{-\infty}{x}{k-\mu}  u(x)\big\}$,
We next verify the identity holds for $\mu\in[k-1,k)$.
\begin{equation*}
\begin{split}
& (\D{}{}{}+\lambda)^k\, \big\{\TI{-\infty}{x}{k-\mu}  u(x)\big\}=(\D{}{}{}+\lambda)(\D{}{}{}+\lambda)^{k-1}\,\big\{e^{-\lambda x}\I{-\infty}{x}{k-\mu}\big\{e^{\lambda x} u(x)\big\}\big\}
 \\&~~=(\D{}{}{}+\lambda)\big\{e^{-\lambda x}\D{-\infty}{x}{\mu-1}\big\{e^{\lambda x} u(x)\big\}\big\}
%\\&
=e^{-\lambda x}\D{-\infty}{x}{\mu}\big\{e^{\lambda x} u(x)\big\}=\TD{-\infty}{x}{\mu} u(x).
\end{split}
\end{equation*}
\end{itemize}
%\begin{equation}
% {\mathscr F}[\mathbb{D}_-^{\rho,\lambda} u](\omega)={\mathscr F}[e^{\lambda x} D_{\!-}^\mu\big\{e^{-\lambda x}u(x)\big\}](\omega)=(\lambda-\ri \omega)^\mu {\mathscr F}[u](\omega).
%\end{equation}
This ends the proof.
\end{proof}

We collect below some useful properties (cf. \cite{temperedsabzikar2014}). % and summarized in the following Lemma.
\begin{lemma}\label{temperedProperty}
Given $\lambda>0$ and $\mu\in[k-1,k),~k\in{\mathbb{N}}$, the tempered fractional derivative
\begin{equation}\label{Temdef1}
%\TD{}{\pm}{\mu} u(x)=\TD{}{\pm}{k} \TI{}{\pm}{k-\mu} u(x).
\TD{-\infty}{x}{\mu} u(x)=\TD{-\infty}{x}{k} \TI{-\infty}{x}{k-\mu} u(x),\quad \TD{x}{\infty}{\mu} u(x)=\TD{x}{\infty}{k} \TI{x}{\infty}{k-\mu} u(x).
\end{equation}
 In addition, we have
\begin{align}
&\TI{-\infty}{x}{\mu+\nu} u(x)=\TI{-\infty}{x}{\mu}\TI{-\infty}{x}{\nu} u(x),\quad\quad\TI{x}{\infty}{\mu+\nu} u(x)=\TI{x}{\infty}{\mu} \TI{x}{\infty}{\nu} u(x),\label{semigroup1}\\
&\TD{-\infty}{x}{\mu+\nu} u(x)=\TD{-\infty}{x}{\mu}\TD{-\infty}{x}{\nu} u(x),\quad\TD{x}{\infty}{\mu+\nu} u(x)=\TD{x}{\infty}{\mu}\TD{x}{\infty}{\nu} u(x),\label{semigroup2}\\
 &(\TD{-\infty}{x}{\mu} u,v)=( u,\TD{x}{\infty}{\mu} v ),\qquad\qquad\qquad (\TD{x}{\infty}{\mu} u,v)=( u, \TD{-\infty}{x}{\mu}v ), \label{integralbypart}
  \end{align}
  where $ \mu,\nu\geq0$. %and $ \langle u,v\rangle=\int_{\mathbb{R}}u\bar v~ dx.$
\end{lemma}
\begin{rem}
For a suitable function  $f(x),~x\in\mathbb{R}^+$,  its reflection $g(y)=f(-y),~y\in\mathbb{R}^-$ satisfies
\begin{equation}\label{rela+-}
\begin{aligned}
\TI{-\infty}{y}\mu g(y)=&\frac{e^{-\lambda y}}{\Gamma(\mu)}\int_{-\infty}^y e^{\lambda \tau}(y-\tau)^{\mu-1}f(-\tau){\rm d}\tau\overset{t=-\tau}{=}\frac{e^{-\lambda y}}{{\Gamma(\mu)}}\int^{\infty}_{-y} e^{-\lambda t}(y+t)^{\mu-1}f(t)\rm {\rm d}t\\
\overset{x=-y}{=}&\frac{e^{\lambda x}}{{\Gamma(\mu)}}\int^{\infty}_{x} e^{-\lambda t}(t-x)^{\mu-1}f(t){\rm d}t=\TI{x}{\infty}\mu f(x).
\end{aligned}
\end{equation}
Hence, we can use \eqref{Temdef1} and  derivative relation $\dfrac{{\rm d}^k}{{\rm d}y^k}=(-1)^k\dfrac{{\rm d}^k}{{\rm d}x^k}$ to obtain the tempered derivative relation
\begin{equation}\label{rela+-D}
\TD{-\infty}{y}\mu f(-y)=\TD{x}{\infty}\mu f(x),\quad y=-x,~x\in\mathbb{R}^+.
\end{equation}
\qed
\end{rem}

\subsection{Laguerre polynomials and some useful formulas}

For any $a\in {\mathbb R}$ and  $j\in {\mathbb N}_0,$ we recall that the rising factorial in the Pochhammer symbol and the Gamma function  have the relation:
\begin{equation}\label{anotation}
(a)_0=1; \;\;\;  % (a)_j=a(a+1)\cdots (a+j-1),\;\;  {\rm for}\;\; j>0;\quad
(a)_j:=a(a+1)\cdots (a+j-1)=\frac{\Gamma(a+j)}{\Gamma(a)},\;\; {\rm for}\;\; j\ge 1.
\end{equation}
%whose relation with the Gamma function is given  by
%\begin{equation}\label{binomialaj}
%%{{j+a-1} \choose j}=\frac{(a)_j}{j!},\quad a\in {\mathbb R},\;\; j\in {\mathbb N};\quad
%(a)_j=\frac{\Gamma(a+j)}{\Gamma(a)},\;\;\; a\in {\mathbb R}^+.
%\end{equation}
%Moreover, we have
%\begin{equation}\label{binomialaj}
%{a \choose b}=\frac{\Gamma(a+1)}{\Gamma(b+1)\Gamma(a-b+1)},\quad  a,b>-1,\;\; a-b>-1,\;\; a,b\in {\mathbb R}.
%\end{equation}
Recall the hypergeometric function (cf.  \cite{Abr.I64}):
\begin{equation}\label{hyperboscs}
{}_1F_1(a;b; x)=\sum_{j=0}^\infty \frac{(a)_j }{(b)_j}\frac{x^j}{j!},\;\;\; a,b,x\in {\mathbb R}^+,\;\; -b\not \in {\mathbb N}_0.
\end{equation}
%where $a,b,c\in {\mathbb R}$ and $-c\not \in {\mathbb N}.$
   If  $b-a>0,$ then ${}_1F_1(a;b; x)$ is absolutely convergent for all $x\in \mathbb R.$  If $a$  is a negative integer, then it reduces to a  polynomial.

 The Laguerre polynomial
  with  parameter $\alpha>-1$ is defined as in  Szeg\"o \cite[(5.3.3)]{szeg75}:
\begin{equation}\label{Laguerre}
L_n^{(\alpha)}(x)=\frac{(\alpha+1)_n}{n!}~{}_1F_1\big(-n;\alpha+1;x\big),\;\;\; n\geq 1,\;\; x\in {\mathbb R}^+,
\end{equation}
and  $L_0^{(\alpha)}(x)\equiv 1.$
Note that
\begin{equation}\label{Laguzerovalue}
L_n^{(\alpha)}(0)=\dfrac{(\alpha+1)_n}{n!},
 \end{equation}
and the  Laguerre polynomials (with  $\alpha>-1$)  are orthogonal with respect to the weight function $x^\alpha e^{-x},$ %$\omega^{\alpha}(x) = x^\alpha e^{-x},$
 namely,  %on $(-1,1)$:
\begin{equation}\label{Lagorth}
    \int_{0}^\infty {L}_n^{(\alpha)}(x)~ {L}_m^{(\alpha)}(x) ~x^\alpha e^{-x} \, {\rm d}x= \gamma _n^{\alpha}\, \delta_{mn},
    \quad  \gamma _n^{\alpha} =\frac{\Gamma(n+\alpha+1)}{\Gamma(n+1)}.
\end{equation}
%where the normalisation constant
%\begin{equation}\label{co-gamma}
%\gamma _n^{\alpha} =\frac{\Gamma(n+\alpha+1)}{\Gamma(n+1)}.
%\end{equation}
They are eigenfunctions of the Sturm-Liouville problem:
\begin{equation}\label{LaguSLProb}
 x^{-\alpha}e^x\partial_x\big( x^{\alpha+1}e^{-x}\partial_xL_n^{(\alpha)}(x)\big)+\lambda_n L_n^{(\alpha)}(x)=0,\quad \lambda_n=n.
\end{equation}
We have the following relations:
\begin{equation}
L_n^{(\alpha)}(x)=\partial _x L_{n}^{(\alpha)}(x)-\partial _x L_{n+1}^{(\alpha)}(x),\label{LaguDerivative1}
\end{equation}
\begin{equation}
 x\partial _x L_n^{(\alpha)}(x)=nL_{n}^{(\alpha)}(x)-(n+\alpha)L_{n-1}^{(\alpha)}(x),\label{LaguDerivative2}
\end{equation}
\begin{equation}
\partial _x L_n^{(\alpha)}(x)=- L_{n-1}^{(\alpha+1)}(x)=-\sum_{k=0}^{n-1}L_{k}^{(\alpha)}(x).\label{LaguDerivative3}
\end{equation}
%In particular, $\gamma _n^{\alpha}={h}_n^{\alpha,\alpha}$.
In particular, for $\alpha=-k,~k=1,2,\ldots$ (See Szeg\"o \cite[(5.2.1)]{szeg75}),
\begin{equation*}
L_n^{(-k)}(x)=(-1)^k\frac{\Gamma(n-k+1)}{\Gamma(n+1)}x^k L_{n-k}^{(k)}(x),\quad n\geq k.
\end{equation*}

For notational convenience, we denote
\begin{equation}\label{hnab}
{h}_n^{a,b}:=\frac{\Gamma(n+1+a)}{\Gamma(n+1+a-b)}.
\end{equation}
We present below some formulas related to Laguerre polynomials and fractional integrals and derivatives, which play an important role  in the  algorithm development and analysis later.  We provide their derivations in Appendix \ref{AppenixB}.
\begin{lemma}\label{Lemmalagupolyderi} For  $\rho\in  {\mathbb R}^+,$ we have
\begin{equation}\begin{aligned}\label{lagupolyint}
\I{0}{x}\rho \{x^\alpha L_n^{(\alpha)}(x)\}=
%\frac{\Gamma(n+\alpha+1)}{\Gamma(n+\alpha+1+\rho)}
h^{\alpha,-\rho}_n\,x^{\alpha+\rho} L_n^{(\alpha+\rho)}(x),\quad \alpha>-1;
\end{aligned}\end{equation}
\begin{equation}\begin{aligned}\label{lagupolyderi}
\D{0}{x}\mu \{x^{\alpha} L_n^{(\alpha)}(x)\}=
%\frac{\Gamma(n+\alpha+1)}{\Gamma(n+\alpha-s+1)}
h^{\alpha,\mu}_n\, x^{\alpha-\mu} L_n^{(\alpha-\mu)}(x),\quad \alpha>\mu-1,
\end{aligned}\end{equation}
and
\begin{align}
& \I{x}{\infty}\rho \{e^{-x} L_n^{(\alpha)}(x)\}= e^{-x} L_n^{(\alpha-\mu)}(x), \quad \alpha>\mu-1; \label{Laguerreintegralr}\\[4pt]
%\end{equation}
%\begin{equation}
& \D{x}{\infty}\mu \{ e^{-x} L_n^{(\alpha)}(x)\}=e^{-x} L_n^{(\alpha+\mu)}(x),\quad \alpha>-1. \label{Laguerrederivativer}
\end{align}
%In addition, if $\alpha-[\mu]>0$, then
%\begin{equation}\label{lagupolycaputoderi}
%\CD{0\,}{x}\mu \{x^{\alpha} L_n^{(\alpha)}(x)\}=h^{\alpha,\mu}_n~x^{\alpha-\mu} L_n^{(\alpha-\mu)}(x).
%\end{equation}
Moreover, we have that for  $k\in\mathbb{N}$ and $\alpha>k-1$,
 \begin{equation}\label{lagufunderi}
 \D{}{}{k}  \big\{x^\alpha e^{-x} L_{n}^{(\alpha)}(x)\big\}=\frac{\Gamma(n+k+1)}{\Gamma(n+1)}x^{\alpha-k}L^{(\alpha-k)}_{n+k}(x)e^{-x}.
 \end{equation}
\end{lemma}
%We also refer to Appendix \ref{AppenixA:Laguerre} for some other properties.

\section{Generalized Laguerre functions}\label{SectionLaguerre}
\setcounter{equation}{0}
\setcounter{lmm}{0}
\setcounter{thm}{0}

In this section,  we introduce the generalized Laguerre functions (GLFs), and study its approximation properties.
%It is seen from Definition \ref{temperedDeri}  that the tempered fractional derivatives are introduced through the Fourier transform, having the alternative representations in Proposition \ref{newcoms}. %, when $u(x)$ is defined on the whole line.
In what follows, the operators  $\TI{0}{x}\mu, \TD{0}{x}\mu$ on the half line  should be understood as $0$ in place of $-\infty$ in   \eqref{temperedInt+} and
\eqref{temperedderiLeft}-\eqref{temperedderiLeftB}.
\subsection{Definition and properties}
We first introduce the  GJFs and their associated properties related to tempered fractional integrals/derivatives.
\begin{defn}[{\bf GLFs}]\label{Def_GLFs} {\em
For real $\alpha\in\mathbb{R}$ and $\lambda>0,$ we define the GLFs as
\begin{equation}\label{GLFs}
\GLa{\alpha}{n}{(x)}:=\begin{cases}
x^{-\alpha} e^{-\lambda x}\,{L}_{n}^{(-\alpha)}(2\lambda x),\quad \alpha<0,\\
e^{-\lambda x}\,{L}_{n}^{(\alpha)}(2\lambda x),\qquad\quad\, \alpha\geq 0
\end{cases}
\end{equation}
for  all $x\in \mathbb{R}^+$ and $n\in\mathbb{N}_0.$}
\end{defn}
\begin{rem}
It's noteworthy that Zhang and Guo \cite{GuoZhangC2012laguerre} introduce the GJFs
\begin{equation}\label{GLFsGuoZhang}
\widetilde{\mathscr{L}}_l^{(\alpha,\beta)}(x)
=\begin{cases}
x^{-\alpha}e^{-\frac{\beta}{2}x}L^{(-\alpha)}(\beta x),\quad &\alpha\leq-1,\;\; l\geq \bar{l}_\alpha=[-\alpha],\\
e^{-\frac{\beta}{2}x}L_l^{(\alpha)}(x),&\alpha>-1,\;\;  l\geq\bar{l}_\alpha=0,
\end{cases}
%\quad L_l^{(\alpha,\beta)}(x)=L_l^{(\alpha)}(\beta x).
\end{equation}
where the scaling factor $\beta>0.$ It is seen  that
we modified the definition in the range of $0<\alpha<1$ (with $\beta=2\lambda$). This turns out to be essential for the numerical solution of FDEs of order $\mu\in (0,1),$ as we shall see in the subsequent sections.
\qed
\end{rem}

We next present the basic properties of GLFs.
Firstly,  one verifies readily from the orthogonality \eqref{Lagorth} and  Definition \ref{Def_GLFs} that for $\alpha\in\mathbb{R}$ and $\lambda>0,$
  \begin{equation}\label{Temp-orth}
\int_0^\infty \GLa{\alpha}{n}{(x)}\GLa{\alpha}{m}{(x)}\,x^{\alpha}dx=\gamma_n^{{|\alpha|},\lambda}\delta_{nm}, \;\;\;
\gamma _n^{|\alpha|,\lambda} =\frac{\gamma_n^{|\alpha|}}{(2\lambda)^{|\alpha|+1}},
\end{equation}
where $\gamma_n^{|\alpha|}$ is defined in \eqref{Lagorth}.
%%\begin{equation}\label{Temp-orth}
%%\int_0^\infty \GLa{-s}{n}{(x)}\GLa{-s}{m}{x}\omega^{-s}dx=\int_0^\infty \GLb{s}{n}{(x)}\GLb{s}{m}{x}\omega^{s}dx=\gamma_n^{s,\lambda}\delta_{nm},
%%\end{equation}
%where
%\begin{equation}\label{Tempco-gamma}
%\gamma _n^{\mu,\lambda} =\frac{\Gamma(n+\mu+1)}{(2\lambda)^{\mu+1}\,\Gamma(n+1)}.
%\end{equation}
%\textbf{(ii) Tempered integral and derivative relations:}
%via the relation \eqref{transform xt}, \eqref{transform2 xt} and the definitions of the tempered derivative, it can be directly derived from Lemma \ref{Lemmalagupolyderi}  that

We have the following  important (left) ``tempered" fractional integral and derivative rules.
\begin{lemma}\label{LemmaTemplagupolyderi} For  $\mu,\nu,\lambda,x \in  {\mathbb R}^+_0,$ we have
\begin{equation}\begin{aligned}\label{TempGLderi1}
\TI{0}{x}{\mu} \GLa{-\nu}{n}{(x)}=h^{\nu,-\mu}_n\GLa{-\nu-\mu}{n}{(x)},\qquad \quad
\end{aligned}\end{equation}
\begin{equation}\begin{aligned}\label{TempGLderi2}
\TD{0}{x}{\mu} \GLa{-\nu}{n}{(x)}=h^{\nu,\mu}_n\GLa{\mu-\nu}{n}{(x)},\quad \nu\geq\mu,
\end{aligned}\end{equation}
and %in particular, there holds % for any positive integer $k$, there holds
\begin{equation}\label{Tempderirelamuk}
\TD{0}{x}{\mu+k} \GLa{-\mu}{n}{(x)}=(-2\lambda)^kh^{\mu,\mu}_n \GLa{k}{n-k}{(x)},\quad n\ge k\in {\mathbb N}_0,
\end{equation}
where $h^{a,b}_n$ is defined in \eqref{hnab}.
\end{lemma}
\begin{proof}
%Since the left tempered fractional integral $\TI{0}{x}{\mu}$ and derivative $\TD{0}{x}{\mu}$  be understood as replacing $-\infty$ by $0$ in   \eqref{temperedInt+} and
%\eqref{temperedderiLeft}-\eqref{temperedderiLeftB}, respectively.
We obtain from  \eqref{temperedInt+} and
\eqref{temperedderiLeft}-\eqref{temperedderiLeftB} (with  replacing $-\infty$ by $0$)
 that
\begin{equation*}
\TI{0}{x}{\mu}\GLa{-\nu}{n}{(x)}=e^{-\lambda x} \I{0}{x}{\mu} \{e^{\lambda x} \GLa{-\nu}{n}{(x)}\}=e^{-\lambda x} \I{0}{x}{\mu} \{ x^\nu L^{(\nu)}_n(2\lambda x)\},
\end{equation*}
and
\begin{equation*}
\TD{0}{x}{\mu}\GLa{-\nu}{n}{(x)}=e^{-\lambda x} \D{0}{x}{\mu} \{e^{\lambda x} \GLa{-\nu}{n}{(x)}\}=e^{-\lambda x} \D{0}{x}{\mu} \{x^\nu L^{(\nu)}_n(2\lambda x)\}.
\end{equation*}
Thus, from \eqref{transform xt} and Lemma \ref{Lemmalagupolyderi},  we obtain  \eqref{TempGLderi1}-\eqref{TempGLderi2}.

Using \eqref{TempGLderi2} and the  derivative relation \eqref{LaguDerivative3} (with $\alpha=\mu$), we obtain
\begin{equation*}\begin{aligned}
\TD{0}{x}{\mu+k} \GLa{-\mu}{n}{(x)}&=\TD{0}{x}{k} \TD{0}{x}{\mu} \GLa{-\mu}{n}{(x)}=e^{-\lambda x}\D{}{}{k} \big\{e^{\lambda x} h^{\mu,\mu}_n\GLa{0}{n}{(x)}\big\}\\&=h^{\mu,\mu}_n\,e^{-\lambda x}\D{}{}{k} \big\{L_{n}^{(0)}(2\lambda x)\big\}=(-2\lambda)^kh^{\mu,\mu}_n \,L_{n-k}^{(k)}(2\lambda x)\,e^{-\lambda x}.
\end{aligned}\end{equation*}
This leads to \eqref{Tempderirelamuk}.
\end{proof}

Similarly, we have the following rules of the (right) ``tempered" fractional integrals and derivatives.
\begin{lemma}\label{LemmaTemplagufunderi} %Let $y=-x,~x\in{\mathbb R}^+_0$.
For  $ \mu,\nu,\lambda, x \in  {\mathbb R}^+_0$,  we have
\begin{equation}\begin{aligned}\label{TempGLderi1II}
\quad\TI{x}{\infty}\mu \GLa{\nu}{n}{(x)}=(2\lambda)^{-\mu}\GLa{\nu-\mu}{n}{(x)},\quad \nu\geq\mu,
\end{aligned}\end{equation}
\begin{equation}\begin{aligned}\label{TempGLderi2II}
\TD{x}{\infty}{\mu} \GLa{\nu}{n}{(x)}=(2\lambda)^{\mu}\GLa{\mu+\nu}{n}{(x)}.\quad \quad\quad
\end{aligned}\end{equation}
\begin{comment}
Furthermore, let $y=-x$, then
\begin{equation}\begin{aligned}\label{TempGLderi1II_R}
\quad\TI{-\infty}{y}\mu \GLa{\nu}{n}{(-y)}=(2\lambda)^{-\mu}\GLa{\nu-\mu}{n}{(-y)},\quad \nu\geq\mu,
\end{aligned}\end{equation}
\begin{equation}\begin{aligned}\label{TempGLderi2II_R}
\TD{-\infty}{y}{\mu} \GLa{\nu}{n}{(-y)}=(2\lambda)^{\mu}\GLa{\mu+\nu}{n}{(-y)}.\quad \quad\quad
\end{aligned}\end{equation}
\end{comment}
\end{lemma}
\begin{proof}
Identities \eqref{TempGLderi1II} and \eqref{TempGLderi2II} can be easily derived from \eqref{transform xt}, \eqref{transform2 xt} and  Lemma \ref{Lemmalagupolyderi}. %Then, via the relations \eqref{rela+-} and \eqref{rela+-D}, it's easy to derive \eqref{TempGLderi1II} and \eqref{TempGLderi2II}.
\end{proof}
%Moreover, for any $\GLa{-k}{n}{x}=x^k\GLb{k}{n}{x},~k\in{N}_0$, via lemma \ref{Lemmalagupolyderi} and relations
%%$$\TD{}{\pm}{\mu}u(x)=\TD{}{\pm}{k}\TD{}{\pm}{k-\mu}u(x)=\TI{}{\pm}{k-\mu}\TD{}{\pm}{k}u(x)$$
%$$\TD{-\infty}{x}{k}u=e^{-\lambda x}\D{}{}{k}\big\{e^{ \lambda x}u\big\},\quad  \TD{x}{\infty}{k}u=e^{\lambda x}(-1)^k\D{}{}{k}\big\{e^{- \lambda x}u\big\},$$ we have that

We highlight the fractional  derivative formulas, which play an important role in the forthcoming  algorithm and analysis.
\begin{thm}\label{LemmaGLFkderi} Let  $k\in\mathbb{N}$ and $k-\nu\leq 0$,
\begin{equation}\begin{aligned}\label{GLFkderi++}
\TD{-\infty}{x}{k} \big\{ \GLa{-\nu}{n}{(x)}\big\}=
%\frac{\Gamma(n+\alpha+1)}{\Gamma(n+\alpha-s+1)}
\frac{\Gamma(n+\nu+1)}{\Gamma(n+\nu-k+1)}~\GLa{k-\nu}{n}{(x)} ,
\end{aligned}\end{equation}
 \begin{equation}\label{GLFkderi+-}
 \qquad\TD{x}{\infty}{k}  \big\{ \GLa{-\nu}{n}{(x)}\big\}=(-1)^k\frac{\Gamma(n+k+1)}{\Gamma(n+1)}\GLa{k-\nu}{n+k}{(x)}.\qquad
 \end{equation}
\end{thm}
\begin{proof}
From Lemma \ref{Lemmalagupolyderi}
and relations
\begin{equation}\label{DkmuA}
\TD{-\infty}{x}{k}u=e^{-\lambda x}\D{}{}{k}\big\{e^{ \lambda x}u\big\},\quad  \TD{x}{\infty}{k}u=e^{\lambda x}(-1)^k\D{}{}{k}\big\{e^{- \lambda x}u\big\},
\end{equation}
 we obtain that for $k-\nu\leq 0$,
\begin{equation*}\begin{aligned}
\TD{-\infty}{x}{k} &  \big\{x^{\nu} \GLa{\nu}{n}{(x)}\big\}=e^{-\lambda x}\D{}{}{k}\big\{(2\lambda)^{-\nu}(2\lambda x)^\nu L^{(\nu)}_n(2\lambda x)\big\}\\
\overset{\eqref{lagupolyderi}}{=}
&\frac{\Gamma(n+1+\nu)}{\Gamma(n+\nu-k+1)} x^{\nu-k} L^{(\nu-k)}_n(2\lambda x)e^{-\lambda x}
=\frac{\Gamma(n+1+\nu)}{\Gamma(n+\nu-k+1)}~\GLa{k-\nu}{n}{(x)},
\end{aligned}\end{equation*}
and
\begin{equation}\begin{aligned}
\TD{x}{\infty}{k} \big\{x^{\nu} \GLa{\nu}{n}{(x)}\big\}=&e^{\lambda x}(-1)^k\D{}{}{k}\big\{(2\lambda)^{-\nu}(2\lambda x)^\nu L^{(\nu)}_n(2\lambda x)e^{-2\lambda x}\big\}\\
\overset{\eqref{lagufunderi}}{=}
&(-1)^k\frac{\Gamma(n+k+1)}{\Gamma(n+1)} x^{\nu-k} L^{(\nu-k)}_{n+k}(2\lambda x)e^{-\lambda x}
\\=&(-1)^k\frac{\Gamma(n+k+1)}{\Gamma(n+1)}~\GLa{k-\nu}{n}{(x)}.
\end{aligned}\end{equation}
This ends the proof.
\end{proof}

Another attractive property of GLFs is that they are eigenfunctions of  Sturm-Liouville problem.
\begin{thm}\label{thm_SL}
Let $s,\nu,x\in\mathbb{R}^+_0$ and $n\in\mathbb{N}_0$. Then,
\begin{equation}\label{ST-}
\qquad x^{\nu}\TD{x}{\infty}{s}\{ x^{s-\nu} \TD{0}{x}{s}\,\GLa{-\nu}{n}{(x)}\}=\lambda_{n,-}^{s,\nu}  \,\GLa{-\nu}{n}{(x)},\quad \nu-s\geq 0,
\end{equation}
and
\begin{equation}\label{ST+}
 x^{-\nu}\TD{0}{x}{s}\{ x^{s+\nu} \TD{x}{\infty}{s}\,\GLa{\nu}{n}{(x)}\}=\lambda_{n,+}^{s,\nu} \,\GLa{\nu}{n}{(x)},\qquad\qquad\,
\end{equation}
where the corresponding  eigenvalues $\lambda_{n,-}^{s,\nu}=(2\lambda)^s h_n^{\nu,s}$ and $\lambda_{n,+}^{s,\nu}=(2\lambda)^s h_n^{\nu+s,s}$.
\end{thm}
\begin{proof}
Due to \eqref{TempGLderi2} and \eqref{TempGLderi2II},
\begin{equation*}
\TD{0}{x}{s} \GLa{-\nu}{n}{(x)}=h^{\nu,s}_n\GLa{s-\nu}{n}{(x)}, \quad \TD{x}{\infty}{s} \GLa{\nu-s}{n}{(x)}=(2\lambda)^{s}\GLa{\nu}{n}{(x)}.
\end{equation*}
It's straightforward to obtain that
\begin{equation*}\begin{aligned}
x^{\nu}&\TD{x}{\infty}{s}\{ x^{s-\nu} \TD{0}{x}{s}\,\GLa{-\nu}{n}{(x)}\}=x^{\nu}\TD{x}{\infty}{s}\{ x^{s-\nu} \TD{0}{x}{s}\GLa{-\nu}{n}{(x)} \}
\\&=h^{\nu,s}_n x^{\nu}\TD{x}{\infty}{s}\{ x^{s-\nu} \GLa{s-\nu}{n}{(x)} \}
=h^{\nu,s}_n x^{\nu}\TD{x}{\infty}{s} \GLa{\nu-s}{n}{(x)}=(2\lambda)^s h_n^{\nu,s} \,\GLa{-\nu}{n}{(x)}.
\end{aligned}\end{equation*}
Similarly, we have
\begin{equation*}\begin{aligned}
x^{-\nu}&\TD{0}{x}{s}\{ x^{s+\nu} \TD{x}{\infty}{s}\,\GLa{\nu}{n}{(x)}\}=x^{-\nu}\TD{0}{x}{s}\{ x^{s+\nu} \TD{x}{\infty}{s}\GLa{\nu}{n}{(x)} \}\\&=(2\lambda)^s x^{-\nu}\TD{0}{x}{s}\{ x^{s+\nu} \GLa{s+\nu}{n}{(x)} \}=(2\lambda)^s x^{-\nu}\TD{0}{x}{s} \GLa{-\nu-s}{n}{(x)}=(2\lambda)^s h_n^{\nu+s,s} \,\GLa{\nu}{n}{(x)}.
\end{aligned}\end{equation*}
This ends the derivation. 
\end{proof}

\begin{rem}
The above results can be viewed as an extension of the standard Sturm-Liouville problem  of generalized Laguerre functions (cf. \eqref{LaguSLProb})  to the tempered fractional derivative. We derive immediately from \eqref{ST-}, \eqref{ST+} and the Stirling's formula (see \eqref{stirlingfor}) that for fixed $s$ and $\nu$,
\begin{equation*}
\lambda_{n,-}^{s,\nu} =\lambda_{n,+}^{s,\nu}=O\big((2\lambda n)^s\big),\quad n\gg 1.
\end{equation*}
When $s\rightarrow 1$ and $\lambda=1/2$, it recovers the $O(n)$ growth of eigenvalues of the standard Sturm-Liouville problem.\qed
\end{rem}
\subsection{Approximation by GLFs }
\subsubsection{Approximation by $\big\{\GLa{\alpha}{n}{(x)}:~\alpha=-\nu<0\big\}_{n=0}^\infty$. }
Denote by $\mathcal{P}_N$ the set of all  polynomials of degree at most $N$, and define the finite dimensional
 space
\begin{equation}\label{SpaceGLF1}
\mathcal{F}^{\nu,\lambda}_N({\mathbb R}^+):=\big\{x^{\nu}e^{-\lambda x} p(x)\,:\, p\in\mathcal{P}_N\big\}, \quad N\in\mathbb{N}_0.
\end{equation}
Define the $L_\omega^2({\mathbb R}^+)$ with the inner product and norm:
  \begin{equation}
  (f,g)_{\omega}:=\int_{{\mathbb R}^+} f\,\bar g\, \omega\, {\rm d}x, \quad \|f\|_{\omega}^2=(f,f)_{\omega},
  \end{equation}
  where $\omega(x)$ be a generic  weight function  and $\bar g$ is the conjugate of the function $g$.
In particular, we omit $\omega$ when $\omega\equiv 1.$
% Specifically, we
%define the weight functions
%\begin{equation}\label{weight}
%% \omega_{\pm}^{\nu}(x)=x^\nu e^{\pm x},\quad x\in\mathbb{R}^+,
% \omega^{a}(x)=x^a,\quad x\in\mathbb{R}^+,\;\;\; a\in\mathbb{R}.
%\end{equation}

To characterize the approximation errors, we define the non-uniformly weighted Sobolev space
\begin{equation}\label{Soblev Sp1}
A_{\nu,\lambda}^m({\mathbb R}^+):=\Big\{u\in L^2_{\omega^{-\nu}}({\mathbb R}^+): \TD{0}{x}{\nu+k} u \in L^2_{\omega^{k}}({\mathbb R}^+),\;k=0,\cdots,m\Big\},\quad m\in \mathbb{N}_0,
\end{equation}
equipped with the norm and semi-norm
\begin{equation}
\|u\|_{A_{\nu,\lambda}^m}:=\Big(\|u\|^2_{\omega^{-\nu}}+\sum_{k=0}^{m}\|\TD{0}{x}{\nu+k}u\|^2_{\omega^{k}}\Big)^{{1}/{2}},\quad|u|_{A_{\nu,\lambda}^m}:=\|\TD{0}{x}{\nu+m}u\|_{\omega^{m}},
\end{equation}
where  the weight function $\omega^{a}(x)=x^a.$

%{\bf We should be use $\pi_N^{-\nu,\lambda}$ below!!}

Consider  the orthogonal projection $\pi_N^{-\nu,\lambda}:\,{L^2_{\omega^{-\nu}}}({\mathbb R}^+)\rightarrow \mathcal{F}^{\nu,\lambda}_N({\mathbb R}^+)$ defined by
  \begin{equation}\label{Projection1L2}
 (\pi_N^{-\nu,\lambda} u-u, ~\phi)_{\omega^{-\nu}}=0,\quad \forall \phi\in \mathcal{F}^{\nu,\lambda}_N({\mathbb R}^+).
 \end{equation}
 Then, by the orthogonality \eqref{Temp-orth},  $u$ and its $L^2$-orthogonal projection can be expanded as
 \begin{equation}\label{expanform}
 u(x)=\sum_{n=0}^\infty ~\hat{u}_n \GLa{-\nu}{n}{(x)},\quad
  (\pi_N^{-\nu,\lambda} u)(x)=\sum_{n=0}^N ~\hat{u}_n \GLa{-\nu}{n}{(x)},
 \end{equation}
 where  %th $\hat{u}_N$ determined by
 $$\hat{u}_n=\big(u,~{\mathcal L}_n^{(-\nu,\lambda)}\big)_{\omega^{-\nu}}\big/\gamma^{\nu,\lambda}_n.$$
%Now we are ready to obtain the first approximation estimate.
%Moreover, we denote the below finite dimensional function spaces to approximate the above two spaces respectively
%\begin{equation}\label{SpaceGLF1}
%\mathcal{F}^{\nu,\lambda}_N(\mathbb{R}^+):=span\big\{\GLa{-\nu}{n}{(x)}:~n=0,1,\ldots N\big\},
%\end{equation}
%\begin{equation}\label{SpaceGLF2}
%{}^{^-\hspace{-4pt}}\mathcal{F}^{\beta,\lambda}_N(\mathbb{R}^+):=span\big\{\GLb{\beta}{n}{x}:~n=0,1,\ldots N\big\}.
%\end{equation}

%Now we are ready to obtain the main approximation results. For the concision of the procedure, we separate the process into two subsections.
%%\subsubsection{Approximation to space $A^m_{\nu,\lambda}(\Lambda),~\nu>-1,~m\in\mathbb{N}_0$}
%Let $N$ be any positive integer and denote by
% the orthogonal projection $\pi_N^{-\nu,\lambda}:~{L^2_{\omega^{-\nu}}}(\Lambda)\rightarrow \mathcal{F}^{\nu,\lambda}_N(\Lambda)$ which satisfies
%  \begin{equation}\label{Projection1L2}
% (\pi_N^{-\nu,\lambda} u-u, ~\phi)_{\omega^{-\nu}}=0,\quad \forall \phi\in \mathcal{F}^{\nu,\lambda}_N(\Lambda).
% \end{equation}
% Then, by the orthogonality \eqref{Temp-orth}, function $u$ and its projection can be expanded as
% \begin{equation}\label{expanform}
% u=\sum_{n=0}^\infty ~\hat{u}_n \GLa{-\nu}{n}{(x)},\quad
%  \pi_N^{-\nu,\lambda} u=\sum_{n=0}^N ~\hat{u}_n \GLa{-\nu}{n}{(x)},
% \end{equation}
% where the coefficients $\hat{u}_N$ is decided by
% $$\hat{u}_n=\dfrac{1}{\gamma^{\nu,\lambda}_n}\left(u,~\GLa{-\nu}{n}{(x)}\right)_{\omega^{-\nu}}.$$

\begin{thm}\label{ThmAppI}
For  $\lambda, \nu>0$, we have that  for any $u\in A^m_{\nu,\lambda}(\mathbb{R}^+)
$ with $m\leq N+1$,
\begin{equation}\label{Thm31}
\|\pi_N^{-\nu,\lambda} u-u\|_{\omega^{-\nu}}\leq ~c~(2\lambda N)^{-\frac{\nu+m}{2}} ~\|\TD{0}{x}{\nu+m}u\|_{\omega^{m}},%~|u|_{A^m_{\nu,\lambda}}.
\end{equation}
and for any  $k\le m,$ %with $k,m\in\mathbb{N}_0$,
\begin{equation}\label{Thm32}
\big\|\TD{0}{x}{\nu+k}\big(\pi_N^{-\nu,\lambda} u-u\big)\big\|_{\omega^{k}}\leq ~c~({2\lambda}N)^{\frac{k-m}{2}} ~\|\TD{0}{x}{\nu+m}u\|_{\omega^{m}},
\end{equation}
where $c\approx 1$ for large $N.$
\end{thm}
\begin{proof}% {\bf Need to check the proof! The notation is mixed for Laguerre functions?}
%Due to $u\in L^2_{\omega^{-\nu}}(\Lambda)$, $u$ and $\pi_N^{-\nu,\lambda} u$ can be written as
By \eqref{expanform}, we have
\begin{equation*}
(u-\pi_N^{-\nu,\lambda} u)(x)=\sum_{n=N+1}^\infty \,\hat{u}_n\, \GLa{-\nu}{n}{(x)}.
 \end{equation*}
By the orthogonality \eqref{Temp-orth} and \eqref{Tempderirelamuk},
\begin{equation*}
%\|\GLa{-\nu}{n}{(x)}\|^2_{\omega^{-\nu}}=\gamma_n^{\nu,\lambda},\quad
 \big\|\TD{0}{x}{\nu+k}\GLa{-\nu}{n}{}\big\|^2_{\omega^{k}}=(-2\lambda)^{2k}\,(h^{\nu,\nu}_n)^2\int_{0}^\infty \big(L^{(k)}_{n-k}(2\lambda x)\big)^2\,e^{-2\lambda x}{\omega^{k}}(x){\rm d}x=(d_{n,k}^{\nu,\lambda})^2\gamma_{n-k}^{k,\lambda},
\end{equation*}
where we denoted  $d_{n,k}^{\nu,\lambda}:=(2\lambda)^k\,h^{\nu,\nu}_n$ and used the fact: %~k=0,1,\ldots,m.$
\begin{equation*}
\int_0^\infty L^{(m)}_{n-m}(2\lambda x)L^{(k)}_{n-k}(2\lambda x)\,e^{-2\lambda x} \omega^{k}dx=\frac{\gamma_{n-k}^{k}}{(2\lambda)^{k+1}}\delta_{km}=\gamma_{n-k}^{k,\lambda}\delta_{km},
\end{equation*}
Thus we can obtain
\begin{equation*}\begin{aligned}
&\|\pi_N^{-\nu,\lambda} u-u\|^2_{\omega^{-\nu}}
%=\Big\|\sum_{n=N+1}^\infty \hat{u}_n {\mathcal L}^{(-\nu,\lambda)}_{n}\big \|^2_{\omega^{-\nu}}
=\sum_{n=N+1}^\infty(\hat{u}_n)^2\gamma_{n}^{\nu,\lambda},
\quad
\big|\pi_N^{-\nu,\lambda} u-u\big|^2_{A^k_{\nu,\lambda}}
%=\|\sum_{n=N+1}^\infty \hat{u}_n d_{n,k}^{\nu,\lambda} L^{(k)_{n-k}(2\lambda x)\,e^{-\lambda x} \|^2_{\omega^{k}}
=\sum_{n=N+1}^\infty(\hat{u}_nd_{n,k}^{\nu,\lambda})^2\gamma_{n-k}^{k,\lambda}, \\
&\big|u\big|^2_{A^m_{\nu,\lambda}}
%=\|\sum_{n=m}^\infty \hat{u}_n d_{n,m}^{\nu,\lambda} L^{(m)}_{n-m}(2\lambda x)\,e^{-\lambda x} \|^2_{\omega^{m}}
=\sum_{n=m}^\infty(\hat{u}_nd_{n,m}^{\nu,\lambda})^2\gamma_{n-m}^{m,\lambda}.
\end{aligned}\end{equation*}
%Comparing the coefficients of the above equalities,  it's easy to derive the following inequalities
Then one verifies readily that
 \begin{equation*}\begin{aligned}
\|\pi_N^{-\nu,\lambda} u-u\|^2_{\omega^{-\nu}}&\leq\frac{\gamma_{N+1}^{\nu,\lambda}}{(d_{N+1,m}^{\nu,\lambda})^2\gamma_{N+1-m}^{m,\lambda}}\big|u\big|^2_{A^m_{\nu,\lambda}},\\
\quad
\big|\pi_N^{-\nu,\lambda} u-u\big|^2_{A^k_{\nu,\lambda}}&\leq\left(\frac{d_{N+1,k}^{\nu,\lambda}}{d_{N+1,m}^{\nu,\lambda}}\right)^2\frac{\gamma_{N+1-k}^{k,\lambda}}{\gamma_{N+1-m}^{m,\lambda}}\big|u\big|^2_{A^m_{\nu,\lambda}}.
\end{aligned}\end{equation*}
 Recall the property of the Gamma function (see \cite[(6.1.38)]{Abr.I64}):
\begin{equation}\label{stirlingfor}
\Gamma(x+1)=\sqrt{2\pi}
x^{x+1/2}\exp\Big(-x+\frac{\theta}{12x}\Big),\quad \forall\,
x>0,\;\;0<\theta<1.
\end{equation}
One can then obtain that   for any constants  $a, b,$ and for  $n\ge1, $  $n+a>1$ and $n+b>1,$
\begin{equation}\label{Gammaratio}
\frac{\Gamma(n+a)}{\Gamma(n+b)}\le \nu_n^{a,b} n^{a-b},
\end{equation}
where
\begin{equation}\label{ConstUpsilon}
\nu_n^{a,b}=\exp\Big(\frac{a-b}{2(n+b-1)}+\frac{1}{12(n+a-1)}+\frac{(a-b)^2}{n}\Big).
\end{equation}
Therefore,
\begin{equation*}\begin{aligned}
\frac{\gamma_{N+1}^{\nu,\lambda}}{(d_{N+1,m}^{\nu,\lambda})^2\gamma_{N+1-m}^{m,\lambda}}&=\frac{\Gamma(N+2-m)}{(2\lambda)^{\nu+m}\Gamma(N+2+\nu)} \leq (2\lambda)^{-\nu-m} \nu_n^{2-m,2+\nu} N^{-\nu-m}, \\
\left(\frac{d_{N+1,k}^{\nu,\lambda}}{d_{N+1,m}^{\nu,\lambda}}\right)^2\frac{\gamma_{N+1-k}^{k,\lambda}}{\gamma_{N+1-m}^{m,\lambda}}&=\frac{(2\lambda)^{k}\Gamma(N+2-m)}{(2\lambda)^{m}\Gamma(N+2-k)}<(2\lambda)^{k-m} \nu_n^{2-m,2-k} N^{k-m},
\end{aligned}\end{equation*}
where $\nu_n^{2-m,2+\nu}\approx 1$ and $\nu_n^{2-m,2-k}\approx 1$ for fixed $m$ and $n\ge N\gg1 $.
Then \eqref{Thm31}-\eqref{Thm32} follow.
\end{proof}
\subsubsection{Approximation by $\big\{\GLa{\alpha}{n}{(x)}:~\alpha=\nu\geq0\big\}_{n=0}^\infty$. }
Introduce the non-uniformly weighted Sobolev space:
\begin{equation}\label{Soblev Sp2}
B_{\nu,\lambda}^r(\mathbb{R}^+):=\Big\{u\in L^2_{\omega^{\nu}}({\mathbb R}^+): \TD{x}{\infty}{s} u \in L^2_{\omega^{\nu+s}}({\mathbb R}^+),\;0\leq s\leq r\Big\},\quad r\in\mathbb{R}^+_0,
\end{equation}
endowed with norm and semi-norm
\begin{equation}
\|u\|_{B_{\nu,\lambda}^r}:=\Big(\|u\|^2_{\omega^{\nu}}+|u|^2_{B_{\nu,\lambda}^r}\Big)^{{1}/{2}},\quad|u|_{B_\nu^r}:=\|\TD{x}{\infty}{\nu+r}u\|_{\omega^{\nu+r}}.
\end{equation}
Consider
 the orthogonal projection $\Pi_N^{\nu,\lambda}:~{L^2_{\omega^{\nu}}}({\mathbb R}^+)\rightarrow \mathcal{F}^{0,\lambda}_N({\mathbb R}^+),$ defined by
  \begin{equation}\label{Projection2L2}
 \big(\Pi_N^{\nu,\lambda} u-u, ~\phi\big)_{\omega^{\nu}}=0,\quad \forall \phi\in \mathcal{F}^{0,\lambda}_N({\mathbb R}^+),\quad \nu>-1,
 \end{equation}

\begin{thm}\label{ThmAppII}
Let $\lambda,r ,\nu>0$. For any $u\in B^r_{\nu,\lambda}({\mathbb R}^+)
$ with $0\le s\le r\le N,$ we have
\begin{equation}\label{Thm21}
\big\|\TD{x}{\infty}{s} \big\{\Pi_N^{\nu,\lambda} u-u\big\}\big\|_{\omega^{\nu+s}}\leq ~c~({2\lambda}{N})^{\frac{s-r}{2}} ~\|\TD{x}{\infty}{r}u\|_{\omega^{\nu+r}},
\end{equation}
where $c\approx 1$ for large $N.$
\end{thm}
\begin{proof}
Note that by definition,
\begin{equation*}
 u-\Pi_N^{\nu,\lambda}u =\sum_{n=N+1}^\infty ~\hat{u}_n  {\mathcal L}_n^{(\nu,\lambda)}(x),\quad \hat{u}_n=\big(u,~ {\mathcal L}_n^{(\nu,\lambda)}\big)_{\omega^{\nu}}\big/\gamma^{\nu,\lambda}_n.
 \end{equation*}
Then by   \eqref{TempGLderi2II}, and the orthgogonality,
\begin{equation*}
 \|\TD{x}{\infty}{s}\GLa{\nu}{n}{}\|^2_{\omega^{\nu+s}}=\|\GLa{\nu+s}{n}{}\|^2_{\omega^{\nu+s}}=\gamma_{n}^{\nu+s,\lambda},
\end{equation*}
we can  derive
\begin{equation*}\begin{aligned}
\big|\Pi_N^{\nu,\lambda}u-u\big|^2_{B^s_{\nu,\lambda}}&=\|\sum_{n=N+1}^\infty \hat{u}_n (2\lambda)^s\GLa{\nu+s}{n}{} \|^2_{\omega^{\nu+s}}=\sum_{n=N+1}^\infty(\hat{u}_n)^2(2\lambda)^{2s}\gamma_{n}^{\nu+s,\lambda},
\\
\big|u\big|^2_{B^r_{\nu+r,\lambda}}&=\|\sum_{n=0}^\infty \hat{u}_n(2\lambda)^r  \GLa{\nu+r}{n} \|^2_{\omega^{\nu+r}}=\sum_{n=0}^\infty(\hat{u}_n)^2(2\lambda)^{2r}\gamma_{n}^{\nu+r,\lambda}.
\end{aligned}\end{equation*}
Then,
 \begin{equation*}\begin{aligned}
\big|\Pi_N^{\nu,\lambda}u-u\big|^2_{B^s_{\nu,\lambda}}&=\sum_{n=N+1}^\infty(\hat{u}_n)^2\gamma_{n}^{\nu+s,\lambda}
\leq (2\lambda)^{2s-2r}\frac{\gamma_{N+1}^{\nu+s,\lambda}}{\gamma_{N+1}^{\nu+r,\lambda}}\sum_{n=N+1}^\infty(\hat{u}_n)^2\gamma_{n}^{\nu+r,\lambda},
\end{aligned}\end{equation*}
where by \eqref{Gammaratio}-\eqref{ConstUpsilon}, we obtain
$$\dfrac{\gamma_{N+1}^{\nu+s,\lambda}}{\gamma_{N+1}^{\nu+r,\lambda}}=\dfrac{(2\lambda)^r \Gamma(N+\nu+s+2)}{(2\lambda)^s \Gamma(N+\nu+r+2)}\leq c ({2\lambda})^{r-s}N^{s-r}.$$
 Consequently, we have
\begin{equation*}
%\|\Pi_N^{\nu,\lambda} u-u\|_{B^s_{\nu,\lambda}}
\|\TD{x}{\infty}{s} \big\{\Pi_N^{\nu,\lambda} u-u\big\}\|_{\omega^{\nu+s}}\leq ~c\, ({2\lambda}N)^{\frac{s-r}{2}} \,|u|_{B^r_{\nu,\lambda}}.
\end{equation*}
This ends the proof.
\end{proof}

\subsection{A model problem and numerical results} \label{sub33}
%In order to figure out the role of the tempered fractional derivative in the practical problem,  below model problem is considered
In what follows, we consider the GLF approximation to  a model tempered fractional equation of order
$s\in[k-1,k)$ with $k\in\mathbb{N}:$
 \begin{equation}\label{problemoriginal}
\TD{0}{x}{s}u(x)=f(x), \;\; x\in \mathbb{R}^+,\;\; \lambda>0; \;\;   u^{(j)}(0)=0,\quad j=0,1,\ldots,k-1, %f\in L^2(\Lambda)\cap
\end{equation}
where $f\in L^2(\mathbb{R}^+)$ is a given function.
%It's beneficial to investigate the solution's regularity before the numerical implementation.
Using the fractional derivative relation \eqref{rulesa}, one can  find %\eqref{problemoriginal},
\begin{equation*}
u(x)=\TI{0}{x}{s}f(x)+\sum\limits_{i=1}^kc_i\, x^{s-i}e^{-\lambda x},
\end{equation*}
where $\{c_i\}$ can be determined by the conditions at $x=0.$
%%Moreover, if $f(x)$ is analytic in the vicinity of the zero, then
%Furthermore, due to the homogeneous initial values such that $c_i=0,~i=1,\ldots,k$ and
In fact, we have all $c_i=0,$ and
\begin{equation}\label{SolutionExact2}
u(x)=\TI{0}{x}{s}f(x)=\frac{e^{-\lambda x}}{\Gamma(s)} \int_0^x (x-\tau)^{s-1} e^{\lambda \tau} f(\tau){\rm d}\tau= \frac{x^s}{\Gamma(s)}\int_0^1 (1-t)^{s-1} e^{-\lambda (1-t)x}f(xt) {\rm d}t.
\end{equation}
We see that if $f(x)$ is smooth, then $u(x)=x^s F(x),$ where $F(x)$ is  smoother than $f(x).$
With this understanding, we construct the GLF Petrov-Galerkin  approximation as:  find $u_N\in\mathcal{F}^{s,\lambda}_N(\mathbb{R}^+)$ (defined in \eqref{SpaceGLF1}) such that
\begin{equation}\label{GP-initial}
(\TD{0}{x}{s} u_N,v_N)=(f,v_N), \quad \forall\, v_N\in \mathcal{F}^{0,\lambda}_N(\mathbb{R}^+).
\end{equation}
%Now we analyze the efficiency of the numerical scheme \eqref{GP-initial}. Firstly, we write the power of function $f$ and its projection $\pi^{0,\lambda}_Nf$ by GLF as follows
We expand $f$ and $u_N$ as
\begin{equation}\label{expand}
 f(x)=\sum_{n=0}^\infty ~\hat{f}_n \GLa{0}{n}{(x)}, ~\quad u_N=\sum_{n=0}^N ~\hat{u}_n \GLa{s}{n}{(x)} .
  %\pi^{0,\lambda}_Nf(x)=\sum_{n=0}^N ~\hat{f}_n \GLa{0}{n}{(x)}.
 \end{equation}
 Using the derivative relation \eqref{TempGLderi2}, we find immediately that $\hat u_n=\frac{\gamma_n^{0,\lambda}}{h_n^{s,s}}\hat f_n$ for $n=0,1,\cdots,N$,
 %From  \eqref{problemoriginal} and \eqref{GP-initial}, we obtain $$(\TD{0}{x}{s} u_N-f,v_N)=(\TD{0}{x}{s} u_N-\TD{0}{x}{s} u,v_N)=0,\quad\forall~ v_N\in\mathcal{F}^{0,\lambda}_N(\mathbb{R}^+),$$
which also implies $\TD{0}{x}{s} u_N=\pi^{0,\lambda}_N f.$
%\begin{equation*}
%\TD{0}{x}{s} u_N=\pi^{0,\lambda}_N f,\quad\|\TD{0}{x}{s} (u_N-u)\|=\|\pi^{0,\lambda}_N f-f\|.
%\end{equation*}

Moreover, we can show that the numerical solution $u_N$  is precisely the orthogonal projection in the following sense:
\begin{equation}\label{unprojn}
( u_N-u\,,\,w_N)_{\omega^{-s}}=0,\quad\forall\, w_N\in\mathcal{F}^{s,\lambda}_N(\mathbb{R}^+).
\end{equation}
To this end, we first show
\begin{equation}\label{forAppendixproof-a}
( u_N-u\,,\,\TD{x}{\infty}{s}v_N)=(\TD{0}{x}{s} u_N-\TD{0}{x}{s} u,v_N)=0,\quad\forall\, v_N\in\mathcal{F}^{0,\lambda}_N(\mathbb{R}^+).
\end{equation}
Indeed, thanks to $u^{(j)}(0)=0$ for $j=0,\ldots,k-1$, we have
\begin{equation*}
\TD{0}{x}{s}\{u_N-u\}=e^{-\lambda x}\D{0}{x}{s} \{e^{\lambda x} (u_N-u)\}=e^{-\lambda x}\I{0}{x}{k-s}\D{0}{x}{k} \{e^{\lambda x} (u_N-u)\}.
\end{equation*}
Then,
\begin{equation*}\begin{aligned}
&\left(\TD{0}{x}{s} u_N-\TD{0}{x}{s} u\,,\,v_N\right)=\left(\I{0}{x}{k-s}\D{0}{x}{k} \{e^{\lambda x} (u_N-u)\},e^{-\lambda x}v_N\right)
%\\& =\left(D^{k} \{e^{\lambda x} (u_N-u)\} \,,\,\TI{x}{\infty}{k-s}v_N\right)
\\& \qquad=\left( e^{\lambda x} (u_N-u)\,,\,(-1)^k{\rm D}^{k}\TI{x}{\infty}{k-s}v_N\right)
=\left(u_N-u\,,\,\TD{x}{\infty}{s}v_N\right),
\end{aligned}\end{equation*}
so  \eqref{forAppendixproof-a} is valid.
In addition, thanks to  Lemma \ref{Lemmalagupolyderi} and  \eqref{transform2 xt}, we have
 %\comm{\color{red} The breve notation is not defined yet! Please check it!}
$$\TD{x}{\infty}{s}\GLa{0}{n}{(x)}=e^{\lambda x}\D{x}{\infty}{s}\big\{e^{-2\lambda x}L^{(0)}_{n}(2\lambda{x})\big\}=(2\lambda)^s e^{-\lambda x}L^{(s)}_{n}(2\lambda{x})=(2\lambda)^sx^{-s}\GLa{-s}{n}{(x)}.$$
Hence, \eqref{unprojn} is valid.

%The numerical solution is equivalent to projection $\pi_N^{-s,\lambda} u$ and has the following estimate,
Thanks to  \eqref{unprojn}, we derive from  Theorem \ref{ThmAppI} the following estimate where the convergence rate only  depends  on the regularity of the source term.
  \begin{thm}\label{ThmAppRight} Let $u$ and $u_N$ be respectively the solutions of \eqref{problemoriginal} and \eqref{GP-initial}. Then for  $\TD{0}{x}{m}f\in L^2_{\omega^m}(I)$ with  $m\in {\mathbb N}_0,$ we have   %$u\in A^m_{s,\lambda}(\mathbb{R}^+)$, then
\begin{equation}\label{Thm11}
\|u-u_N\|_{\omega^{-s}}\leq \,c\,(2\lambda N)^{-\frac{s+m}{2}} \,\|\TD{0}{x}{s+m}u\|_{\omega^{m}}=c\,(2\lambda N)^{-\frac{s+m}{2}}\,\|\TD{0}{x}{m}f\|_{\omega^{m}},%~|u|_{A^m_{\alpha,\lambda}}.
\end{equation}
where $c\approx 1$ for large $N.$
  \end{thm}
%\comm{\color{red} Need more description of the results. Reference solutions by large $N$ or also copy some from our previous part to explain the result!}

We provide some numerical results to illustrate the convergence behaviour.    We take $f(x)= e^{-x}\sin x$ and then evaluate the exact solution by \eqref{SolutionExact2}.  Note that as   $\TD{0}{x}{m}f=e^{-\lambda x} \D{}{}{m}\{e^{\lambda x}f\}$,  a direct calculation leads to
\begin{equation*}
e^{-\lambda x}\D{}{}{m}\{e^{\lambda x} f\}=e^{-\lambda x}\sum_{k=0}^m {m \choose{k}} \lambda^{m-k}e^{\lambda x} \D{}{}{k}f=\sum_{k=0}^m {m \choose{k}} \lambda^{m-k} \D{}{}{k}f.
\end{equation*}
We infer from  \eqref{Thm11} that the spectral accuracy can be achieved  by the GLF approximation.  Indeed, we observe from Figure \ref{InitialP} such a convergence behaviour.  %Also, we inspect that the accuracy might slight

%As  a  numerical test with smooth data $f(x)=e^{-x}\sin x$ in Figure \ref{InitialP} validates the error analysis. The left in this figure shows that  the tempered model's  convergence rate depends on the tempered parameter $\lambda$. The right is designed to exhibit the GLF is a perfect choice for solving problem \eqref{problemoriginal}.
%
%
%In view of  \eqref{SolutionExact2}, we have that  $u=x^sF(x)$ is singular at $x=0$ even though $f$ is a smooth function.  The GLF Petrov-Galerkin  method \eqref{GP-initial} is devoted to  efficiently  solve this problem. The above Theorem showed it's an  efficient method and  we get information that the convergence rate of the numerical method determined by the regularity of $f$. Indeed,  as the tempered derivative of integer order

\begin{figure}[htp]
\begin{minipage}{0.45\linewidth}
\begin{center}
\includegraphics[scale=0.370]{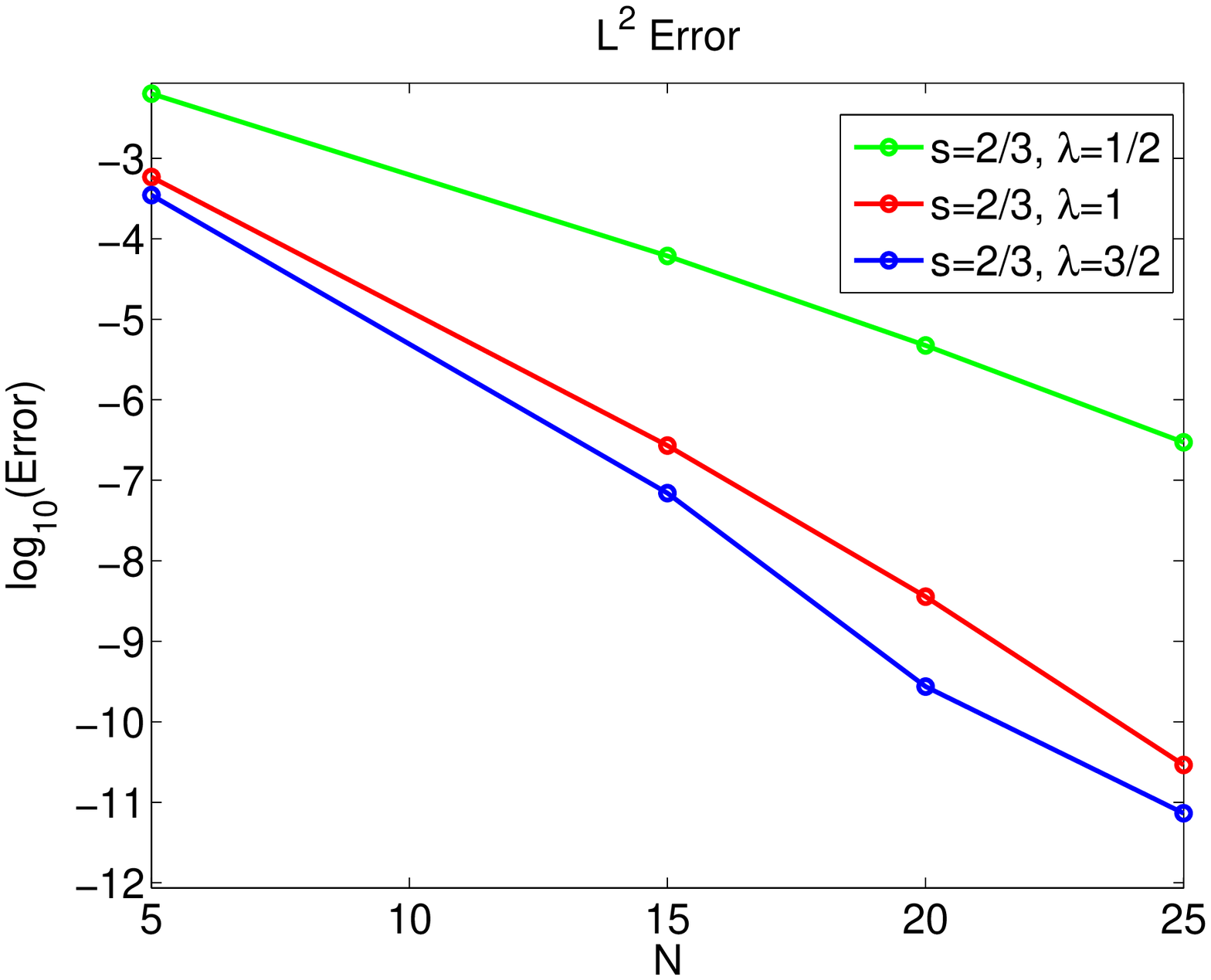}\end{center}
\end{minipage}
\begin{minipage}{0.45\linewidth}
\begin{center}
\includegraphics[scale=0.370]{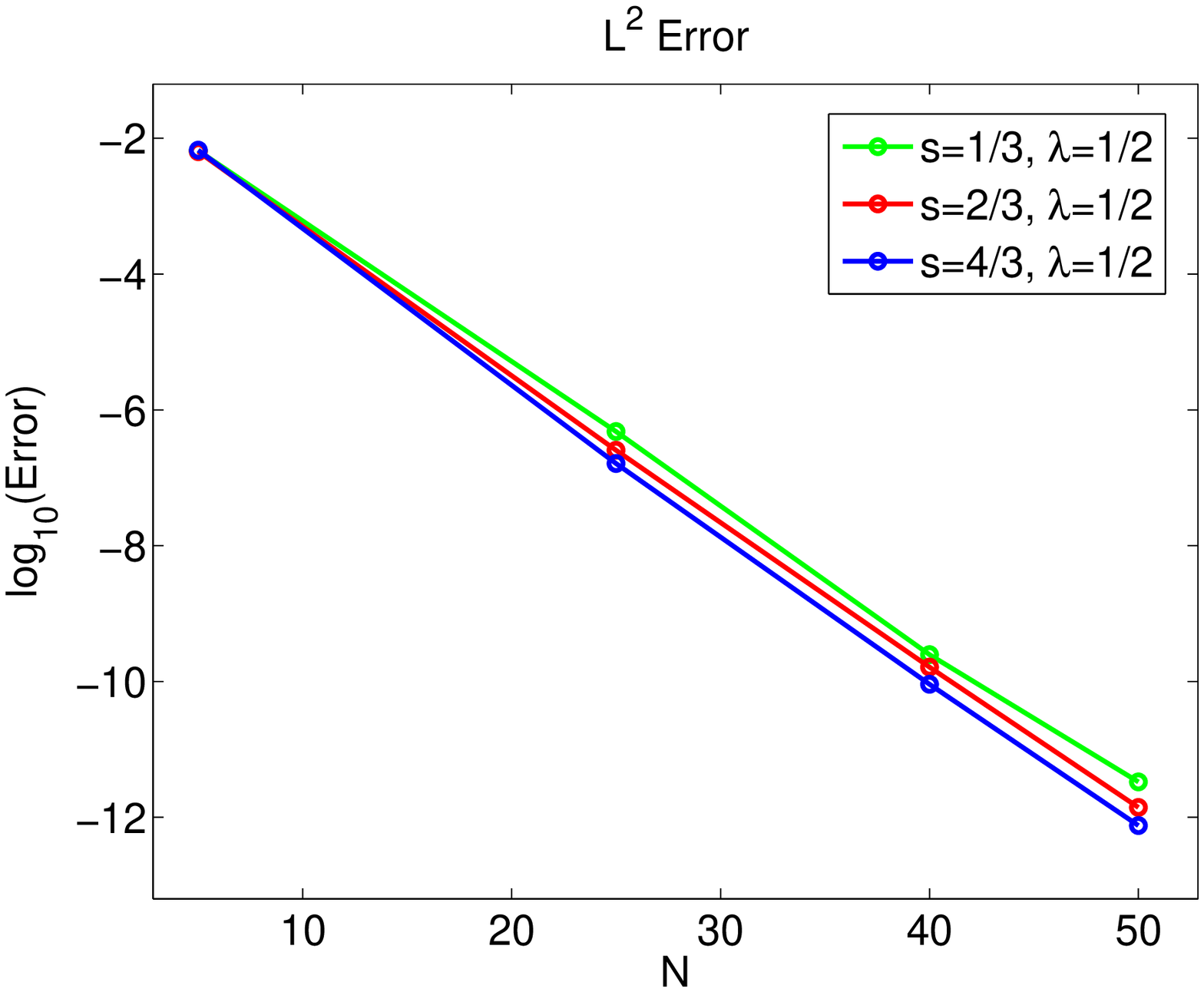}\end{center}
\end{minipage}
\caption{\small Convergence of the  GLF approximation to (\ref{problemoriginal}) with $f(x)=e^{-x} \sin x\,.$ }\label{InitialP}
\end{figure}

\section{Application to  Tempered fractional diffusion equation on the half line}\label{SectionApp2}
%In this section, we concentrate on the numerical approximation to the Tempered fractional diffusion model proposed by Sabzikar, Meerschaert and Chen \cite{Temperedsabzikar2014}  which can be described by the below equation,
%\begin{equation}\begin{split}\label{tempfracdiffuequa}
%\partial_t u=(-1)^k\big\{p\partial_{+}^{\mu,\lambda}+q\partial_{-}^{\mu,\lambda}\big\}u,\quad \mu\in{(k-1,k)},k=1,2.
%\end{split}\end{equation}
%%\begin{equation}\begin{split}
%%\partial_t u=\big(p\TD{-\infty}{x}{\mu}+q\TD{x}{\infty}{\mu}\big) u-\lambda^\mu u, \quad 0<\mu<1,\\
%%\partial_t u=\big(p\TD{-\infty}{x}{\mu}+q\TD{x}{\infty}{\mu}\big) u+(q-p)\mu \lambda^{\mu-1}\partial_x u-\lambda^\mu u,\quad 1<\mu<2,
%%\end{split}\end{equation}
%where $0\leq p,q\leq1,~p+q=1,$  and
% \begin{equation}\label{temprelamu}
%\partial_{\pm}^{\mu,\lambda}u=\begin{cases}
%\TD{}{\pm}{\mu}u-\lambda^\mu u,& 0<\mu<1,\\
%\TD{}{\pm}{\mu}u-\pm\mu\lambda^{\mu-1}\partial_x u-\lambda^\mu u,& 1<\mu<2.
%\end{cases}
%\end{equation}
 %For any real $\mu\in[k-1,k)$ with positive integer $k$, the Tempered fractional derivatives  $\TD{}{\pm}{\mu}u$ are respectively defined by
%\begin{equation}\begin{split}\label{temperderi}
%\TD{-\infty}{x}{\mu}u:= \frac{e^{-\lambda x}}{\Gamma(k-\mu)}\frac{d^k}{dx^k}\int_{-\infty}^x \frac{e^{\lambda \tau}u(\tau)}{(x-\tau)^{\mu-k+1}}{\rm d}\tau,\quad \TD{x}{\infty}{\mu}u:= \frac{e^{\lambda x}}{\Gamma(k-\mu)}\frac{d^k}{dx^k}\int^{\infty}_x \frac{e^{-\lambda \tau}u(\tau)}{(\tau-x)^{\mu-k+1}} {\rm d}\tau.
%\end{split}\end{equation}

In this section,  we apply the GLFs to approximate a tempered fractional diffusion equation on the half-line.

\subsection{The tempered fractional diffusion equation on the half line}\label{subsectionspecial}
%\subsubsection{The half line tempered fractional diffusion equation}

Consider the tempered fractional diffusion equation of order $\mu\in (0,1)$ on the half line:
\begin{equation}\begin{split}\label{temperFDEhalf}
\begin{cases}
\partial_t u(x,t)+\TD{0}{x}{\mu}u(x,t)-\lambda^\mu u(x,t)=f(x,t),\quad &(x,t)\in\mathbb{R}^+\times(0,T],
\\ u(0,t)=0,\quad  \lim\limits_{x\rightarrow \infty}u(x,t)=0,\quad & 0<t\leq T,
\\u(x,0)=u_0(x), \quad &x\in\mathbb{R}^+.
\end{cases}
\end{split}\end{equation}
%This half-line tempered fractional diffusion can be regarded as  a special case of the original model \eqref{temperFDEoriginal},  in which we take $p=1,~q=0$ and $0<\mu<1$. On the other hand, it

This equation models the particles jumping on the half line $\mathbb{R}^+$ with the probability density function (see \cite[(8)]{temperedsabzikar2014}):
\begin{equation*}
f_{\varepsilon}(x)=C^{-1}_{\varepsilon}x^{-\mu-1}e^{-\lambda x}\mathbf{1}_{({\varepsilon,\infty})}(x),\quad 0<\mu<1.
\end{equation*}

\begin{rem}
We note that equation \eqref{temperFDEhalf} is equivalent to  the half line form of the TFDE \eqref{temperFDEoriginal}  in which
$$\partial_{+,x}^{\mu,\lambda} u=\TD{0}{x}{\mu}u-\lambda^\mu u,\quad 0<\mu<1. $$
Indeed, we can show that for $\mu\in (0,1)$ and real $\lambda>0$,
$$\TD{0}{x}{\mu}u=e^{-\lambda x}\D{0}{x}{\mu}\big\{e^{\lambda x}{{u}(x)}\big\}=e^{-\lambda x}\D{-\infty}{x}{\mu}\big\{e^{\lambda x}{\tilde{u}(x)}\big\},\quad x\in\mathbb{R}^+,$$
where $\tilde u=u$ for  $x\in \mathbb{R}^+$ and $\tilde u=0$ for $x\in (-\infty,0)$.  Moreover, we have
 \begin{equation*}\begin{split}
{\mathscr F}\big[e^{-\lambda x}\D{-\infty}{x}{\mu}\big\{e^{\lambda x}{\tilde{u}(x)}\big\}\big](\omega)=&\int_{\mathbb{R}} \D{}{}{}\I{-\infty}{x}{1-\mu}\big\{e^{\lambda x}{\tilde{u}(x)}\big\}~e^{-(\lambda+\ri \omega) x}{\rm d} x
\\=(\lambda+\ri \omega) &\int_{\mathbb{R}} \I{-\infty}{x}{1-\mu}\big\{e^{\lambda x}{\tilde{u}(x)}\big\}~e^{-(\lambda+\ri \omega) x}{\rm d} x
\\{=}(\lambda+\ri \omega)& \int_{\mathbb{R}}e^{\lambda x}{\tilde{u}(x)}~\I{x}{\infty}{1-\mu}e^{-(\lambda+\ri \omega) x}{\rm d} x
%\\&{=}(\lambda+\ri \omega)^{\mu} \int_{\mathbb{R}}{\tilde{u}(x)}e^{-\ri \omega x}{\rm d} x
\\=(\lambda+\ri \omega)^\mu& {\mathscr F}[\tilde {u}](\omega)
\overset{\eqref{Temdef2FourierB}}{=}{\mathscr F}\big[\TD{-\infty}{x}{\mu}{{\tilde{u}(x)}}\big](\omega).
\end{split}\end{equation*}
This implies  $\tilde{u}\in W^{\mu,2}_\lambda(\mathbb{R})$  and  the extended tempered fractional derivative $\TD{0}{x}{\mu}u$  can be understood
in the sense of the original definition in  \cite{temperedsabzikar2014}.  \qed
\end{rem}
%\comm{\color{red} Need to improve!}

\subsection{Spectral-Galerkin  scheme} Define
  $$H^{\mu,\lambda}(\mathbb{R}^+):=\big\{v\in L^2(\mathbb{R}^+):~\TD{0}{x}{s}v\in L^2(\mathbb{R}^+),~0<s\leq \mu \big\},\;\;\; s,\lambda>0,\;\; \mu\in (0,1), $$
  with the semi norm and norm
$$|v|_{\mu,\lambda}=\|\TD{0}{x}{\mu}v\|,\quad\|v\|_{\mu,\lambda}=\big(\|v\|^2+ |v|^2_{\mu,\lambda}\big)^{1/2}.$$
Furthermore, let
$H^{\mu,\lambda}_{0}(\mathbb{R}^+)$ be the closure of $C_0^\infty(\mathbb{R}^+)$ with respect to the norm $\|\cdot\|_{\mu,\lambda}$.

Thanks to the homogeneous boundary condition and  \eqref{integralbypart},  a weak form of  \eqref{temperFDEhalf} is  to find $u(\cdot,t)\in H^{\mu,\lambda}_{0}(\mathbb{R}^+)$
%$u\in L^2\big(0,T;  H^{\mu,\lambda}_{0}(\mathbb{R}^+)\big)\bigcap H^1\big(0,T;L^2(\mathbb{R}^+)\big)$
 such that
 \begin{equation}\label{weakformhalf}
% \begin{cases}
 \big(\partial_tu(\cdot, t),v\big)+a_\mu\big(u(\cdot, t),v\big)=\big(f(\cdot, t),v\big), \quad \forall v\in H^{\mu,\lambda}_{0}(\mathbb{R}^+),\\
 %\end{cases}
 \end{equation}
 with $ u(x,0)=u_0(x),$ where
 \begin{equation}\begin{split}\label{Apq}
 a_\mu(u,v):&=(\TD{0}{x}{\mu}u,v)-\lambda^\mu(u,v).
 \end{split}\end{equation}

%As the discussion in subsection \ref{sub33},  the GLF-I $\GLa{-\mu}{n}{(x)}$ is an appropriate basis candidate and the related Petrov-Galerkin scheme  is suitable for approximating the solution.  But here we choose  the GLF-I $\GLa{-\nu}{n}{(x)},~\max\{0,\mu-\frac{1}{2}\}<\nu\leq1$ as the basis and its span $\mathcal{F}^{\nu,\lambda}_N({\mathbb R}^+)$ as the approximate space.  The main reason is due to the singularity of the solution near $0$ is complex, we can adjust the parameter $\nu$ to solve the problem efficiently (see Figure \ref{TFDEf02}).  Moreover, the Petrov-Galerkin can not leads to a sparse mass matrix which lose the advantage. Hence, we construct  the semi-discrete Galerkin approximation scheme to solve the problem.
 %It can be read as
 The semi-discrete Galerkin approximation scheme  is  to find
 $u_N(\cdot,t)\in \mathcal{F}^{\nu,\lambda}_N({\mathbb R}^+) $
%$u\in L^2\big(0,T;  H^{\mu,\lambda}_{0}(\mathbb{R}^+)\big)\bigcap H^1\big(0,T;L^2(\mathbb{R}^+)\big)$
 such that
 \begin{equation}\label{Galerkinscheme_half}
% \begin{cases}
 \big(\partial_tu_N(\cdot, t),v\big)+a_\mu\big(u_N(\cdot, t),v\big)=\big(f(\cdot, t),v\big), \quad \forall v\in \mathcal{F}^{\nu,\lambda}_N({\mathbb R}^+),
 %\end{cases}
 \end{equation}
 with
 $$u_N(x,0)=u_{0,N}(x)=\sum_{n=0}^{N}c_{0,n}\,\GLa{-\nu}{n}{(x)}.$$
 Here, we choose $\max\big\{0,\mu-\frac{1}{2}\big\}<\nu\leq1$  so that $u(0,t)=0$.% and $|(\TD{0}{x}{\mu}u_N,v)|<\infty$.

 Now,  we  set
% \begin{equation}\label{numerexpan}
% u_N(x,t)=\sum_{n=0}^{N}c_n(t)\,\GLa{-\nu}{n}{(x)},\quad  u_{0,N}(x,t)=\sum_{n=0}^{N}c_{0,n}\,\GLa{-\nu}{n}{(x)}.
% \end{equation}
 \begin{equation}\label{numerexpan}
 u_N(x,t)=\sum_{n=0}^{N}c_n(t)\, \varphi_{n}({x}), \quad     \varphi_{n}({x}):=\GLa{-\nu}{n}{(x)}.  %\quad  u_{0,N}(x)=\sum_{n=0}^{N}c_{0,n}\,\varphi_{n}({x}),
 \end{equation}
We derive from the scheme  \eqref{Galerkinscheme_half}  that %produces a  ordinary differential equation system
 \begin{equation}\label{MatrixSystem}
 \mathbf{M}\frac{d}{dt} \vec{\mathbf{c}}(t)+\mathbf{A}\vec{\mathbf{c}}(t)=\vec{\mathbf{f}}(t); \quad \vec{\mathbf{c}}(0)=\vec{\mathbf{c}}_0.
 \end{equation}
 where for fixed $t>0$, vectors
 \begin{equation}\begin{split}
 & \vec{\mathbf{c}}(t)=\big(c_0(t),c_1(t),\ldots,c_N(t)\big)^T,\quad\vec{\mathbf{{c}}}_0=\big(c_{0,0}(t),c_{0,1}(t),\ldots c_{0,N}(t)\big)^T,
 \\ &\vec{\mathbf{f}}(t)=\big(f_0(t),f_1(t),\ldots f_N(t)\big)^T,\qquad f_n(t)=(f,\varphi_{n}),\quad{0\leq n\leq N}.
 \end{split}\end{equation}
 and
 \begin{equation}\begin{aligned}
 \mathbf{M}_{mn}&=(\varphi_{n},\varphi_{m}),\quad \mathbf{A}_{mn}&=a_\mu(\varphi_{n},\varphi_{m}), \quad m,n=0,1,2,\ldots,N.
 \end{aligned}\end{equation}
%\begin{equation}\label{SpaceGLF1}
%\mathcal{F}^{\nu,\lambda}_N({\mathbb R}^+)=\big\{p(x)\,x^{\alpha}e^{-\lambda x}:~p(x)\in\mathcal{P}_N(\mathbb{R}^+)\big\}
%\end{equation}
%
%Before construct the weak formulation , we define the corresponding Tempered type Hilbert spaces, for real $r>0$,
%  $$H^{r,\lambda}_\pm(\mathbb{R}^+):=\big\{v\in L^2(\mathbb{R}^+):~\TD{}{\pm}{s}v\in L^2(\mathbb{R}^+),~0<s\leq r \big\}$$ with the semi norm and norm
%$$|v|_{r,\pm}=\|\TD{}{\pm}{r}v\|,\quad\|v\|_{r,\pm}=\big(\|v\|^2+ |v|^2_{r,\pm}\big)^{1/2}.$$
%Furthermore, define
%$$H^{r,\lambda}_{0,\pm}(\mathbb{R}^+):=\big\{v \big\}$$
%Then, the weak formulation can be read as:
%find $u\in L^2\big(0,T;  H^{\mu,\lambda}_{0,-}(\mathbb{R}^+)\big)\bigcap H^1\big(0,T;L^2(\mathbb{R}^+)\big)$
% such that
% \begin{equation}\label{weakform}
% \begin{cases}
% \big(\partial_tu(t),v\big)+a^\mu_{pq}\big(u(t),v\big)=\big(f(t),v\big), \quad \forall v\in H^{\mu,\lambda}_{0,-}(\mathbb{R}^+),\\
% u(0)=u_0(x).
% \end{cases}
% \end{equation}
% The bilinear form $a^\mu_{pq}(\cdot,\cdot),~ 0<\mu<1$ is defined by
% \begin{equation}\begin{split}\label{Apq}
% a^\mu_{pq}(u,v):&=-\big\{p(\TD{0}{x}{\mu}u,v)+q(\TD{x}{\infty}{\mu}u,v)\big\} +\lambda^\mu(u,v),\\
%                 &=-\big\{p(\TD{0}{x}{\mu}u,v)+q(u,\TD{0}{x}{\mu}v)\big\} +\lambda^\mu(u,v).
% \end{split}\end{equation}
\subsection{Numerical results} For clarity, we test three cases:

(i). $u(x,t)= xe^{-\lambda x}\cos{(t)}$. By a direct calculation, the source term is given by
%In order to validate the numerical scheme, we take the exact solution of  \eqref{temperFDEhalf} to be, so the right hand side function is given by
 $$f(x,t)=-xe^{-\lambda x}\sin(t)+\big(\dfrac{\Gamma(2)}{\Gamma(2-\mu)}x^{1-\mu}-\lambda^\mu x\big)e^{-\lambda x}\cos{(t)}.$$
%$$f(x,t)=-xe^{-\lambda x}\sin(t)+\big[p\dfrac{\Gamma(2)}{\Gamma(2-\mu)}x^{1-\mu}+q(2\lambda)^\mu(x-\frac{\mu}{2\lambda})-\lambda^\mu x\big]e^{-\lambda x}\cos{(t)}.$$
%Due to choosing the smooth function as the solution of the tempered fractional diffusion equation with the resource $f(x,t)$, \textbf{GLFs}-II $\GLa{-1}{n}{x}=x L^{(1)}_n(2\lambda x) e^{-\lambda x}$ is the best candidate for numerically solving the problem \eqref{weakform}.
The left of Figure \ref{TFDEuf} illustrates that error decays to zero dramatically when using the spectral method with basis $\GLa{-1}{n}{(x)}$ in the space and the third-order explicit Runge-Kutta method in time direction for $\lambda=2/3,~\mu=2/3$.

(ii). $f(x,t)=\cos(x)e^{-x}\sin(t)$ and fix  $\lambda$ and $\mu$ as before, the right graph verifies  that the solution is singular even though  $f(x,t)$ is a smooth function. %\begin{equation}\begin{split}
%\TD{x}{\infty}{\mu}(xe^{-\lambda x})&=\TD{x}{\infty}{\mu}\{\GLb{0}{0}{x}-\GLb{0}{1}{x}\}\\
%&=(2\lambda)^\mu\{\GLb{\mu}{0}{x}-\GLb{\mu}{1}{x}\}\\
%&=(2\lambda)^\mu(x-\frac{\mu}{2\lambda})e^{-\lambda x}
%\end{split}\end{equation}
\begin{figure}[htp]
\begin{minipage}{0.48\linewidth}
\begin{center}
\includegraphics[scale=0.370]{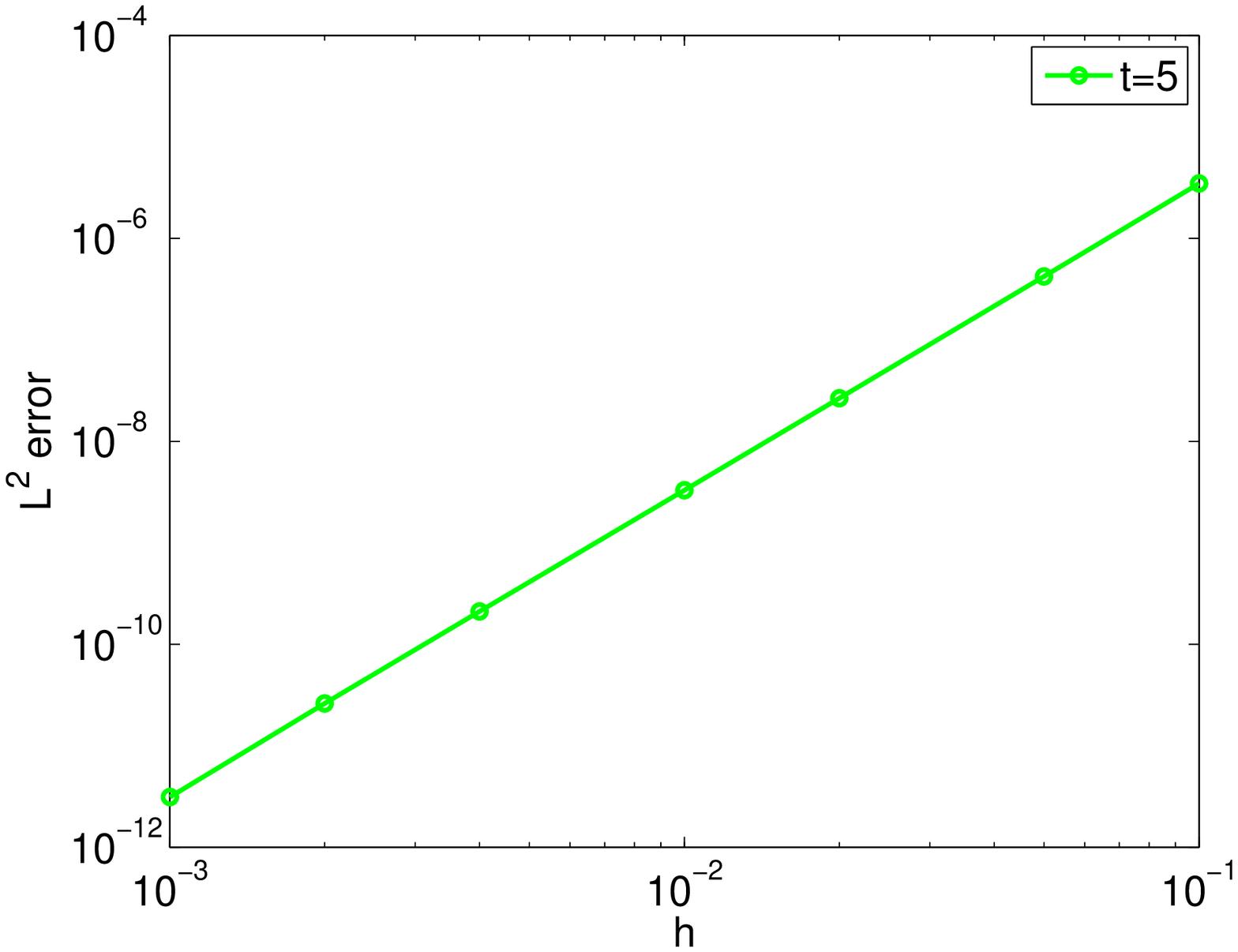}\end{center}
\end{minipage}
\begin{minipage}{0.48\linewidth}
\begin{center}
\includegraphics[scale=0.370]{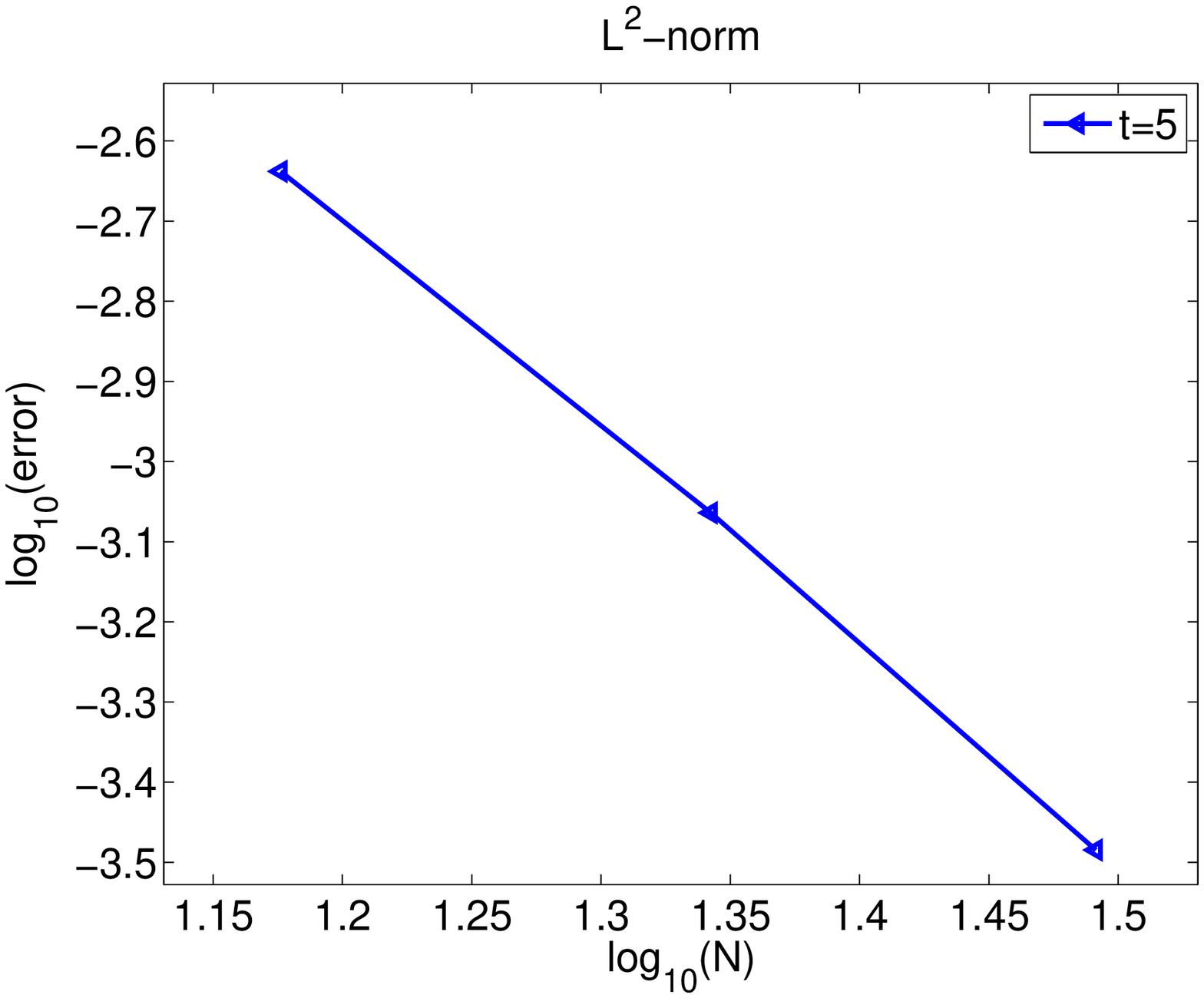}\end{center}
\end{minipage}
\caption{Left: $u=x\exp(-\lambda x)\cos(t)$.\quad Right: $f=\cos(x)\exp(-x)\sin(t)$. }\label{TFDEuf}
\end{figure}

(iii). Consider the case $f(x,t)\equiv0$. Let $\mu=2/3,~\lambda=2/3$ in \eqref{temperFDEhalf}. The left of the Figure \ref{TFDEf02} exhibits the evolution of the tempered fractional diffusion model with the initial distribution $u_0(x)=xe^{-x}$. The right describes the approximate rate  by diverse basis $\GLa{-\nu}{n}{(x)}$ at time $t=10$.
%\begin{figure}[htp]
%\begin{minipage}{0.48\linewidth}
%\begin{center}
%\includegraphics[scale=0.370]{TFDEf0uNp2}\end{center}
%\end{minipage}
%\begin{minipage}{0.48\linewidth}
%\begin{center}
%\includegraphics[scale=0.370]{TFDEf0uNp1}\end{center}
%\end{minipage}
%\caption{ TFDEs: $f\equiv 0$, \quad $\lambda=2/3,~\mu=2/3.$}\label{TFDEf01}
%\end{figure}
\begin{figure}[htp]
\begin{minipage}{0.48\linewidth}
\begin{center}
\includegraphics[scale=0.370]{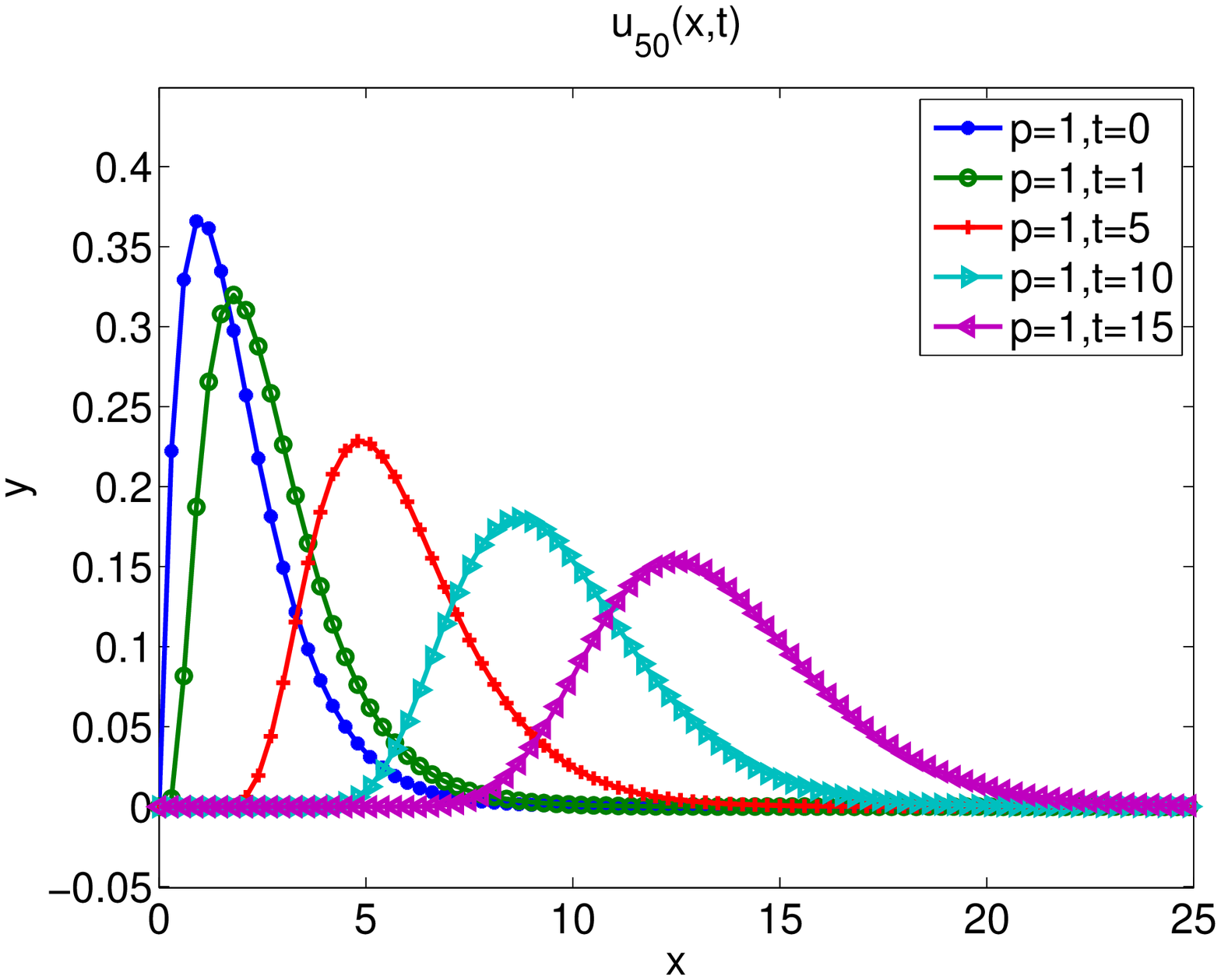}\end{center}
\end{minipage}
\begin{minipage}{0.48\linewidth}
\begin{center}
\includegraphics[scale=0.370]{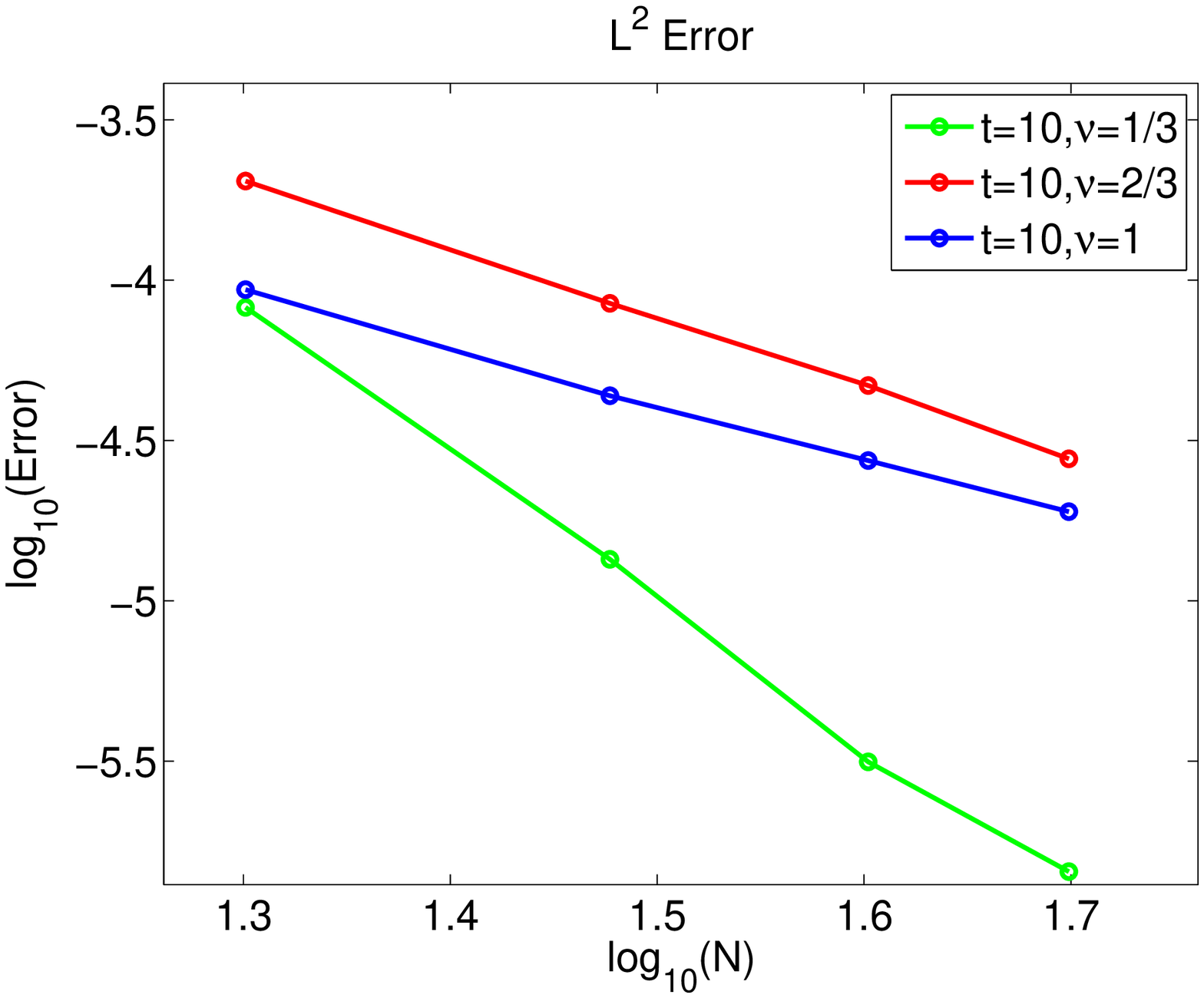}\end{center}
\end{minipage}
\caption{ TFDEs: $f\equiv 0$,  $\lambda=2/3,~\mu=2/3.$}\label{TFDEf02}
\end{figure}

\section{Tempered fractional diffusion equation on  the whole line}\label{SectionTFDE}

In this section, we present a  spectral-element method with two-subdomains for the tempered fractional diffusion equation on the whole line originally proposed by
\cite{temperedsabzikar2014}.
%With a reflection of the GLFs  on ${\mathbb R}^+$ and  the aid of an exponential ``bubble" function (see \eqref{basis}),
%we can develop an efficient spectral-element method for this model.

 \subsection{Tempered fractional diffusion equation}
We consider the tempered fractional diffusion equation % \eqref{temperFDEoriginal}
  of order  $\mu\in(k-1,k),~k=1,2$ on the whole line:
\begin{equation}\begin{cases}\label{tempfracdiffuequa}
\partial_t u(x,t)=(-1)^k C_T\big\{p\partial_{+,x}^{\mu,\lambda}+q\partial_{-,x}^{\mu,\lambda}\big\}u(x,t)+f(x,t),\\
u(x,0)=u_0(x), \quad \lim_{|x|\rightarrow \infty} u(x,t)=0,
\end{cases}\end{equation}
%with a suitable initial distribution $u(x,0)=u_0(x)$, where
where  $p,q$ are constants such that $0\leq p,q\leq1,~p+q=1,$ $C_T$ is a constant and the fractional operators are
\begin{itemize}
\item For $0<\mu<1,$
\begin{equation}\begin{aligned}\label{temprelamu1}
\partial_{+,x}^{\mu,\lambda}u=\TD{-\infty}{x}{\mu}u-\lambda^\mu u,\quad\partial_{-,x}^{\mu,\lambda}u=\TD{x}{\infty}{\mu}u-\lambda^\mu u;
\end{aligned}\end{equation}
 \item For $1<\mu<2$,
 \begin{equation}\label{temprelamu2}
 \partial_{+,x}^{\mu,\lambda}u=\TD{-\infty}{x}{\mu}u-\mu\lambda^{\mu-1}\partial_x u-\lambda^\mu u,\quad
\partial_{-,x}^{\mu,\lambda}u=\TD{x}{\infty}{\mu}u+\mu\lambda^{\mu-1}\partial_x u-\lambda^\mu u.
\end{equation}
\end{itemize}
  \subsection{A two-domain spectral-element method}
Let $\Lambda:=(a,b), -\infty\leq a<b\leq\infty$, and $\omega>0$ be a generic  weight function. For any $m\in\mathbb{N}$ and a given weight function $\omega$, we denote
$$H_\omega^m(\Lambda):=\big\{v\in L_\omega^2(\Lambda):~\partial_x^kv\in L_\omega^2(\Lambda),~0<k\leq m \big\}$$ with the semi-norm and norm
$$|v|_{m,\omega,\Lambda}=\|\partial_x^m v\|_{\omega,\Lambda},\quad\|v\|_{m,\omega,\Lambda}=\Big(\sum\limits_{k=0}^m |v|^2_{k,\omega,\Lambda}\Big)^{1/2}.$$
In particular, we omit the subscript $\omega$ when  $\omega\equiv1$.

Moreover, for real $r>0$, we define
$$H^{r,\lambda}(\mathbb{R}):=\big\{v\in L^2(\mathbb{R}):~\TD{-\infty}{x}{r}v\in L^2(\mathbb{R})\big\}$$ with the semi norm and norm
$$|v|_{r,\lambda}=\|\TD{-\infty}{x}{r}v\|,\quad\|v\|_{r,\lambda}=\big(\|v\|^2+ |v|^2_{r,\lambda}\big)^{1/2}.$$

\begin{comment}
%According to the parameter $\mu$, we install  \eqref{temprelamu1} and \eqref{temprelamu2} into \eqref{tempfracdiffuequa},
A weak formulation of \eqref{tempfracdiffuequa}  is to
find $u\in L^2\big(0,T; H^1(\mathbb{R})\cap H^{\mu-k+1,\lambda}(\mathbb{R})\big)\cap H^1\big(0,T;L^2(\mathbb{R})\big)$
 such that
 \begin{equation}\label{weakform}
 \begin{cases}
 \big(\partial_tu(\cdot,t),v\big)+a^\mu_{pq}\big(u(\cdot,t),v\big)=\big(f(\cdot,t),v\big), \quad \forall v\in H^{1}(\mathbb{R})\cap H^{\mu-k+1,\lambda}(\mathbb{R}),\\
 u(x,0)=u_0(x),
 \end{cases}
 \end{equation}
 \end{comment}

%Due to the infinity of the interval and no basis functions at present be found to efficiently solve the FDEs on whole line, the spectral domain decomposition method may the reasonable way for solving \eqref{weakform}.

We decompose the whole line as follows
$$\mathbb{R}={\Lambda}_1\cup{\Lambda}_2, \quad \Lambda_1=(-\infty,0),\quad \Lambda_2=[0,\infty),$$
and denote $u_{\Lambda_j}(x,t):=u(x,t)\big|_{\Lambda_j},~j=1,2$.  Introduce the approximation space:
 \begin{equation}\label{appsps}
V^{\lambda}_{\bs N}({\mathbb R}) :=\Big\{\phi\in C(\mathbb{R})\,:\,  \phi_{\Lambda_i}(x)=e^{-\lambda |x|} p,\;\; p_{\Lambda_i} \in\mathcal{P}_{N_i}(\Lambda_i),\;\; i=1,2\Big\},\quad \bs N=(N_1,N_2),
 \end{equation}
and define
\begin{equation}\label{basis}
\phi^{\ast}(x)=e^{-\lambda|x|},\;\;
\phi^{-}_{n_1}(x)=\begin{cases}
\GLa{-1}{n_1}{(-x)},&x\leq0,\\[4pt]
0,&x>0,
\end{cases}
\;\;
\phi^{+}_{n_2}(x)=\begin{cases}
0,&x\leq0,\\[4pt]
\GLa{-1}{n_2}{(x)},&x>0,
\end{cases}
\end{equation}
where $\GLa{-1}{n}(x)=e^{-\lambda x}xL_n^{(1)}(2\lambda x)$.   One verifies readily that
 \begin{equation}\label{spansps}
V^{\lambda}_{\bs N}({\mathbb R}) ={\rm span}\big\{\phi^{\ast}(x); \;\;  \phi^{-}_{n_1}(x), \;\; 0\le n_1\le N_1-1; \;\; \phi^{+}_{n_2}(x),\;\;  0\le n_2\le N_2-1\big\}.
 \end{equation}
%Then we can respectively approximate $u_{\Lambda_2}(x,t)$ and $u_{\Lambda_1}(x,t)$ by Laguerre functions and the  reflected Laguerre functions.
%
Then, our semi-discrete spectral-Galerkin method is to find
 $u_{\bs N}(\cdot,t)\in V^{\lambda}_{\bs N}({\mathbb R}) $
%$u\in L^2\big(0,T;  H^{\mu,\lambda}_{0}(\mathbb{R}^+)\big)\bigcap H^1\big(0,T;L^2(\mathbb{R}^+)\big)$
 such that
 \begin{equation}\label{Galerkinscheme}
 \begin{cases}
 \big(\partial_tu_{\bs N}(\cdot, t),v\big)+a^\mu_{pq}\big(u_{\bs N}(\cdot, t),v\big)=\big(f(\cdot,t),v\big), \quad & \forall v\in  V^{\lambda}_{\bs N}({\mathbb R}),\\[4pt]
 \big(u_{\bs N}(\cdot,0),v)=(u_{0},v),\quad & \forall v\in  V^{\lambda}_{\bs N}({\mathbb R}),
 \end{cases}
 \end{equation}
  where the bilinear form $a^\mu_{pq}(\cdot,\cdot)$ is defined by
 \begin{equation}\label{apq}
 a^\mu_{pq}(u,v):=\begin{cases}
\quad p(\TD{-\infty}{x}{\mu}u,v)+q(u,\TD{-\infty}{x}{\mu}v) -\lambda^\mu(u,v),~\qquad & 0<\mu<1,\\
-\big\{p(\TD{-\infty}{x}{\mu-1}u,\TD{x}{\infty}{1}v)+q(\TD{x}{\infty}{1}u,\TD{-\infty}{x}{\mu-1}v)\big\}\\
\qquad\quad\,+\lambda^\mu(u,v)+(p-q)\mu\lambda^{\mu-1}(\partial_xu,v)
,~~\quad & 1<\mu<2.
 \end{cases}
 \end{equation}
% where
% and $u_{N,0}(x)$ is the projection,  from $L^2(\mathbb{R})$ to
% $ \hat{\mathcal{P}}^{\lambda}_N({\mathbb R})$,  of initial distribution $u_0(x)$.
%%\subsubsection{Numerical implementation}

We provide below some details of the algorithm.

%We expand the solution as  % the numerical solution as
 \begin{equation}\begin{aligned}\label{numerexpan}
 u_{\bs N}(x,t)&=c^{\ast}(t)\phi^{\ast}({x})+\sum_{n_1=0}^{N_1-1}c^{-}_{n_1}(t)\phi^{-}_{n_1}({x})+\sum_{n_2=0}^{N_2-1}c^{+}_{n_2}(t)\phi^{+}_{n_2}({x}),\\
 u_{\bs N}(x,0)&=c^{\ast}_0\phi^{\ast}({x})+\sum_{n_1=0}^{N_1-1}c^{-}_{0,n_1}\phi^{-}_{n_1}({x})+\sum_{n_2=0}^{N_2-1}c^{+}_{0,n_2}\phi^{+}_{n_2}({x}).
 \end{aligned}\end{equation}
 %\comm{\color{red}  Need to quote the properties used here for the derivations!}\\ \comm{\bf Have done!}
 Let $H(x)$ be the Heaviside  function as before.
 Thanks to the tempered fractional derivative and integral relations with GLFs, and a reflected mapping from positive half line $\mathbb{R}^+$ to negative half line $\mathbb{R}^-$,  we can derive the following identities (see Appendix \ref{AppenixC}):%, for clarity, the proofs are  listed in the Appendix \ref{AppenixProof}.
 \begin{equation}\label{TDphi}
 \begin{split}
&\TD{x}{\infty}{1}\phi^{\ast}(x)=-2\lambda e^{-\lambda x}H(x),\quad \TD{x}{\infty}{1}\phi^{-}_{n_1}(x)=(n_1+1)\GLa{0}{n_1}{(-x)}H(-x),\\[6pt]
%\end{equation*}
%\begin{equation*}\begin{aligned}
& \TD{-\infty}{x}{s}\phi^{\ast}(x)=
\begin{cases}
(2\lambda )^{s}e^{\lambda x},&x\leq 0,\\[6pt]
\dfrac{2\lambda e^{\lambda x}}{\Gamma(1-s)}\displaystyle \int_x^{\infty}\dfrac{ e^{-2\lambda t}}{t^s} {\rm d}t
%x^{1-s}e^{-\lambda x}\dfrac{2\lambda }{\Gamma(1-s)}\int_0^{\infty}\dfrac{ e^{-2\lambda xt}}{(1+t)^s} {\rm d}t
,&x>0,
\end{cases}\\[6pt]
& \TD{-\infty}{x}{s}\phi^{-}_{n_1}(x)=
\begin{cases}
-(2\lambda)^{s-1}(n_1+1)L^{(s-1)}_{n_1+1}{(-2\lambda x)}e^{\lambda x},&x\leq0,\\[6pt]
-e^{\lambda x}\dfrac{{n_1}+1}{\Gamma(1-s)}\displaystyle\int_x^{\infty}\dfrac{L^{(0)}_{{n_1}+1}{(2\lambda (t-x))}e^{-2\lambda t}}{t^s} {\rm d}t
%-x^{1-s}e^{-\lambda x}\dfrac{n_1+1}{\Gamma(1-s)}\int_0^{\infty}\dfrac{L^{(0)}_{n_1+1}{(2\lambda xt)}e^{-2\lambda xt}}{(1+t)^s} {\rm d}t
,&x>0,
\end{cases}\\[6pt]
&\TD{x}{\infty}{1}\phi^{+}_{n_2}(x)=-(n_2+1)\GLa{0}{n_2+1}{(x)}H(x),\quad \\[6pt]
&\TD{-\infty}{x}{s}\phi^{+}_{n_2}(x)=\dfrac{\Gamma(n_2+2)}{\Gamma(n_2+2-s)}x^{1-s}\GLa{1-s}{n_2}{(x)}H(x).
\end{split}\end{equation}
\noindent Then \eqref{Galerkinscheme} leads to the system %of ordinary differential equation which can be read as
 \begin{equation}\label{MatrixSystem}
 \mathbf{M}\frac{d}{dt} \overrightarrow{{\mathbf{C}}}(t)+\mathbf{A}\overrightarrow{{\mathbf{C}}}(t)=\overrightarrow{{\mathbf{F}}}(t),%,\quad \vec{\mathbf{c}}(0)=\vec{\mathbf{c}}_0.
 \end{equation}
 where
 \begin{equation*}\begin{split}
 &\overrightarrow{{\mathbf{C}}}(t)=\big(c^\ast(t),\overrightarrow{{\mathbf{C}}}^-(t),\overrightarrow{\mathbf{C}}^+(t)\big)^T,\qquad \overrightarrow{{\mathbf{F}}}(t)=\big(f^\ast(t),\overrightarrow{{\mathbf{F}}}^-(t),\overrightarrow{\mathbf{F}}^+(t)\big)^T,\\
% \end{split}\end{equation*}
% \begin{equation*}\begin{split}
 & \overrightarrow{{\mathbf{C}}}^-(t)
 =\big(c^-_0(t),c^-_1(t),\ldots,c^-_{N_1-1}(t)\big)^T,\quad\qquad\overrightarrow{{\mathbf{C}}}^+(t)=\big(c^+_0(t),c^+_1(t),\ldots,c^+_{N_2-1}(t)\big)^T.\\ &\overrightarrow{{\mathbf{F}}}^-(t)
 =\big(f^-_0(t),f^-_1(t),\ldots,f^-_{N_1-1}(t)\big)^T,~\,\,\quad\quad\overrightarrow{{\mathbf{F}}}^+(t)=\big(f^+_0(t),f^+_1(t),\ldots,f^+_{N_2-1}(t)\big)^T.\\
 &f^\ast(t)=(f,\phi^\ast),\quad f^-_{n_1}(t)=(f,\phi^-_{n_1}),\qquad f^+_{n_2}(t)=(f,\phi^+_{n_2}),\quad{0\leq n_i\leq N_i-1},~i=1,2,%\quad f^+_n(t)=(f,\phi^+_{m}),{0\leq m\leq N_2-1}.
 \end{split}\end{equation*}
 and the matrices
 %\begin{equation}\label{Matrixsys}
%\mathbf{M}=\left(
%  \begin{array}{ccc}
%    {m}^{(*,*)}  &\vec{\mathbf{m}}_{N_1}^{(-,*)} & \vec{\mathbf{m}}_{N_2}^{(+,*)} \\
%    \\
%     \big(\vec{\mathbf{m}}_{N_1}^{(*,-)}\big)^T &\mathbf{M}_{N_1\times N_1}^{(-,-)}  &  \mathbf{M}_{N_1\times N_2}^{(+,-)} \\
%     \\
%    \big(\vec{\mathbf{m}}_{N_2}^{(*,+)}\big)^T  & \mathbf{M}_{N_2\times N_1}^{(-,+)} & \mathbf{M}_{N_2\times N_2}^{(+,+)}   \\
%  \end{array}
%\right),
%\quad
%\mathbf{A}=\left(
%  \begin{array}{ccc}
%    {a}^{(*,*)}  &\vec{\mathbf{a}}_{N_1}^{(-,*)} & \vec{\mathbf{a}}_{N_2}^{(+,*)} \\
%    \\
%     \big(\vec{\mathbf{a}}_{N_1}^{(*,-)}\big)^T &\mathbf{A}_{N_1\times N_1}^{(-,-)}  &  \mathbf{A}_{N_1\times N_2}^{(+,-)} \\
%     \\
%    \big(\vec{\mathbf{a}}_{N_2}^{(*,+)}\big)^T  & \mathbf{A}_{N_2\times N_1}^{(-,+)} & \mathbf{A}_{N_2\times N_2}^{(+,+)}   \\
%  \end{array}
%\right),
%\end{equation}
\begin{equation}\label{Matrixsys}
\mathbf{M}=\left(
  \begin{array}{ccc}
    \mathbf{M}_{1\times 1}^{(*,*)}  &{\mathbf{M}}_{1\times N_1}^{(*,-)} &{\mathbf{M}}_{1\times N_2}^{(*,+)} \\
    \\
    {\mathbf{M}}_{N_1\times 1}^{(-,*)} &\mathbf{M}_{N_1\times N_1}^{(-,-)}  &  \mathbf{M}_{N_1\times N_2}^{(-,+)} \\
     \\
    {\mathbf{M}}_{N_2\times 1}^{(+,*)}  & \mathbf{M}_{N_2\times N_1}^{(+,-)} & \mathbf{M}_{N_2\times N_2}^{(+,+)}   \\
  \end{array}
\right),
\quad
\mathbf{A}=\left(
  \begin{array}{ccc}
    {\mathbf{A}}_{1\times 1}^{(*,*)}  &{\mathbf{A}}_{1\times N_1}^{(*,-)} & {\mathbf{A}}_{1\times N_2}^{(*,+)} \\
    \\
     {\mathbf{A}}_{N_1\times 1}^{(-,*)} &\mathbf{A}_{N_1\times N_1}^{(-,-)}  &  \mathbf{A}_{N_1\times N_2}^{(-,+)} \\
     \\
    \mathbf{A}_{{N_2}\times 1}^{(+,*)}  & \mathbf{A}_{N_2\times N_1}^{(+,-)} & \mathbf{A}_{N_2\times N_2}^{(+,+)}   \\
  \end{array}
\right),
\end{equation}
with the entries
 \begin{equation*}\begin{aligned}
 &\mathbf{M}_{c\times d}^{(a,b)}(i+1,j+1)=(\phi^b_{j},\phi^a_{i}),\quad \mathbf{A}_{c\times d}^{(a,b)}(i+1,j+1)=a_{pq}^\mu(\phi^b_{j},\phi^a_{i}),\\ &a,b=*,-,+,\quad c,d=1,N_1,N_2,\qquad 0\leq i\leq c-1,\quad 0\leq j\leq d-1,
 \end{aligned}\end{equation*}
 and $\overrightarrow{{\mathbf{C}}}(0)$ is determined by the initial data.

 The proof of the tempered derivative relation \eqref{TDphi},  and the detail on the entries of the matrix $\mathbf{A}$ can be found in Appendix \ref{AppenixC}.
 Base on the semi-discrete scheme \eqref{MatrixSystem},  we further use the third-order explicit Runge-Kutta method in time direction with step size $h=10^{-3}$ to numerically solve the problem.

\begin{comment}
  \begin{color}{red}
 The derivation of the tempered derivative relation \eqref{TDphi},  and the detail on the entries of the matrixes $\mathbf{M}$ and $\mathbf{A}$ can be found in Appendix \ref{AppenixC}.
 Base on the semi-discrete scheme \eqref{MatrixSystem},  we further use the third-order explicit Runge-Kutta method in time direction with step size $h=10^{-3}$ to numerically solve the problem.
 \end{color}
 \begin{color}{blue}
 {\bf 1.  Need some detail on the entries of these matrices. 2. How do you solve the ODE system? By Runge-Kutta? This needs to be specified.}
 \end{color}
 \end{comment}

\subsection{Numerical results}
We solve  \eqref{tempfracdiffuequa} with $C_T=1$ and $u_0=10e^{-5|x|}$ as the initial distribution by using the method presented in the previous section.
We first test the accuracy of our method.  In Figure \ref{figure3},  we  plot the  the convergence rate of the spectral method at $T=5$ with fixed time step  $h=10^{-3}$, in which  $f(x,t)\equiv0$ and $f(x,t)= \cos{t} ~e^{-x^2}$ are the resource terms of the left and the right respectively.
\begin{comment}
\begin{color}{red}
We first test the accuracy of our method.  In Figure \ref{figure3},  we  plot the  the convergence rate of the spectral method at $T=5$ with fixed time step  $h=10^{-3}$, in which  $f(x,t)\equiv0$ and $f(x,t)= \cos{t} ~e^{-x^2}$ are the resource terms of the left and the right respectively.
\end{color}
\begin{color}{blue}
 How do you compute the errors? at a fixed time with small time steps?
\end{color}
\end{comment}
\begin{figure}[htp]
\begin{minipage}{0.48\linewidth}
\begin{center}
\includegraphics[scale=0.370]{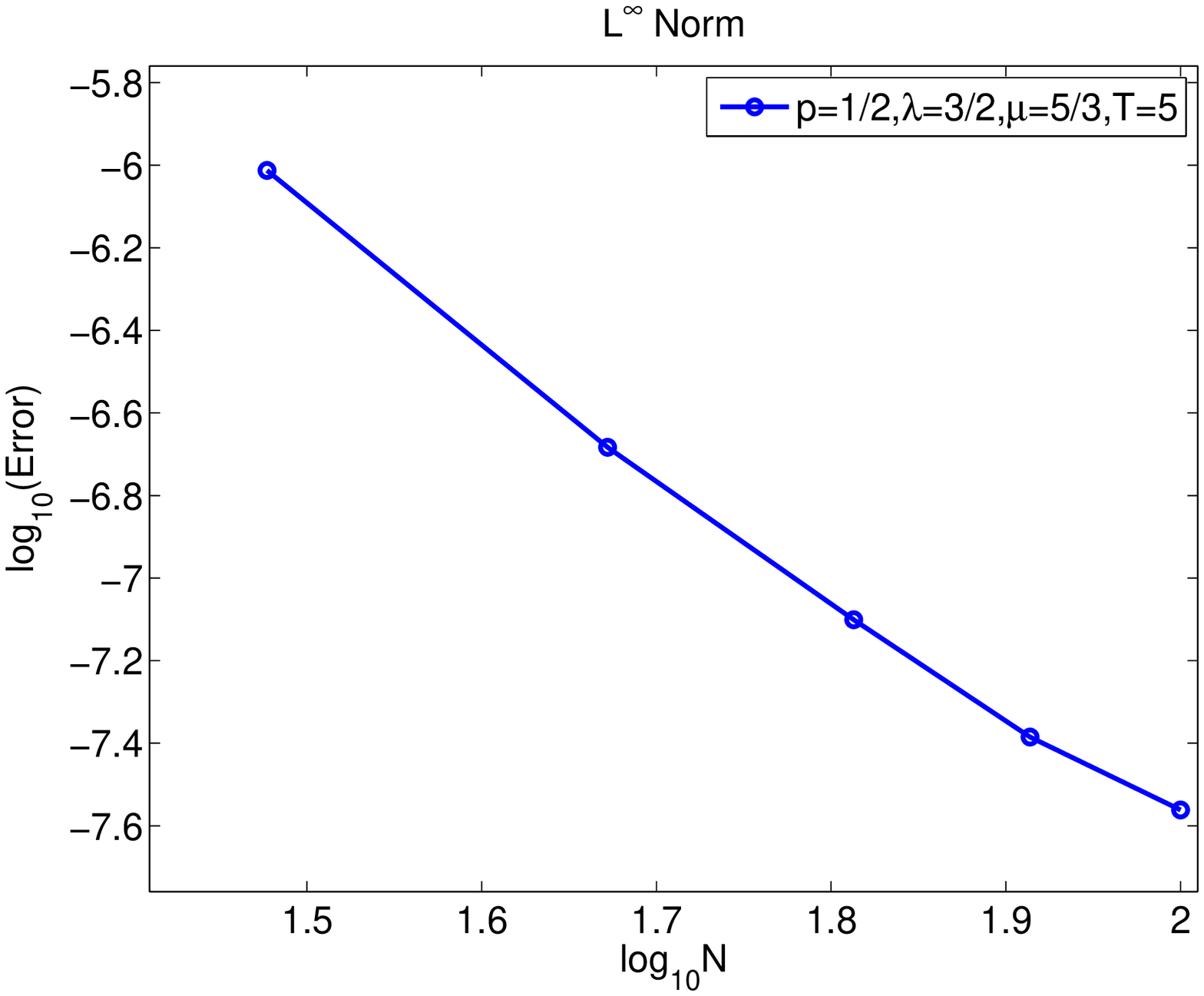}\end{center}
\end{minipage}
\begin{minipage}{0.48\linewidth}
\begin{center}
\includegraphics[scale=0.370]{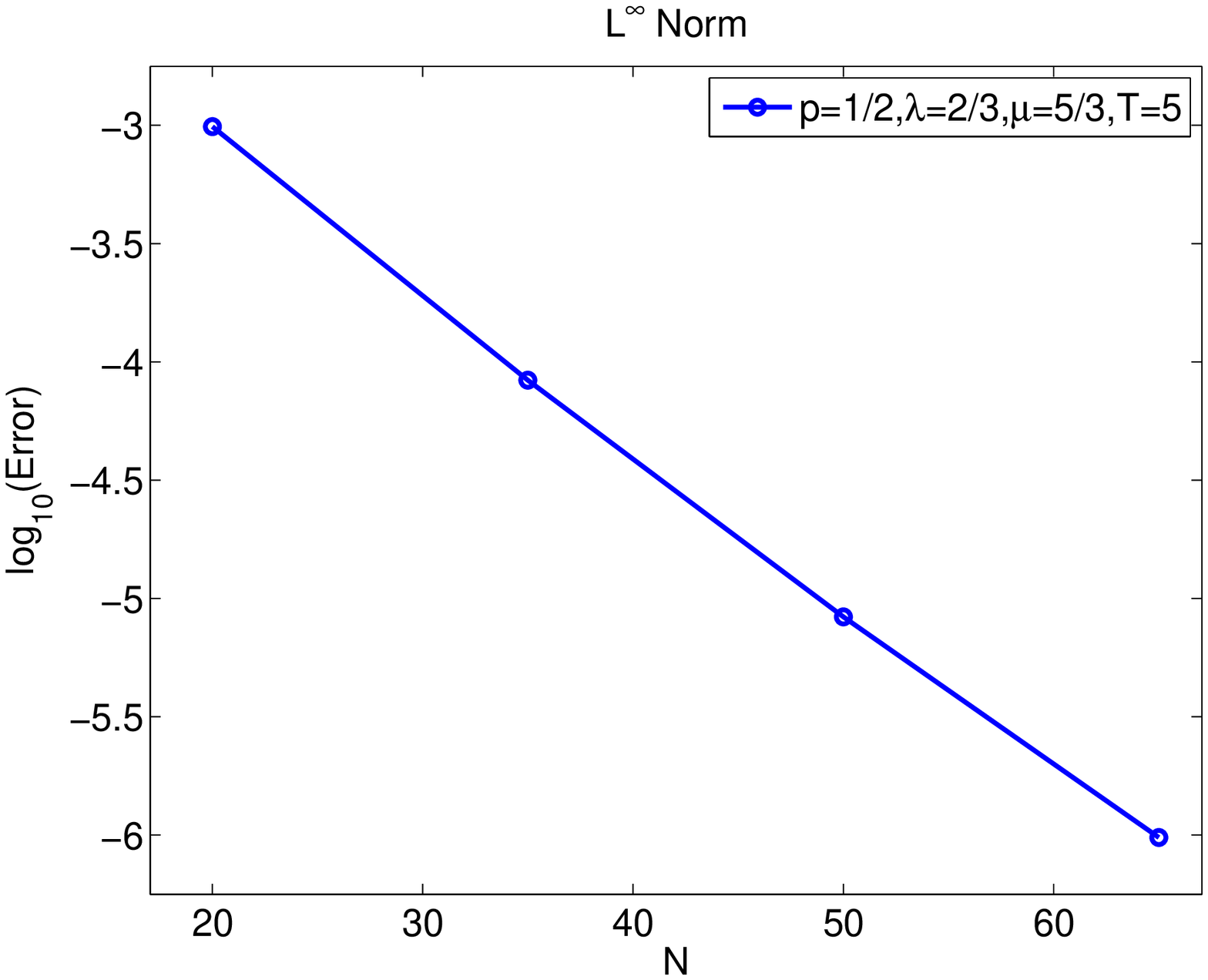}\end{center}
\end{minipage}
\caption{Left:  $f(x,t)\equiv 0.$ \quad Right:   $f(x,t)=\cos{t}~ e^{-x^2}$.}\label{figure3}
\end{figure}

Next, we examine  behaviors of the solution under various situations.
In Figure \ref{figure1}, we  plot the snapshots at different times of the tempered fractional diffusion with $p=1/3,~q=2/3$ and $p=3/4,~q=1/4$, respectively.  The case with  $p=q=1/2$ is plotted in  Figure \ref{figure2}. %Note that in this case the center of  solution  profile always fix at axis $x=0$.  In the right of the Figure \ref{figure2}, we plot the solution profile at $t=2$ with different  $\lambda=1,2,4,8$.

\begin{figure}[htp]
\begin{minipage}{0.48\linewidth}
\begin{center}
\includegraphics[scale=0.370]{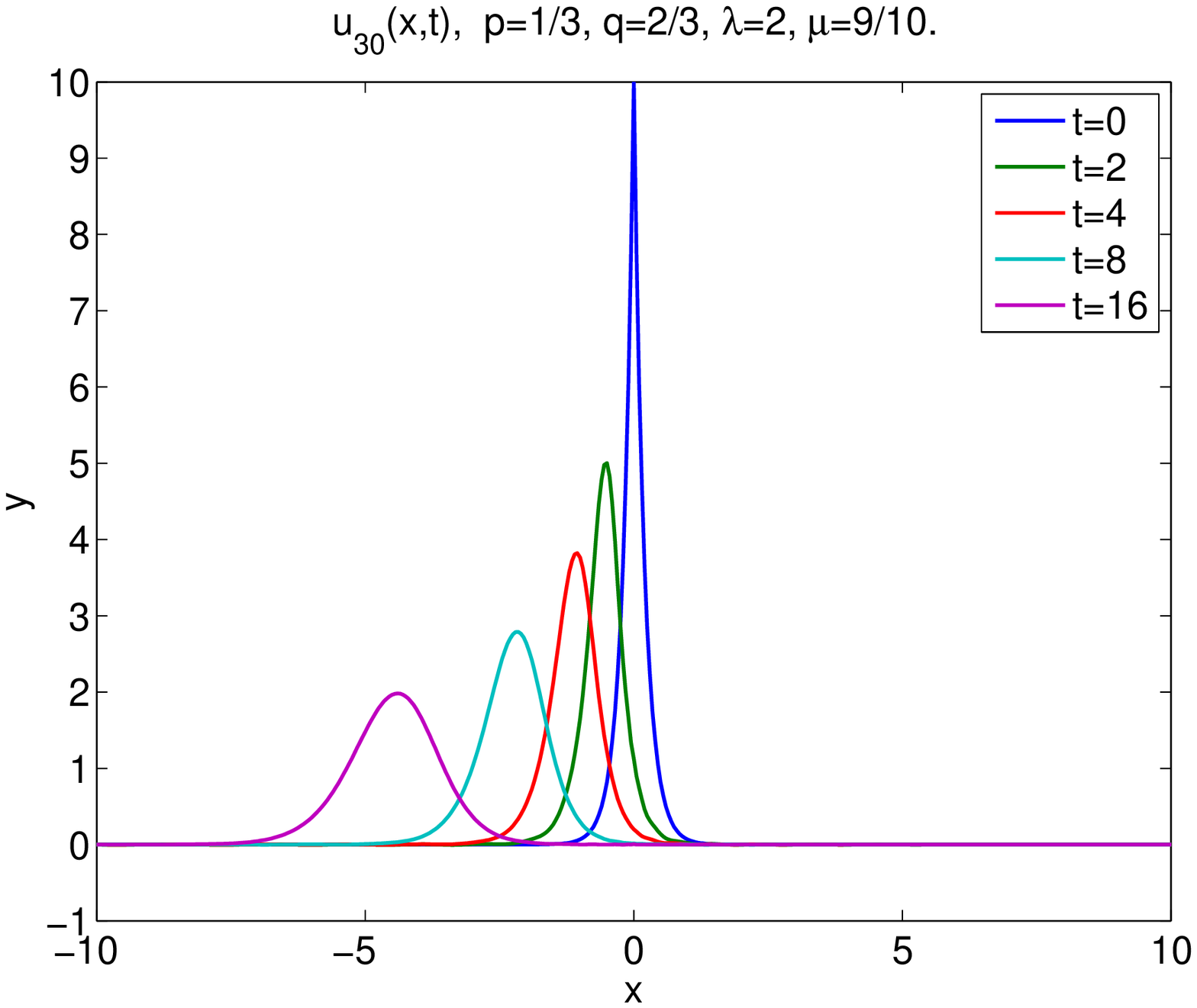}\end{center}
\end{minipage}
\begin{minipage}{0.48\linewidth}
\begin{center}
\includegraphics[scale=0.370]{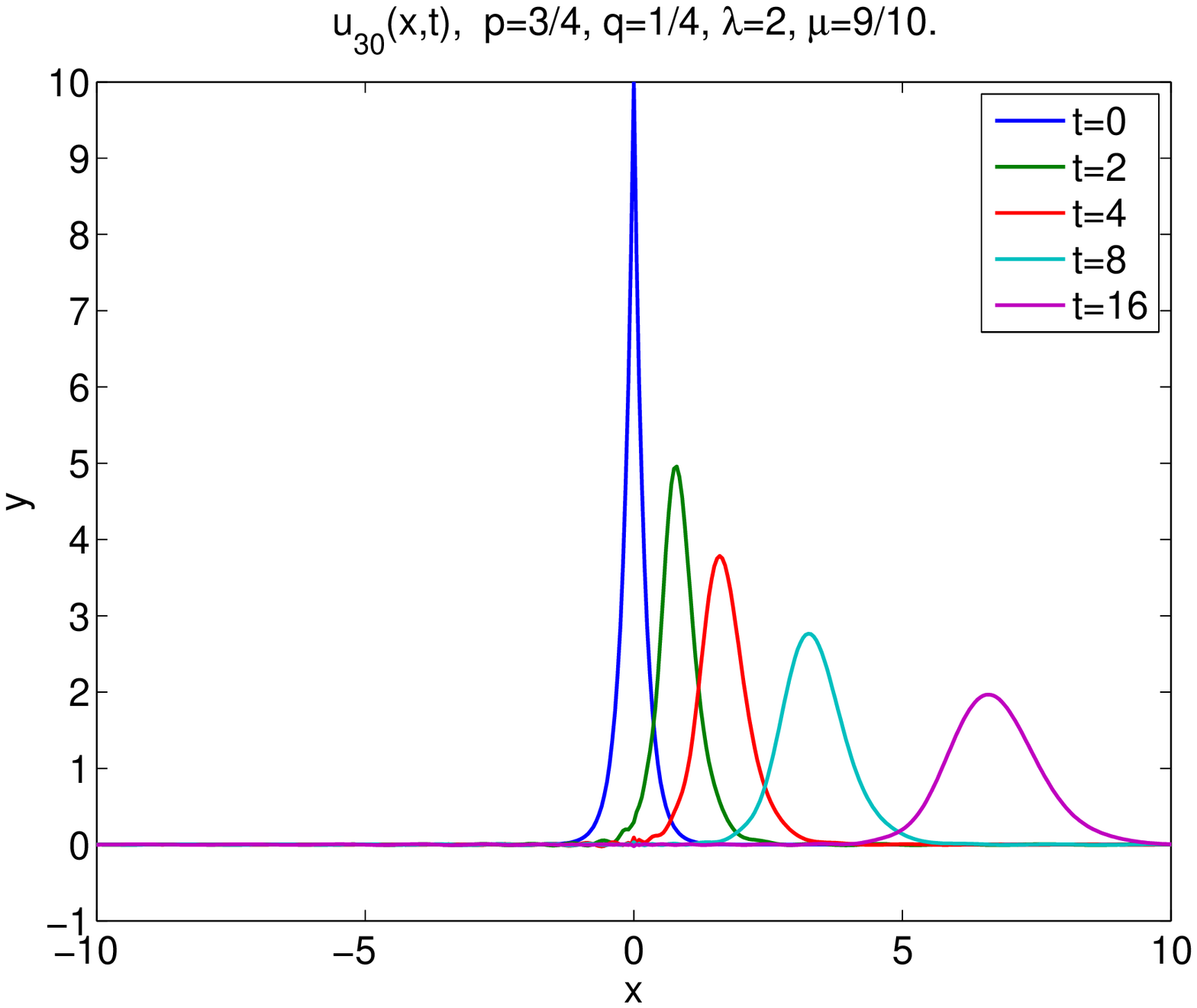}\end{center}
\end{minipage}
\caption{Left: $p=1/3,~q=2/3.$ \quad Right: $p=3/4,~q=1/4.$ }\label{figure1}%The evolution of the tempered fractional diffusion process (\textbf{TFDP}) with initial condition $u(x,0)=10~ e^{-5|x|}$  }
\end{figure}

\begin{figure}[htp]
\begin{minipage}{0.48\linewidth}
\begin{center}
\includegraphics[scale=0.370]{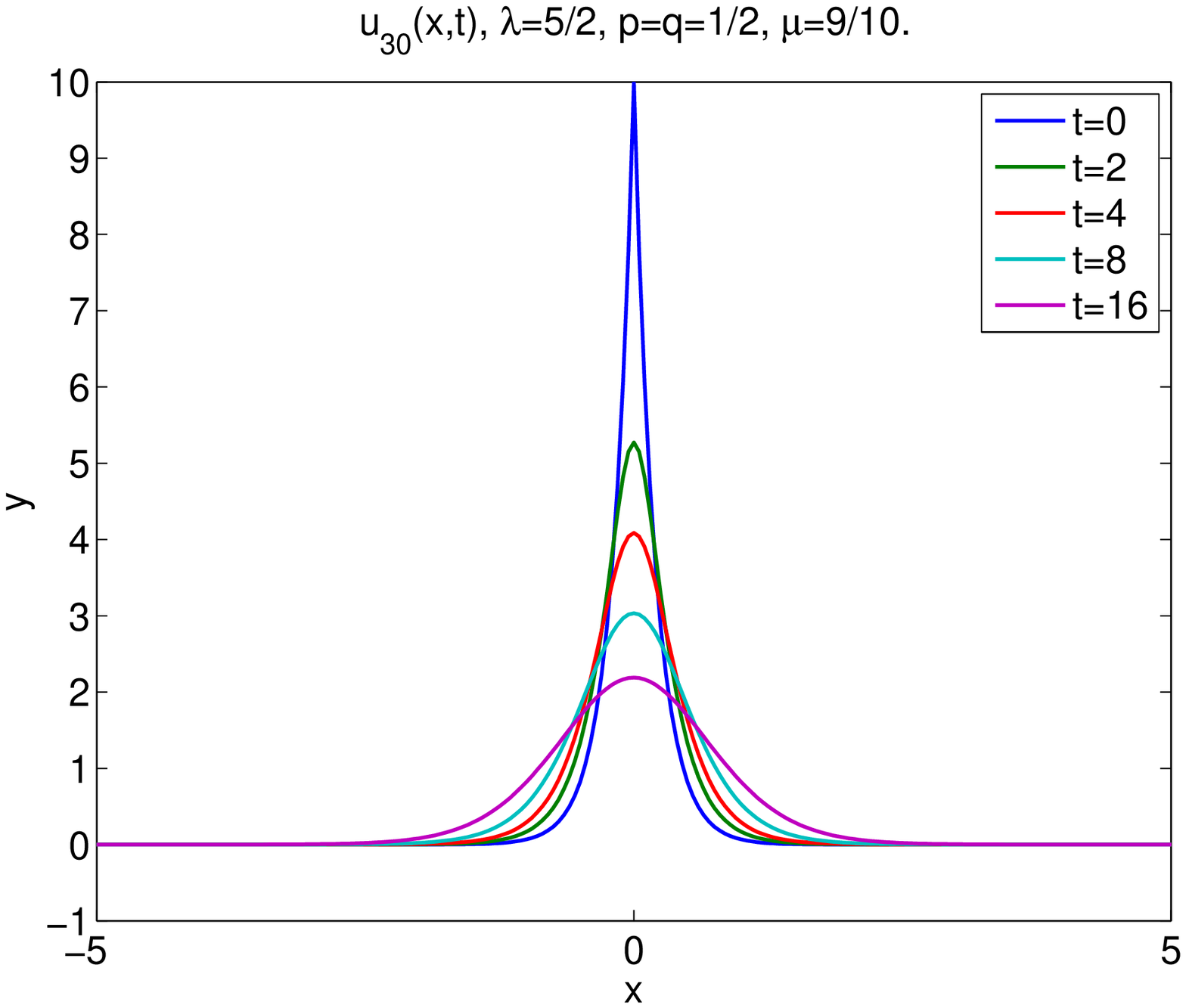}\end{center}
\end{minipage}
\begin{minipage}{0.48\linewidth}
\begin{center}
\includegraphics[scale=0.370]{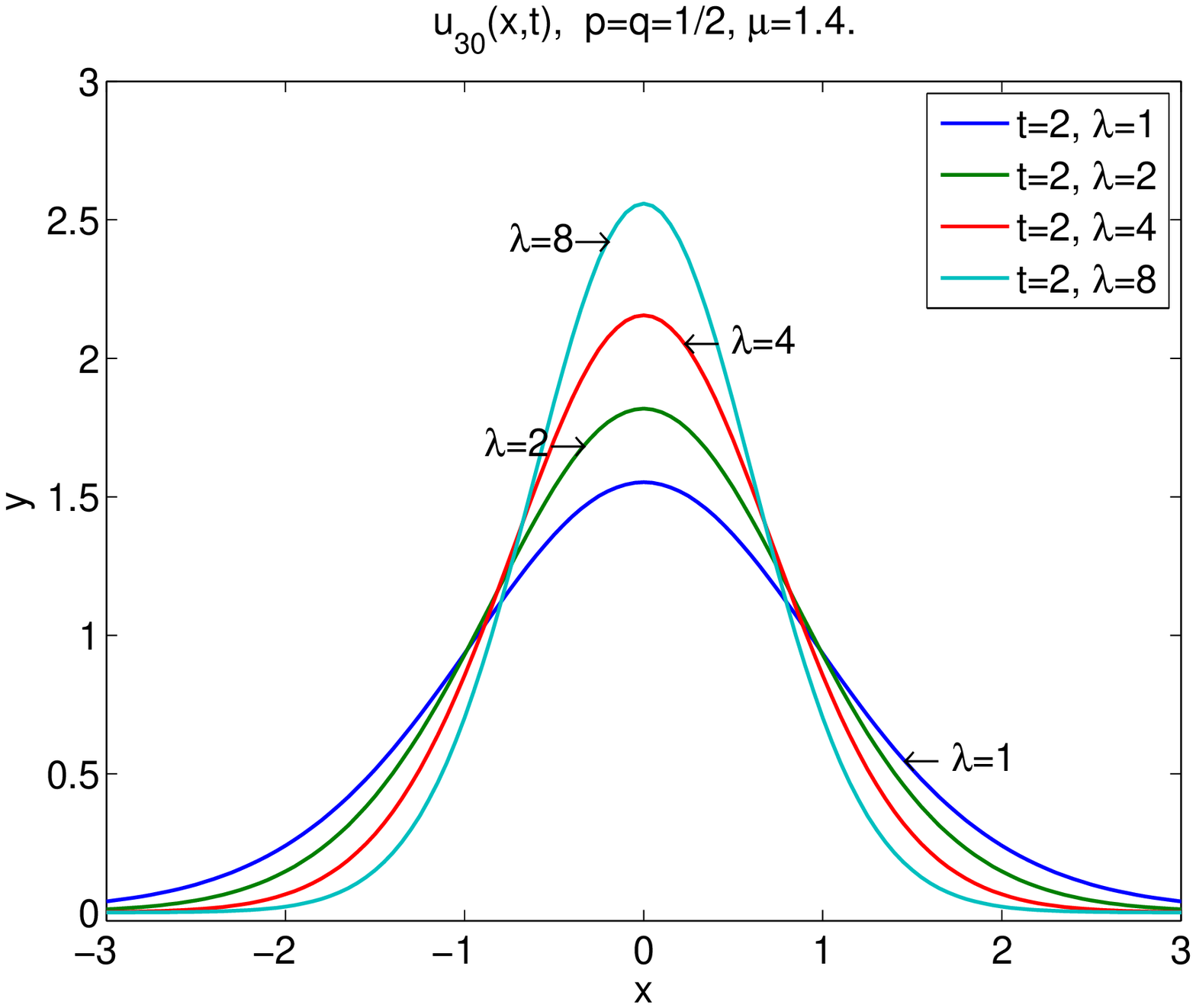}\end{center}
\end{minipage}
\caption{Left: $p=q=1/2,~\lambda=5/2.$ \quad Right: $p=q=1/2,~t=2.$ }\label{figure2}%The evolution of the tempered fractional diffusion process (\textbf{TFDP}) with initial condition $u(x,0)=10~ e^{-5|x|}$  }\label{figure1}
\end{figure}

%Note  that  solution of \eqref{tempfracdiffuequa} is the limit, as $\varepsilon\rightarrow 0$,  of the solutions to the  diffusion process of the particles which jump with the probability density function (cf.
 %\cite[(8)]{temperedsabzikar2014})
%\begin{equation}\label{PDF}
%f_{\varepsilon}(x)=C^{-1}_{\varepsilon}[px^{-\mu-1}e^{-\lambda x}\mathbf{1}_{({\varepsilon,\infty})}(x)+ q|x|^{-\mu-1}e^{-\lambda |x|}\mathbf{1}_{({-\infty, \varepsilon})}(x)],\quad p+q=1.
%\end{equation}

\begin{itemize}
\item The parameters $p$ and $q$ reflect the directional preference of the particle jumping. More precisely, if $p>q$, the particles tend to jump to the right,  and if $p<q$,  the particles tend to jump to the left, see Figure \ref{figure1}.
 In particular,  $p=q$ produces a symmetric profile in the case of $f(x,t)=0$, see the left in  Figure \ref{figure2}.

\begin{comment}{\em Numerical experiments showed in Figure \ref{figure1} and the left of Figure \ref{figure2} simulate  the activities. In particular, when $p=1 ~(q=1)$, the above probability density function of the particles jump indicates that all particles will jump with the right (left) direction which straightforwardly leads to the special case discussed in Subsection \ref{subsectionspecial}.}\comm{\color{red} Needs to improve!}\end{comment}

\item  The  parameter $\lambda$ determines the probability of the  jump distance of the particles. A larger $\lambda$ indicates a shorter  jump distance, see the right of Figure \ref{figure2}.
%This preference is a good interpretation for the word ``tempered".
\item To compare with  the usual fractional diffusion equation, i.e., $\lambda=0$,  , we plot in Figure \ref{figure_compare}  the particle distributions of the usual  fractional diffusion and the tempered fractional diffusion  with initial distribution $u_0(x)=10e^{-4 x^2}$ at time $t=10$.
We observe that  the tail of the tempered fractional diffusion behaves like $|x|^{-\mu-1}e^{-\lambda |x|}$ for large $|x|$ while that that of the  usual fractional diffusion  behaves like $|x|^{-\mu-1}$.
\end{itemize}

\begin{figure}[htp]
\begin{center}
\includegraphics[scale=0.570]{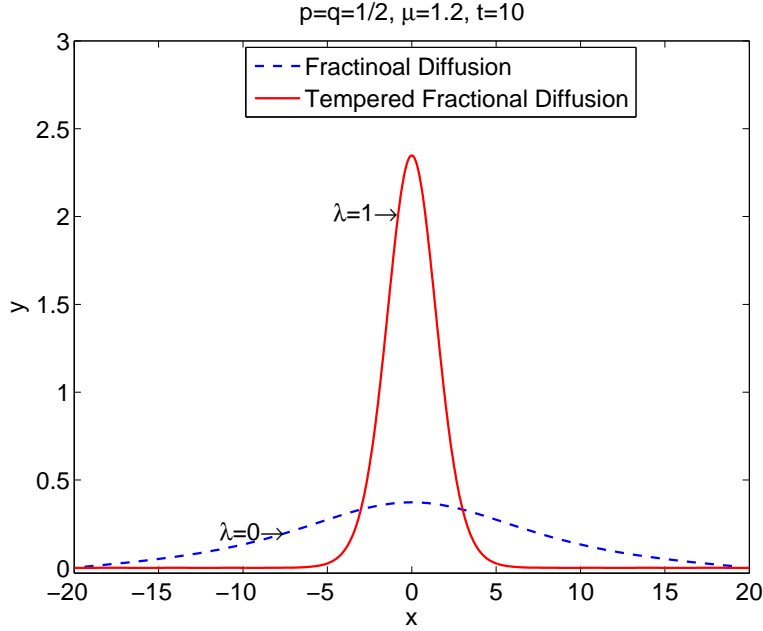}
\end{center}
\caption{Initial distribution $u_0(x)=10e^{-4 x^2}$.}\label{figure_compare}%The evolution of the tempered fractional \end{figure}
\end{figure}

\section{Concluding remarks}
We presented in this paper efficient spectral methods using the generalized Laguerre functions for solving the tempered fractional differential equations. Our numerical methods and analysis are based on an  important observation  that the tempered fractional derivative, when restricted to the half line,
is intrinsically related to the generalized Laguerre functions that we defined in Sections \ref{SectionLaguerre}. By exploring the properties of generalized Laguerre functions, we derived optimal approximation results in properly weighted Sobolev spaces, and showed that

 we define two classes of generalized Laguerre functions, study their approximation properties, and apply them for solving  simple one sided tempered fractional equations. In Section \ref{SectionApp2}, we develop a spectral-Galerkin method for solving a tempered fractional diffusion equation   on the half line. Finally, we present a spectral-Galerkin method for solving the tempered fractional diffusion equation  on the whole line

%\begin{figure}[htp]
%\begin{minipage}{0.48\linewidth}
%\begin{center}
%\includegraphics[scale=0.370]{GraphSEM/figureErr12}\end{center}
%\end{minipage}
%\begin{minipage}{0.48\linewidth}
%\begin{center}
%\includegraphics[scale=0.370]{GraphSEM/figurefxt}\end{center}
%\end{minipage}
%\caption{Left: \textbf{TFDP}   with zeros resource  \quad Right:  \textbf{TFDP}  with $f(x,t)=\cos{t}~ e^{-x^2}$}\label{figure3}
%\end{figure}

 \begin{appendix}
\section{Proof of Lemma \ref{Lemmalagupolyderi}} \label{AppenixB}
\renewcommand{\theequation}{A.\arabic{equation}}
\setcounter{equation}{0}

We first  prove  \eqref{lagupolyint}-\eqref{lagupolyderi}.  Recall the fractional integral formula of hypergeometric functions (see \cite[P. 287]{Andrews99}):
%Recall the important  Bateman's fractional integral formula \cite{Bateman1909}
  for real $b,\rho\ge 0,$
  \begin{equation}\label{Represtation}
  x^{b+\rho-1}\,{}_1F_1(a;b+\rho; x)=\frac{\Gamma(b+\rho)}{\Gamma(b)\Gamma(\rho)}\int_0^x (x-t)^{\rho-1}t^{b-1}
  {}_1F_1(a;b; t)\, {\rm d}t,\quad x\in \mathbb{R}^+.
  \end{equation}
Taking $a=-n,~b=\alpha+1$ and using  the hypergeometric representation \eqref{Laguerre} of the Laguerre polynomials,  we obtain
\begin{equation*}
x^{\alpha+\rho} L_n^{(\alpha+\rho)}(x)
=\frac{\Gamma(n+\alpha+\rho+1)}{\Gamma(n+\alpha+1)\Gamma(\rho)}\int_{0}^x {(x-t)^{\rho-1}} t^\alpha
L_n^{(\alpha)}(t)\, {\rm d}t,
\end{equation*}
which yields  \eqref{lagupolyint}, i.e.,
$$
\I{0}{x}{\rho} \{x^\alpha L_n^{(\alpha)}(x)\}=h_n^{\alpha,-\rho}\,x^{\alpha+\rho}\, L_n^{(\alpha+\rho)}(x).
$$
Then, performing $\D{0}{x}{\mu}$ on both sides and taking $\alpha+\mu\to \alpha,$ we
derive from the relation \eqref{rulesa} that for $\alpha-\mu>-1,$
$$\D{0}{x}{\mu} \{x^{\alpha} L_n^{(\alpha)}(x)\}= \frac 1 {h_n^{\alpha-\mu,-\rho}} x^{\alpha-\mu} L_n^{(\alpha-\mu)}(x)=\frac{\Gamma(n+\alpha+1)}{\Gamma(n+\alpha-\mu+1)}
x^{\alpha-\mu} L_n^{(\alpha-\mu)}(x).$$
This leads to \eqref{lagupolyderi}.
%Furthermore, if $\alpha-[\mu]>0$,  using the property: $\CD{0}{x}{s}=\I{0}{x}{[\mu]+1-\mu}\D{0}{x}{[\mu]+1},$
%we derive that
%\begin{equation*}
%\CD{0}{x}\mu \{x^{\alpha} L_n^{(\alpha)}(x)\}=h^{\alpha,\mu}_n~x^{\alpha-\mu} L_n^{(\alpha-\mu)}(x).
%\end{equation*}
%This ends the verification.
\begin{comment}
\end{proof}
\begin{lemma}\label{Lemmalagufunderi}
 Let  $\mu\in  {\mathbb R}^+.$
\begin{equation}\label{Laguerreintegralr}
\quad\I{x}{\infty}\rho \{e^{-x} L_n^{(\alpha)}(x)\}= e^{-x} L_n^{(\alpha-\mu)}(x), \quad \alpha>\mu-1,
\end{equation}
\begin{equation}\label{Laguerrederivativer}
\D{x}{\infty}\mu \{ e^{-x} L_n^{(\alpha)}(x)\}=e^{-x} L_n^{(\alpha+\mu)}(x),\quad \alpha>-1.
\end{equation}
%\begin{equation}\label{Laguerrederivativer}
%\D{x}{\infty}\mu \{ e^{-x} L_n^{(\alpha)}(x)\}=\CD{}{-}s \{ e^{-x} L_n^{(\alpha)}(x)\}=e^{-x} L_n^{(\alpha+s)}(x).
%\end{equation}
\end{lemma}
\begin{proof}
\end{comment}

We now turn to \eqref{Laguerreintegralr}-\eqref{Laguerrederivativer}.
According to  \cite[(6.146), P. 191 ]{pod99} (or  \cite[(B-7.2), P. 307]{popov1982concentration}), we have
\begin{equation*}
\I{x}{\infty}{\rho} \{e^{-x} L_n^{(\alpha+\rho)}(x)\}= e^{-x} L_n^{(\alpha)}(x),\quad\alpha>-1,~\rho>0.
\end{equation*}
Similarly,  from the property: $\D{x}{\infty}{\rho} \I{x}{\infty}{\rho}\, u(x)=u(x),$ we derive
\begin{equation*}
\D{x}{\infty}{\mu} \{ e^{-x} L_n^{(\alpha)}(x)\}=e^{-x} L_n^{(\alpha+\mu)}(x).
\end{equation*}
%%Moreover, since\;   $\CD{}{-}{s}=\I{x}{\infty}{[s]+1-s}\D{x}{\infty}{[s]+1},$ we have the same result for the Caputo fractional derivative.
%\begin{comment}
%\end{proof}
%\begin{lemma}\label{Lemmalaguderi} Let  $k\in\mathbb{N}$ and $\alpha>k-1$,
% \begin{equation}\label{lagufunderi}
% \D{}{}{k}  \big\{x^\alpha L_{n}^{(\alpha)}(x)e^{-x}\big\}=\frac{\Gamma(n+k+1)}{\Gamma(n+1)}x^{\alpha-k}L^{(\alpha-k)}_{n+k}(x)e^{-x}.
% \end{equation}
% \end{lemma}
% \begin{proof}
% \end{comment}

 Finally, we  prove  \eqref{lagufunderi}.
  Noting that
 \begin{equation}\label{bridge}
 _1F_1(a;c; x)=e^x~  _1F_1(c-a;c; -x),
 \end{equation}
(cf.  \cite[P. 191]{andrews1999special}), we derive from   \eqref{Laguerre} and  \eqref{bridge} that
 \begin{equation}\begin{aligned}\label{partproperty}
 x^\alpha & L_{n}^{(\alpha)}(x)e^{-x}=x^\alpha\frac{(\alpha+1)_n}{n!}~{}_1F_1\big(-n;\alpha+1;x\big)e^{-x}\\
 &=\frac{(\alpha+1)_n}{n!}x^\alpha~{}_1F_1\big(n+\alpha+1;\alpha+1;-x\big)=\frac{(\alpha+1)_n}{n!}x^\alpha~{}_1F_1\big(n+\alpha+1;\alpha+1;-x\big)\\
 &=\frac{(\alpha+1)_n}{n!}\sum_{j=0}^\infty \frac{(n+\alpha+1)_j(-1)^j}{(\alpha+1)_j}\frac{x^{j+\alpha}}{j!}.
 \end{aligned}\end{equation}
Then acting the derivative $\D{}{}{k}$ on \eqref{partproperty} and using the identities \eqref{Laguerre}, \eqref{bridge} again, we obtain
  \begin{equation*}\begin{aligned}
\D{}{}{k} \big\{x^\alpha L_{n}^{(\alpha)}(x)e^{-x}\big\}&=\frac{(\alpha+1)_n}{n!}\sum_{j=0}^\infty \frac{(n+\alpha+1)_j(-1)^j}{(\alpha+1)_j}\frac{\D{}{}{k} x^{j+\alpha}}{j!}\\
 &=\frac{(\alpha+1)_n}{n!}\sum_{j=0}^\infty \frac{(n+\alpha+1)_j(-1)^j}{(\alpha+1)_j}\frac{\Gamma(j+\alpha+1) }{\Gamma(j+\alpha-k+1)}\frac{ x^{j+\alpha-k}}{j!}\\
  &=x^{\alpha-k}\frac{(\alpha+1)_n}{n!}\frac{\Gamma(\alpha+1)}{\Gamma(\alpha-k+1)}\sum_{j=0}^\infty \frac{(n+\alpha+1)_j}{(\alpha-k+1)_j}\frac{ (-x)^{j}}{j!}\\
&=x^{\alpha-k}\frac{(\alpha-k+1)_{n+k}}{n!}{}_1F_1\big(n+\alpha+1;\alpha-k+1;-x\big)\\
&=\frac{(n+k)!}{n!}x^{\alpha-k}\frac{(\alpha-k+1)_{n+k}}{(n+k)!}{}_1F_1\big(-n-k;\alpha-k+1;x\big)e^{-x}\\
&=\frac{\Gamma(n+k+1)}{\Gamma(n+1)}x^{\alpha-k}L^{(\alpha-k)}_{n+k}(x)e^{-x}.\\
 \end{aligned}\end{equation*}
This ends the proof.

\section{The proof of \eqref{TDphi} and the detail on the entries of $\mathbf{A}$  } \label{AppenixC}
\renewcommand{\theequation}{B.\arabic{equation}}
\setcounter{equation}{0}
\textbf{Proof of \eqref{TDphi}}
  \begin{itemize}
  \item for $x\in\mathbb{R}^-$, $0<s<1$,
  \begin{equation}
\begin{aligned}
\TD{-\infty}{x}{s}\phi^{\ast}(x)&=\TI{-\infty}{x}{1-s}\TD{-\infty}{x}{1}\phi^{\ast}(x)=\frac{e^{-\lambda x}}{\Gamma(1-s)}\int_{-\infty}^x\frac{e^{\lambda \tau}(2\lambda)e^{\lambda \tau}}{(x-\tau)^s} {\rm d}\tau\\
&\overset{t=x-\tau}{=}\frac{2\lambda e^{-\lambda x}}{\Gamma(1-s)}\int_0^{\infty}\frac{e^{2\lambda (x-t)}}{t^s} {\rm d}t
=\frac{(2\lambda )^{s}e^{\lambda x}}{\Gamma(1-s)}\int_0^{\infty}e^{-2\lambda t}(2\lambda t)^{-s} {\rm d}(2\lambda t)\\&=(2\lambda )^{s}e^{\lambda x},\\
%\end{aligned}
%\end{equation}
% \begin{equation}
%\begin{aligned}
\TD{-\infty}{x}{s}\phi^{-}_{n_1}(x)&=\TI{-\infty}{x}{1-s}\TD{-\infty}{x}{1}\phi^-_{n_1}(x)\overset{\eqref{GLFkderi+-}}{=}\TI{-\infty}{x}{1-s}\{-({n_1}+1)\GLa{0}{{n_1}+1}{(-x)}\}\\
&\overset{\eqref{Laguerreintegralr}}{=}-({n_1}+1)(2\lambda)^{s-1}L^{(s-1)}_{{n_1}+1}{(-2\lambda x)}e^{-\lambda x}\\
%\end{aligned}
%\end{equation}
% \begin{equation}
% \begin{aligned}
\TD{-\infty}{x}{s}\phi^{+}_{n_2}(x)&=0
\end{aligned}
\end{equation}
 \item for $x\in\mathbb{R}^+$, $0<s<1$,
  \begin{equation}
\begin{aligned}
\TD{-\infty}{x}{s}\phi^{\ast}(x)&=\TI{-\infty}{x}{1-s}\TD{-\infty}{x}{1}\phi^{\ast}(x)=\frac{e^{-\lambda x}}{\Gamma(1-s)}\int_{-\infty}^0\frac{2\lambda e^{2\lambda \tau}}{(x-\tau)^s} {\rm d}\tau
%&\overset{\tau=-xt}{=}x^{1-s}e^{-\lambda x}\frac{2\lambda }{\Gamma(1-s)}\int_0^{\infty}\frac{ e^{-2\lambda xt}}{(1+t)^s} {\rm d}t\\
\overset{\tau=x-t}{=}\frac{2\lambda e^{\lambda x}}{\Gamma(1-s)}\int_x^{\infty}\frac{ e^{-2\lambda t}}{t^s} {\rm d}t,\\
\TD{-\infty}{x}{s}\phi^{-}_{n_1}(x)&=\TI{-\infty}{x}{1-s}\TD{-\infty}{x}{1}\phi^{-}_{n_1}(x)\overset{\eqref{GLFkderi+-}}{=}
\frac{e^{-\lambda x}}{\Gamma(1-s)}\int_{-\infty}^0\frac{-(n_1+1)L^{(0)}_{{n_1}+1}{(-2\lambda\tau)}e^{2\lambda\tau}}{(x-\tau)^s} {\rm d}\tau\\
%&\overset{\tau=-xt}{=}-x^{1-s}e^{-\lambda x}\frac{{n_1}+1}{\Gamma(1-s)}\int_0^{\infty}\frac{L^{(0)}_{{n_1}+1}{(2\lambda t)}e^{-2\lambda t}}{(1+t)^s} {\rm d}t\\
&\overset{\tau=x-t}{=}-e^{\lambda x}\frac{{n_1}+1}{\Gamma(1-s)}\int_x^{\infty}\frac{L^{(0)}_{{n_1}+1}{(2\lambda (t-x))}e^{-2\lambda t}}{t^s} {\rm d}t,\\
\TD{-\infty}{x}{s}\phi^{+}_{n_2}(x)&=\TI{-\infty}{x}{1-s}\TD{-\infty}{x}{1}\phi^{+}_{n_2}(x)\overset{\eqref{GLFkderi++}}{=}\frac{e^{-\lambda x}}{\Gamma(1-s)}\int_0^{x}\frac{({n_2}+1)L^{(0)}_{{n_2}}{(x)}}{(x-\tau)^s} {\rm d}\tau\\
&=\frac{\Gamma(n_2+2)}{\Gamma(n_2+2-s)}x^{1-s}\GLa{1-s}{{n_2}}{(x)}.
\end{aligned}
\end{equation}
\end{itemize}
\begin{comment}
\begin{itemize}
\item
  \begin{equation}
\begin{aligned}
\TD{x}{\infty}{1}\phi^{\ast}(x)&=
\begin{cases}
0,&x\leq 0,\\
-2\lambda e^{-\lambda x},&x>0,
\end{cases}\\
\TD{x}{\infty}{1}\phi^{-}_{n_1}(x)&=
\begin{cases}
(n_1+1)\GLb{0}{n_1}{(-x)},&x\leq0,\\
0,&x>0,
\end{cases}\\
\TD{x}{\infty}{1}\phi^{+}_{n_2}(x)&=
\begin{cases}
0,&x\leq 0,\\
-(n_2+1)\GLb{0}{n_2+1}{(x)},&x>0.
\end{cases}
\end{aligned}
\end{equation}
\item
  \begin{equation}\label{basisTFDeri}
\begin{aligned}
\TD{-\infty}{x}{s}\phi^{\ast}(x)&=
\begin{cases}
(2\lambda )^{s}e^{\lambda x},&x\leq 0,\\
\dfrac{2\lambda e^{\lambda x}}{\Gamma(1-s)}\int_x^{\infty}\dfrac{ e^{-2\lambda t}}{t^s} {\rm d}t
%x^{1-s}e^{-\lambda x}\dfrac{2\lambda }{\Gamma(1-s)}\int_0^{\infty}\dfrac{ e^{-2\lambda xt}}{(1+t)^s} {\rm d}t
,&x>0,
\end{cases}\\
\TD{-\infty}{x}{s}\phi^{-}_{n_1}(x)&=
\begin{cases}
-(2\lambda)^{s-1}(n_1+1)\GLb{s-1}{n_1+1}{(-x)},&x\leq0,\\
-e^{\lambda x}\dfrac{{n_1}+1}{\Gamma(1-s)}\int_x^{\infty}\dfrac{L^{(0)}_{{n_1}+1}{(2\lambda (t-x))}e^{-2\lambda t}}{t^s} {\rm d}t
%-x^{1-s}e^{-\lambda x}\dfrac{n_1+1}{\Gamma(1-s)}\int_0^{\infty}\dfrac{L^{(0)}_{n_1+1}{(2\lambda xt)}e^{-2\lambda xt}}{(1+t)^s} {\rm d}t
,&x>0,
\end{cases}\\
\TD{-\infty}{x}{s}\phi^{+}_{n_2}(x)&=
\begin{cases}
0,&x\leq 0,\\
\dfrac{\Gamma(n_2+2)}{\Gamma(n_2+2-s)}x^{1-s}\GLb{1-s}{n_2}{(x)},&x>0.
\end{cases}
\end{aligned}
\end{equation}
\end{itemize}
\end{comment}
\textbf{The entries of matrix $\mathbf{A}$ with $1<\mu=1+s<2$.}
\begin{comment}
Let  $k\in\mathbb{N}$, $k-\nu\leq 0$ and $x\in\mathbb{R}^-$. Via the relations \eqref{GLFkderi++}, \eqref{GLFkderi+-} and \eqref{rela+-D}, it's easy to prove that
\begin{equation}\begin{aligned}\label{GLFkderi--}
\TD{-\infty}{x}{k} \big\{\GLa{-\nu}{n}{(-x)}\big\}=
%\frac{\Gamma(n+\alpha+1)}{\Gamma(n+\alpha-s+1)}
\frac{\Gamma(n+\nu+1)}{\Gamma(n+\nu-k+1)}~\GLa{k-\nu}{n}{(-x)} ,
\end{aligned}\end{equation}
 \begin{equation}\label{GLFkderi-+}
 \qquad\quad\TD{x}{\infty}{k}  \big\{\GLa{-\nu}{n}{(-x)}\big\}=(-1)^k\frac{\Gamma(n+k+1)}{\Gamma(n+1)}\GLa{k-\nu}{n+k}{(-x)}.\qquad
 \end{equation}
Then,
\end{comment}
 \begin{equation}
\begin{aligned}
\big(\TD{-\infty}{x}{s}\phi^{\ast},&\TD{x}{\infty}{1}\phi^{\ast}\big)=\frac{-(2\lambda)^2}{\Gamma(1-s)}\int_0^\infty \int_x^{\infty}\frac{ e^{-2\lambda t}}{t^s} {\rm d}t{\rm d}x=\frac{-(2\lambda)^2}{\Gamma(1-s)}\int_0^\infty\frac{ e^{-2\lambda t}}{t^s} \int_0^{t}1 {\rm d}x{\rm d}t\\
&=\frac{-(2\lambda)^2}{\Gamma(1-s)}\int_0^\infty{t^{1-s}}{ e^{-2\lambda t}}{\rm d}t\overset{\tau=2\lambda t}{=}\frac{-(2\lambda)^{s}}{\Gamma(1-s)}\int_0^\infty{\tau^{1-s}}{ e^{-\tau}}{\rm d}\tau=(s-1)(2\lambda)^{s}.
\end{aligned}
\end{equation}
Since
$$\D{}{}{}\big\{(2\lambda x)L^{(1)}_{n_2+1}(2\lambda x)\big\}=2\lambda(n_2+2) L^{(0)}_{n_2+1}(2\lambda x),~
\text{i.e.}
~\int_0^{t} L^{(0)}_{n_2+1}(2\lambda x){\rm d}x=\frac{1}{n_2+2} tL^{(1)}_{n_2+1}(2\lambda t),$$
then,
 \begin{equation}
\begin{aligned}
\big(\TD{-\infty}{x}{s}\phi^{\ast},\TD{x}{\infty}{1}\phi^+_{n_2}\big)
&=\frac{-2\lambda(n_2+1)}{\Gamma(1-s)}\int_0^\infty\int_x^{\infty}\frac{ e^{-2\lambda t}}{t^s} {\rm d}t~ L^{(0)}_{n_2+1}(2\lambda x){\rm d}x\\
&{=}\frac{-2\lambda(n_2+1)}{\Gamma(1-s)}\int_0^\infty\frac{ e^{-2\lambda t}}{t^s}\int_0^{t}L^{(0)}_{n_2+1}(2\lambda x) {\rm d}x~ {\rm d}t\\
&{=}\frac{-2\lambda(n_2+1)}{(n_2+2)\Gamma(1-s)}\int_0^\infty t^{1-s}L^{(1)}_{n_2+1}(2\lambda t)e^{-2\lambda t} {\rm d}t.
\end{aligned}
\end{equation}
Similarly, we have
\begin{equation}\begin{aligned}
\big(\TD{-\infty}{x}{s}\phi^{-}_{n_1},\TD{x}{\infty}{1}\phi^{+}_{n_2}\big)
&{=}\frac{({n_1}+1)({n_2}+1)}{\Gamma(1-s)}\int_0^\infty \int_x^{\infty}\frac{L^{(0)}_{{n_1}+1}{(2\lambda (t-x))}e^{-2\lambda t}}{t^s} {\rm d}t ~L^{(0)}_{n_2+1}(2\lambda x) {\rm d}x\\
{=}&\frac{({n_1}+1)({n_2}+1)}{\Gamma(1-s)}\int_0^\infty {t^{-s}}e^{-2\lambda t} \int_0^{t}~L^{(0)}_{n_2+1}(2\lambda x)L^{(0)}_{{n_1}+1}{(2\lambda (t-x))} {\rm d}x  {\rm d}t\\
\overset{x=t\xi}{=}&\frac{({n_1}+1)({n_2}+1)}{\Gamma(1-s)}\int_0^\infty {t^{1-s}}e^{-2\lambda t} \int_0^{1}~ L^{(0)}_{n_2+1}(2\lambda t\xi)L^{(0)}_{{n_1}+1}{(2\lambda t(1-\xi))} {\rm d}\xi  {\rm d}t.
\end{aligned}\end{equation}
\textbf{The entries of matrix $\mathbf{A}$ with $0<\mu=s<1$.}
 \begin{equation}
\begin{aligned}
&\big(\TD{-\infty}{x}{s}\phi^{\ast},\phi^{\ast}\big)=(2\lambda)^s\int_{-\infty}^0e^{2\lambda x} {\rm d}x+\frac{2\lambda}{\Gamma(1-s)}\int_0^\infty \int_x^{\infty}\frac{ e^{-2\lambda t}}{t^s} {\rm d}t{\rm d}x
\\&=(2\lambda)^{s-1}+\frac{2\lambda}{\Gamma(1-s)}\int_0^\infty\frac{ e^{-2\lambda t}}{t^s} \int_0^{t}1 {\rm d}x{\rm d}t=(2\lambda)^{s-1}+\frac{2\lambda}{\Gamma(1-s)}\int_0^\infty{t^{1-s}}{ e^{-2\lambda t}}{\rm d}t\\
&\overset{\tau=2\lambda t}{=}(2\lambda)^{s-1}+\frac{(2\lambda)^{s-1}}{\Gamma(1-s)}\int_0^\infty{\tau^{1-s}}{ e^{-\tau}}{\rm d}\tau
=(2-s)(2\lambda)^{s-1}.
\end{aligned}
\end{equation}
Owe to
$$\D{}{}{}\big\{(2\lambda x)^2L^{(2)}_{n_2}(2\lambda x)\big\}=(2\lambda)^2(n_2+2) xL^{(1)}_{n_2}(2\lambda x),$$
i.e.,
$$\int_0^{t} xL^{(1)}_{n_2}(2\lambda x){\rm d}x=\frac{1}{n_2+2} t^2L^{(2)}_{n_2}(2\lambda t),$$
we obtain that
  \begin{equation}
\begin{aligned}
&\big(\TD{-\infty}{x}{s}\phi^{\ast},\phi^+_{n_2}\big)=\int_0^\infty\frac{e^{-\lambda x}}{\Gamma(1-s)}\int_{-\infty}^0\frac{2\lambda e^{2\lambda \tau}}{(x-\tau)^s} {\rm d}\tau~ xL^{(1)}_{n_2}(2\lambda x)e^{-\lambda x} {\rm d}x\\
&\overset{\tau=x-t}{=}\frac{2\lambda}{\Gamma(1-s)}\int_0^\infty\int_x^{\infty}\frac{ e^{-2\lambda t}}{t^s} {\rm d}t~ xL^{(1)}_{n_2}(2\lambda x){\rm d}x{=}\frac{2\lambda}{\Gamma(1-s)}\int_0^\infty\frac{ e^{-2\lambda t}}{t^s}\int_0^{t}xL^{(1)}_{n_2}(2\lambda x) {\rm d}x~ {\rm d}t\\
&\quad{=}\frac{2\lambda}{(n_2+2)\Gamma(1-s)}\int_0^\infty t^{2-s}L^{(2)}_{n_2}(2\lambda t)e^{-2\lambda t} {\rm d}t.
\end{aligned}
\end{equation}
Similarly, we have
\begin{equation}\begin{aligned}
\big(\TD{-\infty}{x}{s}\phi^{-}_{n_1},\phi^{+}_{n_2}\big)&{=}-\frac{{n_1}+1}{\Gamma(1-s)}\int_0^\infty \int_x^{\infty}\frac{L^{(0)}_{{n_1}+1}{(2\lambda (t-x))}e^{-2\lambda t}}{t^s} {\rm d}t ~xL^{(1)}_{n_2}(2\lambda x) {\rm d}x\\
&{=}-\frac{{n_1}+1}{\Gamma(1-s)}\int_0^\infty {t^{-s}}e^{-2\lambda t} \int_0^{t}~xL^{(1)}_{n_2}(2\lambda x)L^{(0)}_{{n_1}+1}{(2\lambda (t-x))} {\rm d}x  {\rm d}t\\
\overset{x=t\xi}{=}&-\frac{{n_1}+1}{\Gamma(1-s)}\int_0^\infty {t^{2-s}}e^{-2\lambda t} \int_0^{1}~\xi L^{(1)}_{n_2}(2\lambda t\xi)L^{(0)}_{{n_1}+1}{(2\lambda t(1-\xi))} {\rm d}\xi  {\rm d}t\\
\end{aligned}\end{equation}
The above equations are enough to calculate out the matrix $\mathbf{A}$ due to some symmetric properties of the entries.
\end{appendix}

\end{document}